\documentclass[11pt,reqno,tbtags]{amsart}
\usepackage{amssymb}
\usepackage{url}
\usepackage{natbib}
\bibpunct[, ]{[}{]}{;}{n}{,}{,}

\numberwithin{equation}{section}

\allowdisplaybreaks

\newtheorem{theorem}{Theorem}[section]
\newtheorem{lemma}[theorem]{Lemma}

\newtheorem{corollary}[theorem]{Corollary}

\theoremstyle{definition}
\newtheorem{problem}[theorem]{Problem}

\newtheorem*{ack}{Acknowledgement}

\newtheorem{exampleqqq}[theorem]{Example}
\newenvironment{example}{\begin{exampleqqq}}
  {\hfill\qedsymbol\end{exampleqqq}}

\newtheorem{remarkqqq}[theorem]{Remark}
\newenvironment{remark}{\begin{remarkqqq}}
  {\hfill\qedsymbol\end{remarkqqq}}

\theoremstyle{remark}

\newenvironment{romenumerate}[1][-10pt]{
\addtolength{\leftmargini}{#1}\begin{enumerate}
 \renewcommand{\labelenumi}{\textup{(\roman{enumi})}}%
 \renewcommand{\theenumi}{\textup{(\roman{enumi})}}%
 }{\end{enumerate}}

\newcounter{oldenumi}
\newenvironment{romenumerateq}
{\setcounter{oldenumi}{\value{enumi}}
\begin{romenumerate} \setcounter{enumi}{\value{oldenumi}}}
{\end{romenumerate}}

\newenvironment{alphenumerate}[1][-20pt]{
\addtolength{\leftmargini}{#1}\begin{enumerate}
 \renewcommand{\labelenumi}{\textup{(\alph{enumi})}}%
 \renewcommand{\theenumi}{\textup{\labelenumi}}%
 }{\end{enumerate}}

\newcounter{thmenumerate}

\newcommand\pfitemx[1]{\par#1:}
\newcommand\pfitemref[1]{\pfitemx{\ref{#1}}}

\newcounter{stepp}
\newcommand\stepp[1]{\refstepcounter{stepp}\smallskip\noindent
  \emph{Step \arabic{stepp}: #1}\noindent}
\newcommand\pfcase[1]{\refstepcounter{stepp}\smallskip\noindent
  \emph{Case \arabic{stepp}: #1}\noindent}

\newcommand{\refT}[1]{Theorem~\ref{#1}}
\newcommand{\refTs}[1]{Theorems~\ref{#1}}
\newcommand{\refC}[1]{Corollary~\ref{#1}}
\newcommand{\refL}[1]{Lemma~\ref{#1}}
\newcommand{\refLs}[1]{Lemmas~\ref{#1}}
\newcommand{\refR}[1]{Remark~\ref{#1}}
\newcommand{\refS}[1]{Section~\ref{#1}}
\newcommand{\refSs}[1]{Sections~\ref{#1}}
\newcommand{\refSS}[1]{Section~\ref{#1}}

\newcommand{\refE}[1]{Example~\ref{#1}}

\newcommand{\refApp}[1]{Appendix~\ref{#1}}
\newcommand{\refTab}[1]{Table~\ref{#1}}

\newcommand{\refCase}[1]{Case~\ref{#1}}

\begingroup
  \divide\count255 by 60
  \multiply\count255 by -60
  \advance\count255 by \time
  \ifnum \count255 < 10 \xdef\klockan{\the\count1.0\the\count255}
  \else\xdef\klockan{\the\count1.\the\count255}\fi
\endgroup

\newcommand\nopf{\qed}   

\DeclareMathOperator*{\sumx}{\sum\nolimits^{*}}

\DeclareMathOperator*{\sumax}{\sumx_{\bga}}
\DeclareMathOperator*{\sumaxx}{\sum^*_{\bga}} 

\newcommand{\sumj}{\sum_{j=0}^\infty}
\newcommand{\sumk}{\sum_{k=0}^\infty}
\newcommand{\sumki}{\sum_{k=1}^\infty}

\newcommand{\summ}{\sum_{m=0}^\infty}

\newcommand{\sumni}{\sum_{n=1}^\infty}
\newcommand{\sumaa}{\sum_{\aaa\in\cAx}}
\newcommand{\sumaay}{\sum_{|\aaa|\ge1}}
\newcommand{\suma}{\sum_{\ga\in\cA}}

\newcommand{\sumkK}{\sum_{k=1}^K}
\newcommand{\sumklK}{\sum_{k,\ell=1}^K}

\newcommand{\sumin}{\sum_{i=1}^n}

\newcommand{\prodim}{\prod_{i=1}^m}

\newcommand{\sumjoooo}{\sum_{j=-\infty}^\infty}
\newcommand{\sumkoooo}{\sum_{k=-\infty}^\infty}
\newcommand{\summoooo}{\sum_{m=-\infty}^\infty}

\newcommand\set[1]{\ensuremath{\{#1\}}}
\newcommand\bigset[1]{\ensuremath{\bigl\{#1\bigr\}}}

\newcommand\xpar[1]{(#1)}
\newcommand\bigpar[1]{\bigl(#1\bigr)}
\newcommand\Bigpar[1]{\Bigl(#1\Bigr)}

\newcommand\lrpar[1]{\left(#1\right)}
\newcommand\bigsqpar[1]{\bigl[#1\bigr]}
\newcommand\Bigsqpar[1]{\Bigl[#1\Bigr]}

\newcommand\abs[1]{|#1|}
\newcommand\bigabs[1]{\bigl|#1\bigr|}
\newcommand\Bigabs[1]{\Bigl|#1\Bigr|}

\def\rompar(#1){\textup(#1\textup)}    
\newcommand\xfrac[2]{#1/#2}

\newcommand\xqfrac[2]{#1/(#2)}
\newcommand\parfrac[2]{\lrpar{\frac{#1}{#2}}}

\newcommand\Bigparfrac[2]{\Bigpar{\frac{#1}{#2}}}

\def\xexp(#1){e^{#1}}

\newcommand\floor[1]{\lfloor#1\rfloor}

\newcommand\ntoo{\ensuremath{{n\to\infty}}}

\newcommand\ktoo{\ensuremath{{k\to\infty}}}
\newcommand\jtoo{\ensuremath{{j\to\infty}}}
\newcommand\gltoo{\ensuremath{{\gl\to\infty}}}
\newcommand\rtoo{\ensuremath{{r\to\infty}}}

\newcommand\xtoo{\ensuremath{{x\to\infty}}}

\newcommand\norm[1]{\|#1\|}
\newcommand\bignorm[1]{\bigl\|#1\bigr\|}
\newcommand\Bignorm[1]{\Bigl\|#1\Bigr\|}

\newcommand\punkt[1]{\if.#1\else.\spacefactor1000\fi{#1}}
\newcommand\iid{i.i.d\punkt}    
\newcommand\ie{i.e\punkt}
\newcommand\eg{e.g\punkt}

\newcommand\cf{cf\punkt}
\newcommand{\as}{a.s\punkt}
\newcommand{\aex}{a.e\punkt}

\newcommand\ii{\mathrm{i}}

\newcommand{\tend}{\longrightarrow}
\newcommand\dto{\overset{\mathrm{d}}{\tend}}
\newcommand\pto{\overset{\mathrm{p}}{\tend}}

\newcommand\eqd{\overset{\mathrm{d}}{=}}

\newcommand\op{o_{\mathrm p}}
\newcommand\Op{O_{\mathrm p}}

\newcommand\bbR{\mathbb R}
\newcommand\bbC{\mathbb C}
\newcommand\bbN{\mathbb N}  

\newcommand\bbZ{\mathbb Z}
\newcommand\bbQ{\mathbb Q}

\newcounter{CC}
\newcommand{\CC}{\stepcounter{CC}\CCx} 
\newcommand{\CCx}{C_{\arabic{CC}}}     
\newcommand{\CCdef}[1]{\xdef#1{\CCx}}     
\newcommand{\CCname}[1]{\CC\CCdef{#1}}    
\newcommand{\CCreset}{\setcounter{CC}0} 
\newcounter{cc}

\renewcommand\Re{\operatorname{Re}}
\renewcommand\Im{\operatorname{Im}}

\newcommand\E{\operatorname{\mathbb E{}}}
\renewcommand\P{\operatorname{\mathbb P{}}}
\newcommand\Var{\operatorname{Var}}
\newcommand\Cov{\operatorname{Cov}}

\newcommand\Po{\operatorname{Po}}

\newcommand\rank{\operatorname{rank}}
\newcommand\sign{\operatorname{sign}}

\newcommand\ga{\alpha}
\newcommand\gb{\beta}
\newcommand\gd{\delta}

\newcommand\gam{\gamma}
\newcommand\gG{\Gamma}

\newcommand\gl{\lambda}
\newcommand\gL{\Lambda}

\newcommand\gs{\sigma}
\newcommand\gS{\Sigma}
\newcommand\gss{\sigma^2}
\newcommand\gth{\theta}

\newcommand\eps{\varepsilon}

\newcommand\cA{\mathcal A}

\newcommand\cE{\mathcal E}

\newcommand\cS{{\mathcal S}}
\newcommand\cT{{\mathcal T}}

\newcommand\ett[1]{\boldsymbol1\set{#1}}
\newcommand\indic[1]{\boldsymbol1\set{#1}} 

\def\[#1]{[\![#1]\!]}

\newcommand\smatrixx[1]{\left(\begin{smallmatrix}#1\end{smallmatrix}\right)}

\newcommand\qq{^{1/2}}

\newcommand\qqw{^{-1/2}}

\newcommand\qw{^{-1}}
\newcommand\qww{^{-2}}

\renewcommand{\=}{:=}

\newcommand\intoo{\int_0^\infty}

\newcommand\oi{[0,1]}

\newcommand\setoi{\set{0,1}}

\newcommand\dd{\,\ddx}
\newcommand\ddx{\mathrm{d}}

\newcommand{\ui}{uniformly integrable}

\newcommand\lhs{left-hand side}
\newcommand\rhs{right-hand side}

\newcommand\sss[1]{^{(#1)}}
\newcommand\sssk{\sss k}

\newcommand\hpsi{\widehat\psi}

\newcommand\kk{\kappa}
\newcommand\tgs{\widetilde\sigma}
\newcommand\tgss{\tgs^2}

\newcommand\Tq[1]{\cT_{#1}}
\newcommand\Tn{\Tq{n}}
\newcommand\Tb{\Tq{b}}
\newcommand\ellb{b}
\newcommand\xT{\widetilde\cT}
\newcommand\Tqq[1]{\xT_{#1}}
\newcommand\Tgl{\Tqq{\gl}}
\newcommand\Tgn{\Tqq{n}}
\newcommand{\Too}{T_\infty}
\newcommand{\Ti}{\bullet}
\newcommand\cAx{\cA^*}
\newcommand\aaa{{\boldsymbol{\alpha}}}
\newcommand\bbb{\boldsymbol{\beta}}
\newcommand\bga{{\boldsymbol{\alpha}}}
\newcommand\bgb{{\boldsymbol{\beta}}}
\newcommand\bgg{{\boldsymbol{\gamma}}}
\newcommand\ntt{n_{T'}(T)}
\newcommand\Tx{T^*}
\newcommand\Tglga{\Tgl^{\ga}}
\newcommand\Tglgax{\Tgl^{\ga+}}
\newcommand\Tglbga{\Tgl^{\bga}}
\newcommand\Tglbgb{\Tgl^{\bgb}}
\newcommand\Tglbgax{\Tgl^{\bga+}}
 \newcommand\Tglxxx[1]{\Tgl^{#1+}}
 \newcommand\Xglga{X_{\gl,\ga}}
 \newcommand\Xglbga{X_{\gl,\bga}}
 \newcommand\Xbga{X_{\bga}}
 \newcommand\Xbgb{X_{\bgb}}

  \newcommand\Nglx[1]{N_{\gl,#1}}
  \newcommand\Nglga{\Nglx{\ga}}
   \newcommand\Nglbga{\Nglx{\bga}}
\newcommand\approxd{\overset{\mathrm{d}}{\approx}}
\newcommand\xgb{\hat\gb}
\newcommand\hgs{\widehat\gs}
\newcommand\hgss{\hgs^2}
\newcommand\hgS{\widehat\gS}
\newcommand\gln{n}
\newcommand\pmin{p_{\mathrm{min}}}
\newcommand\pmax{p_{\mathrm{max}}}
\newcommand\Aary{$|\cA|$-ary}
\newcommand\bfx{\mathbf{x}}
\newcommand\gf{\varphi}
\newcommand\GF{\Phi}
\newcommand\glbga{\gl P(\bga)}
\newcommand\sumbga{\sum_{\bga\in\cAx}}

\newcommand\vp{\mathbf{p}}
\newcommand\dA{d_{\vp}}
\newcommand\sxx[1]{\mathsf{#1}}
\newcommand\sE{\sxx{E}}
\newcommand\sV{\sxx{V}}
\newcommand\sC{\sxx{C}}
\newcommand\sX{\sxx{X}}
\newcommand\fE{f_{\sE}}
\newcommand\fV{f_{\sV}}
\newcommand\fC{f_{\sC}}
\newcommand\fX{f_{\sX}}
\newcommand\sjw{*}
\newcommand\Mx[1]{#1^*}
\newcommand\Mf{\Mx{f}}
\newcommand\MfE{\Mx{\fE}}
\newcommand\MfV{\Mx{\fV}}
\newcommand\MfC{\Mx{\fC}}
\newcommand\MfX{\Mx{\fX}}

\newcommand\psiE{\psi_{\sE}}
\newcommand\psiV{\psi_{\sV}}
\newcommand\psiC{\psi_{\sC}}
\newcommand\psiX{\psi_{\sX}}
\newcommand\fEq[1]{f_{\sE,#1}}
\newcommand\fVq[1]{f_{\sV,#1}}
\newcommand\fCq[1]{f_{\sC,#1}}
\newcommand\fXq[1]{f_{\sX,#1}}
\newcommand\fEp{\fEq{\sjw}}
\newcommand\fVp{\fVq{\sjw}}
\newcommand\fCp{\fCq{\sjw}}
\newcommand\fXp{\fXq{\sjw}}
\newcommand\MfEq[1]{\Mx{\fEq{#1}}}
\newcommand\MfVq[1]{\Mx{\fVq{#1}}}
\newcommand\MfCq[1]{\Mx{\fCq{#1}}}

\newcommand\psiEq[1]{\psi_{\sE,#1}}
\newcommand\psiVq[1]{\psi_{\sV,#1}}
\newcommand\psiCq[1]{\psi_{\sC,#1}}
\newcommand\psiXq[1]{\psi_{\sX,#1}}
\newcommand\psiEz{\psiEq{+}}
\newcommand\psiEx{\psiEq{\sjw}}
\newcommand\psiVx{\psiVq{\sjw}}
\newcommand\psiCx{\psiCq{\sjw}}
\newcommand\psiXx{\psiXq{\sjw}}
\newcommand\xoo{_1^\infty}
\newcommand\gfb{\gf_{b*}}
\newcommand\GFb{\GF_{b*}}
\newcommand\sta{\overline{\mathfrak{T}}}
\newcommand\stt{\mathfrak{T}}
\newcommand\stap{\sta_+}
\newcommand\sttp{\stt_+}
\newcommand\Pbga{P(\bga)}
\newcommand\ch{{\mathsf {ch}}}
\newcommand\emptystring{\boldsymbol{\epsilon}}
\newcommand\tauk{k}
\newcommand\gfo{\gf_\Ti}
\newcommand\GFo{\GF_\Ti}
\newcommand\gff{\gf_\sjw}
\newcommand\GFF{\GF_\sjw}
\newcommand\nx[1]{\nu_{#1}}
\newcommand\nbga{\nx{\bga}}
\newcommand\xprot[1]{\textup{$#1$-prot}}
\newcommand\kprot{\xprot{k}}

\newcommand\gfkprot{\gf_\kprot}
\newcommand\GFkprot{\GF_\kprot}
\newcommand\rot{\emptystring}
\newcommand\dds{\frac{\ddx}{\dd s}}
\newcommand\multim{\mathbf{m}}
\newcommand\absmoment{[absolute] moment}
\newcommand\absmoments{[absolute] moments}
\newcommand\gQ[1]{\rho(#1)}
\newcommand\gLa{\gL_<}
\newcommand\gLb{\gL_>}
\newcommand\ha{h_<}
\newcommand\hb{h_>}
\newcommand\hh{\eps}
\newcommand\si{\mathsf i}
\newcommand\se{\mathsf e}
\newcommand\absi[1]{|#1|_{\si}}
\newcommand\abse[1]{|#1|_{\se}}
\newcommand\rf{^*}
\newcommand\aTn{a_{T,n}}

\renewcommand\ln{\log}
\newcommand\bb{^{(b)}}
\newcommand\Tch{{T_\textrm{ch}}}
\newcommand\qxx{{**}}
\newcommand\gfxx{\gf_{\qxx}}
\newcommand\GFxx{\GF_{\qxx}}

\newcommand{\Holder}{H\"older}
\newcommand\CS{Cauchy--Schwarz}
\newcommand\CSineq{\CS{} inequality}

\newcommand\REM[1]{{\raggedright\texttt{[#1]}\par\marginal{XXX}}}

\newenvironment{comment}{\setbox0=\vbox\bgroup}{\egroup} 

\newcommand\urladdrx[1]{{\urladdr{\def~{{\tiny$\sim$}}#1}}}

\begin{document}
\title[Central limit theorems 
  in tries]
{Central limit theorems for additive functionals and fringe trees in tries}

\date{5 March 2020}

\author{Svante Janson}
\address{Department of Mathematics, Uppsala University, PO Box 480,
SE-751~06 Uppsala, Sweden}
\thanks{Partly supported by the Knut and Alice Wallenberg Foundation}
\email{svante.janson@math.uu.se}
\urladdrx{http://www.math.uu.se/~svante/}


\begin{comment}  

60 Probability theory and stochastic processes [For additional applications, see 11Kxx, 62-XX, 90-XX, 91-XX, 92-XX, 93-XX, 94-XX]
60C Combinatorial probability
60C05 Combinatorial probability

60K Special processes
60K05 Renewal theory
60K10 Applications (reliability, demand theory, etc.)
60K15 Markov renewal processes, semi-Markov processes
60K20 Applications of Markov renewal processes (reliability, queueing
      networks, etc. ) 
60K35 Interacting random processes; statistical mechanics type models;
      percolation theory [See also 82B43, 82C43] 

68 Computer science
68P Theory of data
68P05 Data structures
68P10 Searching and sorting

\end{comment}

\begin{abstract} 
We give general theorems on asymptotic normality for additive functionals of
random tries generated by a sequence of independent strings. 
These theorems are applied to show asymptotic normality of the
distribution of random fringe trees in a random trie. 
Formulas for asymptotic mean and variance are given. In particular,
the proportion of fringe trees of size $k$ (defined as number of keys)
is asymptotically, ignoring
oscillations, $c/(k(k-1))$ for $k\ge2$, where $c=1/(1+H)$ with $H$ the
entropy of the letters. Another application gives asymptotic normality of the
number of $k$-protected nodes in a random trie. 
For symmetric tries, it is shown that the asymptotic proportion of
$k$-protected nodes (ignoring oscillations) decreases geometrically as
$k\to\infty$. 
\end{abstract}

\maketitle

\section{Introduction}\label{S:intro}

We consider random tries constructed from a number of random (infinite) strings
with letters in a fixed finite alphabet $\cA$.
(The most important case is $\cA=\setoi$, and the reader may
for simplicity assume this without essential loss.)
See \refS{Sprel} for the definition of tries 
and other definitions of terms used here in the introduction.

We assume throughout the paper that the strings are \iid, and moreover, that
the individual letters in the strings are \iid.
The number of strings will be either fixed, or a Poisson variable; we refer
to these as the \emph{fixed $n$ model} (where $n$ is the number of strings)
and the \emph{Poisson model}.

As has been well-known since at least \cite{JR88,JR89}, 
for 
some sets of letter probabilities (in particular, for the symmetric case
with equal probabilities),
there are 
typically
(numerically small) oscillations in the asymptotics of both mean and
variance for functionals of random tries; nevertheless asymptotic normality
holds with suitable normalizations. The cases where oscillations occur are
well understood, either from the location of poles of Mellin transforms,
see \eg{} \cite{FlajoletRV,JS-Analytic}, or from (the arithmetic case of)
renewal theory, see \cite{MRobert:prob, SJ242}.

One of our main results is a central limit theorem
(i.e., asymptotic normality) of this type, including possible oscillations,
for additive functionals
of tries under rather weak conditions,
for both the fixed $n$ model and the Poisson model (\refT{TT}).
This theorem assumes that the toll function is bounded (together with
another technical condition).
We give, as a corollary, a law of large numbers (\refT{TLLN}).

In \refS{Sfringe}, several applications of these theorems are given.
In particular, we study random fringe trees of tries, and show  
central limit theorems for the distribution of them.
We study also the number of $k$-protected nodes in tries, $k\ge2$,
and prove a central limit theorem. We show also that
for symmetric tries, ignoring oscillations, the expected number of
$k$-protected nodes decreases geometrically as \ktoo.
We give also a couple of other applications.

Our method of proof consists of the following three separate parts:

{\addtolength{\leftmargini}{-20pt}
\begin{enumerate} 
\renewcommand{\labelenumi}{\theenumi}%
 \renewcommand{\theenumi}{(\arabic{enumi})}%
\item \label{p1}
To prove asymptotic normality for the Poisson model, we use 
the independence of different branches in the trie and the classical central
limit theorem for sums of independent random variables. The proof requires
several estimates, including a moment estimate that is proved by induction
using a less common version of Rosenthal's inequality (\refL{LPinelis}).

\item \label{p2}
To depoissonize, \ie, transfer results to the fixed $n$ model, we use here a
novel approach, using a conditional limit theorem by \citet{Nerman}.
The main condition for this theorem is that the functionals we consider are
increasing, or at least the difference of two increasing functionals.

\item \label{p3}
To find asymptotic means and variances, we use results from \cite{SJ242}
based on renewal theory, see also \cite{MRobert:prob}. 
\end{enumerate}}

Note that there are several earlier papers on asymptotic normality for
tries, where all three steps have been
proved by detailed analyses of generating functions. That is a wonderful
method, but the method used here avoids the necessity to estimate the generating
functions in the complex plane; this may be useful or convenient in some
applications. Furthermore, our method is easily adapted to more general
sources of random strings, see \refR{Rother}. 
The reader is encouraged to compare, and perhaps combine, the methods for
future work.

We state the results of steps \ref{p1} and \ref{p2} above as
general central limit theorems, in several versions
(\refTs{T1}--\ref{Tknabc}, with proofs in \refS{Spfcentral}),
where the toll function may be unbounded
but we assume some technical conditions on moments of the additive
functional and its toll functional.
Then, as step \ref{p3}, we prove separately (in \refS{Smean})
\refT{TEVC} on mean and variance of additive functionals.
This is based on a theorem from
\cite{SJ242}, which for convenience is stated, and somewhat extended, in
\refApp{A242}. 
Finally, \refT{TT} follows by combining \refT{TEVC} and the general central
limit theorems. (This proof is in \refS{SpfTT}.)

One reason for this organization is that the central limit theorems and the
moment asymptotics are proved by quite different methods, and we find it
instructive 
to present them separately,
and not only their combination \refT{TT}.
This also enables us to present somewhat more general results, as said above.

  \begin{remark}
    \label{Rother}
  The method of proof of normality (steps \ref{p1} and \ref{p2} above)
applies,
under suitable conditions, also to random strings where the letters
are not independent, for example strings from a Markov source, or the bit
expansions of random numbers with a non-uniform distribution on $(0,1)$.
(We still assume that different strings are \iid.)
This will be studied elsewhere.
  \end{remark}

\section{Preliminaries}\label{Sprel}

\subsection{Some general notation}
We use $\pto$ and $\dto$ to denote convergence in probability and
distribution,
respectively, of random variables. 
$\eqd$ denotes equality in distribution.

$(X\mid\cE)$ denotes the random variable $X$ conditioned on the event $\cE$.

For a random variable $X$ and $r>0$,
$\norm{X}_r:=(\E |X|^r)^{1/r}$, the $L^r$ norm.

$C$ denotes various unimportant constants, possibly different at different
occurences. We sometimes for clarity write
$C_1,C_2,\dots$, and we use $C_\gl$ for a ``constant'' that depends on $\gl$.

We use standard $o$ and $O$ notation, for sequences and functions of a real
variable; note that $O$ is used both in a global and an asymptotic sense:
for example,
$f(x)=O(g(x))$ for $x\in S$ means that $|f(x)|\le Cg(x)$ for all $x\in S$
(equivalently, if $g(x)>0$ in $S$, $f(x)/g(x)$ is bounded in $S$),
while
$f(x)=O(g(x))$ as \xtoo{} means that $|f(x)|\le Cg(x)$ for large $x$.
For positive functions or sequences we also use the notations
$\Omega$ and $\Theta$:
$f(x)=\Omega(g(x))$ as \xtoo{} means that $f(x)\ge c g(x)$ for some $c>0$ and
large $x$, or, equivalently, $g(x)=O(f(x))$ as \xtoo;
$f(x)=\Theta(g(x))$ means
$f(x)=O(g(x))$ and $f(x)=\Omega(g(x))$, and similarly for sequences.

For $x\in\bbR^d$, $|x|$ denotes the usual Euclidean norm. (Any other norm
would do as well.)

For $x\in\bbR$, $\floor x$ is the largest integer $\le x$.


$\ln$ denotes the natural logarithm.

\subsection{Strings}\label{SSstrings}

We consider strings with letters in a finite alphabet $\cA$.
($\cA$ is fixed throughout the paper.)
Let $\cAx:=\bigcup_{n=0}^\infty\cA^n$, the set of finite strings  from
$\cA$.
The empty string is denoted by $\emptystring$.

We write $\aaa\preceq\bbb$ if $\aaa$ and $\bbb$ are two strings and $\aaa$ is
a prefix of $\bbb$.

The tries will be constructed from $n$
random infinite strings
$\Xi\sss1,\Xi\sss2,\dots,\allowbreak\Xi\sss{n}$,
where
$\Xi\sssk=\xi_1\sssk\xi_2\sssk\dotsm$
with letters $\xi_i\sssk\in\cA$.
(We may drop the superscript  and write
$\Xi=\xi_1\xi_2\dotsm$ for a generic string in the sequence.)
We suppose that the strings $\Xi\sssk$ are independent,
and furthermore  that the individual letters $\xi_i\sssk$ are \iid.
We thus assume throughout the paper that we are given a probability
distribution $\vp=(p_\ga)_{\ga\in\cA}$, and that
\begin{align}\label{pa}
  \P(\xi_i\sssk=\ga)=p_\ga,
  \qquad \ga\in\cA.
\end{align}
To avoid trivialities, we assume that each $p_\ga>0$ (otherwise we may reduce
$\cA$), and that $|\cA|>1$, and thus each $p_\ga<1$.
We let $\pmin:=\min_\ga p_\ga$ and $\pmax:=\max_\ga p_\ga$, and note that
$0<\pmin\le\pmax<1$.

The \emph{entropy} $H$ is defined by
\begin{align}\label{entropy}
  H:= -\sum_{\ga\in\cA} p_\ga\ln p_\ga >0.
\end{align}

Given a finite string $\ga_1\dotsm\ga_m\in\cAx$, let
$P(\ga_1\dotsm\ga_m)$ be the probability that the random string $\Xi$
has prefix $\ga_1\dotsm\ga_m$, \ie, that $\xi_i=\ga_i$ for $i\le m$.
In particular,
for a single letter, $P(\ga)=p_\ga$, and in general
\begin{equation}\label{paaa}
  P(\ga_1\dotsm\ga_m)
  =\prodim p_{\ga_i}.
\end{equation}

For later use we define, for complex $s\in\bbC$,
\begin{align}\label{Qks}
  \gQ s:=\sum_{\ga\in\cA}p_\ga^s
\end{align}
and note that \eqref{paaa} implies that, for any $m\ge0$,
\begin{align}\label{Qkm}
 \sum_{|\bga|=m}P(\bga)^s  
=\sum_{\ga_1,\dots,\ga_m\in\cAx}\prodim p_{\ga_i}^s
= \gQ s^m.
\end{align}
For any real $r>1$, we have
\begin{align}\label{Qk}
  \gQ r=\sum_{\ga\in\cA}p_\ga^r
< \sum_{\ga\in\cA}p_\ga
=1,
\end{align}
and thus by \eqref{Qkm}
\begin{align}\label{Qrr}
 \sum_{\bga\in\cAx}P(\bga)^r  
=\summ \sum_{|\bga|=m}P(\bga)^r  
=\summ \gQ r^m<\infty.
\end{align}
Furthermore, we note that
\begin{align}
  \label{QH}
\dds \gQ s\big|_{s=1}=\sum_{\ga\in\cA} p_\ga\ln p_\ga = -H.
\end{align}

\subsection{Trees}\label{SStrees}
A \emph{leaf} in a rooted
tree is a node without children; leaves are also called
\emph{external nodes}, while the remaining nodes are called \emph{internal
  nodes}. 

Let $\Too$ be the infinite \Aary{} tree where the nodes are 
the finite strings $\bga\in\cAx$; the root
is the empty string $\emptystring$,
and the children of a node $\bga$ are the nodes $\bga \gam$ with $\gam\in\cA$.
Hence $\bga$ is a (strict) ancestor of $\bgb$ if and only if $\bga\prec\bgb$
(i.e., $\bga$ is a strict prefix of $\bgb$).

A finite \Aary{} tree is  a finite subtree of $\Too$ containing its
root $\emptystring$;
for convenience we regard also the empty tree $\emptyset$ with no
nodes 
as a finite \Aary{} tree.
Let $\sta$ be the countable set of all finite $|\cA|$-ary trees,
and let $\stap:=\sta\setminus\set{\emptyset}$, the subset of nonempty trees.

We may identify trees in $\sta$ with their sets of nodes,
and we write $\abs{T}$ for the number of nodes in $T$;
we denote the numbers of internal and external nodes (= leaves)
by $\absi{T}$ and
$\abse T$, respectively; thus $\abs{T}=\absi{T}+\abse T$.
Let $\sta_n:=\set{T\in\sta:\abse{T}=n}$, the set of finite \Aary{} trees with
exactly $n$ leaves.
(Note that we in the present paper thus count the size by the number of
leaves; this is natural in the context of tries.)

Let $\Ti\in\sta$ denote the tree consisting of only the root
$\emptystring$.
Thus $\abs{\Ti}=\abse{\Ti}=1$ and $\Ti\in\sta_1$.

\subsection{Tries}\label{SStries}

A \emph{trie} (for a given alphabet $\cA$)
is an \Aary{} tree that is constructed in the following way
from a set of $n\ge0$ distinct
strings in $\cA^\infty$,
see \eg{} \cite[Section 6.3]{KnuthIII}
and \cite[Section 7.1]{Drmota}.
If $n=0$,  the trie is defined to be
the empty tree $\emptyset$.
Otherwise,
we begin with a root, and put every string in the root. 
If $n=1$, then we stop there, so the trie has just one node.
Otherwise, \ie, if $n\ge2$, we pass all strings
to new nodes; for each letter $\ga\in\cA$, we pass all strings beginning
with $\ga$, if any, to a new node labelled $\ga$.
We continue recursively, the next time partitioning the strings according to
the second letter, and so on, always looking at the first letter not yet
inspected;
hence, the strings passed to a node $\bga\in\cAx$, if any, are the strings with
prefix $\bga$, and if there are at least two such strings, then they are all
passed further to children of $\bga$.
At the end there is a tree with $n$ leaves, each containing one string. 

Given a set of infinite strings,
let  $\nbga$ be the
number of these strings that have $\bga$ as a prefix, for $\bga\in\cAx$,
and note that the trie $T$ just constructed can be defined as
the subtree of $\Too$ consisting of all nodes $\bga$ such that
one of the following holds:
\begin{itemize}
\item $\nbga\ge2$ (then $\bga$ is an internal node in $T$),
\item  $\nbga=1$ and either $\bga=\emptystring$ or
the parent of $\bga$ is an internal node
(then $\bga$ is an external node in $T$).
\end{itemize}

We are mainly interested in random tries, see below, but we say also that a
deterministic \Aary{} tree is a trie if it can be generated in this way from
some set of strings.
(It is easily seen that a finite \Aary{} tree is a trie if and only if
there is no leaf with a parent that has only one child.)
Denote the set of all tries by $\stt\subset\sta$.
Let $\stt_n:=\sta_n\cap\stt$, the set of tries with $n$ leaves,
and $\sttp:=\bigcup_1^\infty\stt_n=\stt\setminus\set\emptyset$.

Note that adding a new string to the ones generating a trie $T$ means either
adding a new leaf to an internal node of $T$, or converting a leaf to a
path of $k\ge1$ additional internal nodes, and
adding two new leaves to the last node in this path.
We call this \emph{adding a new string to $T$}.

A \emph{functional} of tries
is a function $\gf:\stt\to\bbR$
such that (to avoid uninteresting complications) $\gf(\emptyset)=0$.

We say that a functional $\GF$
of tries
is \emph{increasing} if $\GF(T_1)\le \GF(T_2)$
whenever $T_1$ is a subtree of $T_2$.
It is easily seen that it suffices to consider the case
when $T_2$ is obtained from
$T_1$ by adding a new string.

\subsection{Random tries}\label{SSRtries}

Let $\Tn$ denote the random trie generated by the $n$ \iid{} random
infinite strings
$\Xi\sss1,\Xi\sss2,\dots,\allowbreak\Xi\sss{n}$ (see \refSS{SSstrings}).
Note that $\Tn$ has $n$ leaves, so $\Tn\in\stt_n$.

In the trivial case $n=1$, we see that $\Tq1=\Ti$ is non-random.
(This is the only trie in $\stt_1$.)

We consider also the Poisson version.
In general, for any random variable
$N\in\bbN_0$, independent of the strings $\Xi\sssk$, $k\ge1$,
we may consider the
tree $T_N$ constructed from the $N$ strings $\Xi\sss1,\dots,\Xi\sss{N}$.
We will only consider the case $N=N_\gl\sim\Po(\gl)$ for some $\gl>0$, and
we then use the notation
\begin{equation}
  \label{Tgl}
\Tgl:=T_{N_\gl}.
\end{equation}
In the Poisson case,  we use the notation $\Nglbga$ 
for   the (random) number $\nbga$ of strings with prefix $\bga$,
i.e.,
\begin{equation}
\Nglbga
:=\bigabs{\set{k\le N_\gl:\Xi\sssk\succeq\bga}}.
\end{equation}
By standard properties of the Poisson distribution,
for any $\bga\in\cAx$,
  \begin{equation}\label{pax}
    \Nglbga\sim\Po\bigpar{\gl P(\bga)}
 . \end{equation}
Furthermore, for any finite strings $\bga_1,\dots,\bga_\ell$ such that none of
them is a prefix of another, the random variables
$\Nglx{\bga_1},\dots,\Nglx{\bga_\ell}$
are independent.

\subsection{Bucket tries}\label{SSbucket}
A \emph{bucket trie} (or \emph{$b$-trie}) 
is a generalization of tries; it is constructed from a number
of strings recursively in the same way as a trie, see \refSS{SStries},
but stopping when the number of strings in a node is at most some given
number $b$, known as the \emph{bucket size}.
Thus ordinary tries is the case $b=1$. In general, a leaf (external node)
will contain from 1 to $b$ strings. (The leafs are also called \emph{buckets}.)
In the notation above for random tries, the internal
nodes are  $\set{\bga\in\cAx:\nbga\ge b+1}$.

Note that, for any given bucket size $b\ge2$,
we can construct the trie $T$ based on a set of strings by
first constructing the bucket trie $T'$ with bucket size $b$,
and then letting a small trie grow from each bucket.
Moreover, for \iid{} random strings $\Xi\sss1,\dots,\Xi\sss n$ as above,
conditioned on the bucket trie,
these small tries are independent, and the small trie grown from a bucket
that contains $k$ strings is a copy of $\Tq k$.

We use bucket tries as a tool in some proofs.

\subsection{Fringe trees}\label{SSfringe}

Given a rooted tree $T$ and a node $v$ in $T$, let $T^v$ be the subtree of
$T$  consisting of $v$ and all its
descendents (with $v$ as the root of $T^v$).
Such subtrees are called \emph{fringe subtrees}, 
or just \emph{fringe  trees}, of $T$.
For convenience, we also define $T^v:=\emptyset$, the empty tree,
if $v\notin T$.
We consider in the present paper only trees $T\in\sta$, \ie, finite \Aary{}
trees; we then also regard the fringe trees $T^v$ as elements of $\sta$ in
the obvious way.
(Recall that we have defined trees in $\sta$ as subtrees 
of $\Too$ with root $\emptystring$, the empty string.) Thus, formally,
\begin{align}\label{fringe}
  T^\bga=\set{\bgb\in\cAx:\bga\bgb\in T}
\end{align}
Note that the fringe trees of a trie are tries.
Furthermore,
for a trie $T$ generated as in \refSS{SStries} from a set of strings, and any
$\ga\in T$,
\begin{align}
  \label{bamse}
\abse{T^\bga}=\nbga,
\end{align}
the number of generating strings with prefix $\bga$.

The  \emph{random fringe subtree}  $\Tx$
is the random rooted tree obtained by taking the subtree $T^v$ 
at a uniformly random node $v$ in $T$;
see \cite{Aldous-fringe}. (We assume $T\neq\emptyset$.)
 Let, for \Aary{} trees $T,T'\in\sta$, 
\begin{equation}\label{ntt}
  n_{T'}(T):=\bigabs{\set{v\in T':T^v=T'}},
\end{equation}
\ie, the number 
of subtrees of $T$ that are equal 
to $T'$. Then the distribution of $\Tx$ is
given by
\begin{equation}\label{ptx}
  \P(\Tx=T')=n_{T'}(T)/\abs{T}, \qquad T'\in\sta.
\end{equation}

When $T$ is a random tree, as in \cite{Aldous-fringe}
as well as in the present paper where we consider $\Tn$ and $\Tgl$, 
$\ntt$ is a random variable for each $T'\in\sta$,
and \eqref{ptx} holds for the conditional probability
$ \P\bigpar{\Tx=T'\mid T}$.


\subsection{Additive functionals}

Let $\gf$ be a functional of tries,
and 
consider the functional $\GF$ defined
for a trie $T\in\stt$ 
by the sum
\begin{equation}\label{GF}
  \GF(T)=\GF(T;\gf)\=\sum_{v\in T} \gf(T^v).
\end{equation}
(Thus, $\GF(\emptyset)=0$.)
Recall that we assume $\gf(\emptyset)=0$.
Hence, \eqref{GF} can be written as the formally infinite sum
\begin{align}\label{GFbga}
  \GF(T)=\sum_{\bga\in\cAx} \gf(T^\bga).
\end{align}
Moreover,
  the definition \eqref{GF} can also be written recusively as
  \begin{equation}\label{toll}
	\GF(T)=\gf(T)+\sum_{\ga\in\cA}\GF(T^\ga),
  \end{equation}
where $T^\ga$, $\ga\in\cA$, are the principal branches 
of $T$,
\ie, the fringe subtrees rooted at the children of the root.

A functional $\GF$ that can be written as \eqref{GF}--\eqref{toll} is often
called 
an \emph{additive functional} with
\emph{toll function} $\gf$.
(Any functional can be written in this form for
some $\gf$, so the important part of this terminology is the relation
between $\GF$ and $\gf$.)

\begin{example}
  \label{Eleaves}
  A simple example, which will be important in the sequel, 
  is the toll function
  \begin{align}\label{gf0}
    \gfo(T):=\indic{\abs{T}=1}=\indic{T=\Ti};
  \end{align}
  then \eqref{GF} shows that the corresponding additive functional $\GFo$
  counts the number of leaves in $T$.
  In particular, a random trie $\Tn$ has always
  $n$ leaves, and thus $\GFo(\Tn)=n$ is non-random.
  Similarly, by \eqref{Tgl},
  \begin{align}\label{gF00}
  \GFo(\Tgl)=N_\gl
\sim \Po(\gl).    
  \end{align}
\end{example}

\begin{example}
A more general  example  is to take $\gf(T)=\ett{T=T'}$, the indicator
that $T$ equals
some given tree $T'\in\sta$; then $\GF(T)=n_{T'}(T)$ defined in \eqref{ntt}.
Conversely, for any $\gf$, \eqref{GF} can be written
\begin{equation}\label{GFnt}
  \GF(T)=\sum_{T'\in\sttp}\gf(T') n_{T'}(T);
\end{equation}
hence any additive functional can be written as a
(potentially infinite) linear combination of the 
subtree counts $\ntt$,
where it suffices to consider (nonempty) tries $T'$.
\end{example}

\subsection{Fringe trees of tries}\label{SSfringe2}
For the random trie $\Tgl$ and any
string $\bga\in\cAx$,
we have by
the recursive construction of tries
that the fringe tree
$\Tgl^\bga$ is a trie constructed from $\Nglbga$ strings, except in the
case $\Nglbga=1$, when it is also possible that $\Tgl^\bga=\emptyset$ because
$\bga\notin\Tgl$ 
(when  $\bga$ has a parent that is not an internal node).
We therefore define
\begin{align}\label{byx}
  \Tglbgax:=
  \begin{cases}
    \Tgl^\bga, & \Nglbga\neq1,\\
    \Ti, & \Nglbga=1.
  \end{cases}
\end{align}
Then,
$\Tglbgax$ is always a trie constructed from $\Nglbga$ strings, and thus, 
by \eqref{pax}, for any (fixed) $\bga\in\cAx$,
\begin{align}
  \label{tpax}
  \Tglbgax\eqd \Tqq{\gl P(\bga)}.
\end{align}
Furthermore, $\Tglbga$ and $\Tglbgax$ differ by \eqref{byx} only in the case  
$\Tglbga=\emptyset$ and $\Tglbgax=\Ti$; hence, for any functional $\gf$ on
$\sta$,  
\begin{align}
  \label{tpax1}
  \gf(\Tglbga)=\gf(\Tglbgax)+O(1).
\end{align}
Moreover, if $\gf:\stt\to\bbR$ is a functional such that
$
\gf(\Ti)=0$, then
\begin{align}
  \label{tpaxx}
  \gf(\Tglbga)=
\gf(\Tglbgax)\eqd \gf\bigpar{\Tqq{\gl P(\bga)}}.
\end{align}

For any finite strings $\bga_1,\dots,\bga_\ell$ such that none of
them is a prefix of another, the random tries
$\Tglxxx{\bga_1},\dots,\Tglxxx{\bga_\ell}$
are independent, since this holds for
$\Nglx{\bga_1},\dots,\Nglx{\bga_\ell}$ as pointed out above.
Note that this does not hold for the fringe tries
$\Tgl^{\bga_1},\dots,\Tgl^{\bga_\ell}$ in general, again because of the
special case $\Nglx{\bga}=1$.

For these reasons, we will often
as a technical tool
use $\Tglbgax$
instead of $\Tglbga$.

\begin{remark}\label{Ralt}
  For any additive functional $\GF$ with toll function $\gf$, 
$\GF(\Ti)=\gf(\Ti)$ by \eqref{GF}, and thus it follows from \eqref{byx}
that
  $\GF(\Tglbga)-\gf(\Tglbga)=\GF(\Tglbgax)-\gf(\Tglbgax)$.
Hence, for any $\bga_1,\dots,\bga_\ell$ such that none is a prefix of
another, by the comments just made, 
the random variables  $\GF(\Tglbga)-\gf(\Tglbga)$
are independent.
This could be used in the proofs below as an alternative to using the
modified fringe tree $\Tglbgax$; it seems that the choice is mainly a matter
of taste, but we invite the reader to explore this further.
\end{remark}

%

\subsection{Greatest common divisor}\label{SSgcd}

Given a set $S$ of real numbers, we define $\gcd(S)$ to be the largest
positive real number $d$ such that $S\subseteq d\bbZ$ (equivalently:
$x/d\in\bbZ$ for every $x\in S$), provided that some such $d>0$ exists; if
no such $d$ exists, we define $\gcd(S):=0$.
(We assume that $S$ contains some non-zero element; otherwise this
definition would give $\infty$.)
We will only use this in the case $S:=\set{-\ln p_\ga:\ga\in\cA}$, and we
then use the special notation $\dA:=d(S)$ for this $S$.
We say that $\vp$ is \emph{periodic} if $\dA>0$. (This is when periodic
oscillations typically occur in the results below.)

In particular, if $x,y\neq0$,
then
$
  \gcd(x,y)=0
  \iff x/y\notin\bbQ
$.
Hence, if $\cA=\setoi$, then 
\begin{align}\label{gcd0}
  \dA=0
  \iff \frac{\ln p_1}{\ln p_0}\notin\bbQ.
\end{align}

\subsection{Mellin transform}\label{SSmellin}
If $f$ is a (measurable) function on $(0,\infty)$, its \emph{Mellin  transform}
is defined by
\begin{align}\label{mellin}
  \Mf(s):=\intoo f(s) x^{s-1}\dd x,
\end{align}
for all complex $s$ such that the integral converges absolutely.
(This domain is always a
vertical strip in the complex plane, which may be infinite, finite, or empty.
For simplicity we consider only absolute convergence which suffices for us;
for other purposes one might also consider conditionally convergent
integrals \eqref{mellin}.)
See further \eg{} \cite[Appendix B.7]{FlajoletS}.

\subsection{Convergence and approximation in distribution}
\label{SSapprox}
As said above, we use  $\dto$ to denote convergence in
distribution of random variables; these may take values in some metric space
$\cS$, see \eg{} \cite{Billingsley}.
(We will only use $\cS=\bbR^d$ for some $d$.)
Recall that by definition
\cite{Billingsley},
$X_n\dto Y$ if and only if 
$\E f(X_n)\to\E f(Y)$ as \ntoo{}
for every bounded continuous function
$f:\cS\to\bbR$.  
We extend this notion as follows.
 
Let $(X_n)\xoo$ and $(Y_n)\xoo$
be two sequences of random variables with values in
a metric space $\cS$.
We write $  X_n\approxd Y_n$ if, for every bounded continuous function
$f:\cS\to\bbR$, 
\begin{align}\label{approx}
 \E f(X_n)=\E f(Y_n)+o(1)
\qquad \text{as \ntoo}.
\end{align}

If $\cS=\bbR$,
we say that   \emph{$X_n\approxd Y_n$ with moments of order $s$}
(where $s\in\bbN$)
if \eqref{approx} holds and also
\begin{align}\label{momapprox}
 \E X_n^s=\E Y_n^s+o(1)
\end{align}
with both sides finite.
More generally,
if $\cS=\bbR^d$,
we say that   \emph{$X_n\approxd Y_n$ with moments of order $s$}
if \eqref{approx} holds and also,
for every multi-index $\multim$ with $|\multim|=s$,
\begin{align}\label{momapproxd}
 \E X_n^\multim=\E Y_n^\multim+o(1)
\end{align}
with both sides finite.
Similarly, still for $\cS=\bbR^d$, 
we say that   \emph{$X_n\approxd Y_n$ with absolute moment of order $s$}
(where  $s\in\bbR_+$)
if \eqref{approx} holds and also
\begin{align}\label{absmomapprox}
 \E |X_n|^s=\E |Y_n|^s+o(1)
\end{align}
with both sides finite.

For applications, ordinary moments are usually more interesting, but we use
absolute moments in at least one proof; 
we therefore give statements including both.
For brevity we will write 
``with  \absmoments{} of order $s$'', meaning
with absolute moments of order $s$ 
and, provided $s$ is an integer,
also with moments of order $s$.
(For the relation between these, see \refApp{Aapprox}.)

We use the same notation for variables $X_\gl$ and $Y_\gl$ depending on a
continuous parameter.

\begin{remark}
  \label{Rapprox}
If $Y_n=Y$ for all $n$, then $X_n\approxd Y_n$ is equivalent to $X_n\dto Y$,
by the definitions above.
More generally, the same holds if we assume $Y_n\dto Y$.
\end{remark}

\begin{remark}\label{Rsubsub}
The standard \emph{subsequence principle}
says that a sequence in a metric space converges to a limit $x$ if and only
if every subsequence has a subsubsequence that converges to $x$.
It is well known that this holds also for convergence in distribution, in
any metric space.  (Cf.\ \cite[Section 5.7]{Gut}).
It   holds also for $\approxd$ (and any metric space $\cS$):
  If every subsequence $(n_k)$ has a subsubsequence along which
  $X_n\approxd Y_n$, then $X_n\approxd Y_n$ along the full sequence.
  (This follows by fixing $f:\cS\to\bbR$:
 each subsequence then has a subsubsequence such that
 \eqref{approx} holds, and thus \eqref{approx} holds for the full sequence.)
 The same holds with a continuous parameter.
\end{remark}

 We use the subsequence principle several times in our proofs, often
 omitting some details. Here follows one example, extending 
 to $\approxd$
the standard
result that if $X_n\dto Y$, then 
uniform integrability of $|X_n|^s$ implies
convergence of \absmoment{s} of order $s$,
see \eg{} \cite[Theorem 5.5.9]{Gut}.

 \begin{lemma}
   \label{Lapprox0}
 Let $(X_n)\xoo$  and $(Y_n)\xoo$  be random vectors in $\bbR^d$
such that $X_n\approxd Y_n$.
Let further
  $s>0$, 
and suppose that
the sequence $(|X_n|^s)$ and  $\xpar{|Y_n|^s}$ are uniformly integrable.
Then, 
$X_n\approxd Y_n$ with  \absmoments{} of order $s$.
 \end{lemma}

We give a detailed proof in \refApp{Aapprox}, together with a converse and
some further comments.

\section{A central limit theorem}\label{Smain}
We begin with our main results.
Proofs are given in \refSs{Smean} and~\ref{SpfTT}.

The first theorem is a preliminary result
giving asymptotics for mean and variance of additive
functionals
in the Poisson model under rather weak conditions (implying a linear growth),
including the case of bounded toll functions; it also introduces some
notation that will be used  in the sequel.
Corresponding results for the fixed $n$ model (under stronger conditions)
are included in \refT{TT}.

Recall the definition of
the entropy $H$ in \eqref{entropy},
the greatest common divisor $\dA:=\gcd\set{-\ln p_\ga:\ga\in\cA}$ in
\refSS{SSgcd},
and the Mellin transform $\Mf$ in \eqref{mellin}. 

\begin{theorem}\label{TEVC}
 Let $\gf$ be a toll function and
 let $\GF$ be the corresponding additive functional given by \eqref{GF}.
Suppose that, 
  for some $\eps>0$,
as \gltoo,
\begin{align}
  \E \gf(\Tgl) &= O\bigpar{\gl^{1-\eps}},\label{bze}
                 \\
    \Var \gf(\Tgl) &= O\bigpar{\gl^{1-\eps}}.\label{bzv}
  \end{align}
Let
  \begin{align}
\chi&:=\gf(\Ti), \label{chi}
\\
\fE(\gl)&:=\E\gf(\Tgl)-\chi \gl e^{-\gl},\label{fEx}
    \\
    \fV(\gl)&:=2\Cov\bigpar{\gf(\Tgl),\GF(\Tgl)}-\Var\gf(\Tgl)
\notag\\&\qquad
+2\chi \gl e^{-\gl}\bigpar{\E\GF(\Tgl)-\E\gf(\Tgl)}
-\chi^2 \gl e^{-\gl}\bigpar{1-\gl e^{-\gl}}
,\label{fVx}
              \\
    \fC(\gl)&:=\Cov\bigpar{\gf(\Tgl),N_\gl}
+\chi\gl(\gl-1) e^{-\gl}
.\label{fCx}
  \end{align}
Then the following hold.  
  \begin{romenumerate}
\item\label{TEVC0}
  If\/ $\dA=0$, then, as $\gl\to\infty$,
  \begin{align}
\frac{\E\GF(\Tgl)}{\gl}
    &\to \chi+
\frac1 H \Mx{\fE}(-1)
=
\chi+
\frac1 H \intoo \fE(x)x\qww\dd x\label{uE0},	
    \\
\frac{\Var\GF(\Tgl)}{\gl}
    &\to 
\chi^2+
\frac1 H \Mx{\fV}(-1)
=
\chi^2+
\frac1 H \intoo \fV(x)x\qww\dd x, \label{uV0}	
    \\
\frac{\Cov\bigpar{\GF(\Tgl),N_\gl}}{\gl}
    &\to 
\chi+
\frac1 H \Mx{\fC}(-1)
=
\chi+
\frac1 H \intoo \fC(x)x\qww\dd x. \label{uC0}
  \end{align}
\item\label{TEVCd} 
More generally, for any $\dA$,
as $\gl\to\infty$,
  \begin{align}
\frac{\E\GF(\Tgl)}{\gl}
&= \chi+\frac1 H \psiE(\ln\gl) + o(1),\label{uE+}
    \\
\frac{\Var\GF(\Tgl)}{\gl}
&= \chi^2+\frac1 H \psiV(\ln\gl) + o(1),\label{uV+}
    \\
\frac{\Cov\bigpar{\GF(\Tgl),N_\gl}}{\gl}
&= 
\chi+
\frac1 H \psiC(\ln\gl) + o(1),\label{uC+}
  \end{align}
where $\psiX$, 
for $\sX=\sE,\sV,\sC$,
are bounded continuous functions defined as follows:
\begin{enumerate}
\item 
If $\dA=0$ then $\psiX$ is constant: for all $t$,
\begin{align}\label{psiX0}
\psiX(t):=\Mx{\fX}(-1).  
\end{align}
\item 
If $d=\dA>0$, then $\psiX$ is a continuous
$d$-periodic function having the Fourier series
\begin{equation}\label{tevcfou}
  \psiX(t)\sim\summoooo  \Mx{\fX}\Bigpar{-1-\frac{2\pi m}{d}\ii}e^{2\pi\ii  mt/d}.
\end{equation}
Furthermore, 
\begin{equation}
  \label{tevcsum}
\psiX(t)=d\sumkoooo e^{kd-t}\fX\bigpar{e^{t-kd}}.
\end{equation}
Moreover, if
$\sX=\sE$, or if
$\fX'(\gl)=O(\gl^{-\eps_1})$ as \gltoo{} for some $\eps_1>0$, 
then 
the Fourier series \eqref{tevcfou} converges absolutely, and thus $\sim$ may
be replaced by $=$ in \eqref{tevcfou}.
\end{enumerate}

\item\label{TEVC+}
 If\/ $\gf(T)\ge0$ for every trie $T$
and $\gf(T')>0$ for some trie $T'$,
then
$\inf_t\bigpar{H\qw\psiE(t)+\chi}>0$, and thus
  $\E\GF(\Tgl)=\Theta(\gl)$ as \gltoo.
\end{romenumerate}
\end{theorem}

\begin{remark}\label{REVC}
  When $\dA>0$, the constant term in \eqref{tevcfou} is $\MfX(-1)$.
Thus we may regard the \rhs{s} of \eqref{uE0}--\eqref{uC0} as
``average asymptotic values'' of the \lhs{s} also when $\dA>0$, remembering
that then
the asymptotics really also include oscillations around these values.
As is well known, the oscillation are numerically small in typical examples.
\end{remark}

\begin{remark}\label{REVC0}
It can be seen above, and in more detail later in the proof, that
fringe subtrees $\Ti$ (leaves) play a special role; see also
\refSS{SSfringe2}.
The formulas in \refT{TEVC} simplify somewhat
in the case $\gf(\Ti)=0$, where such fringe subtrees are ignored. 
(This case is very common in applications, 
see \refS{Sfringe} for examples.)
In particular, if $\gf(\Ti)=0$, then \eqref{fEx}--\eqref{fCx} simplify to
  \begin{align}
    \fE(\gl)&:=\E\gf(\Tgl),\label{fE0}
    \\
    \fV(\gl)&:=2\Cov\bigpar{\gf(\Tgl),\GF(\Tgl)}-\Var\gf(\Tgl),\label{fV0}
              \\
    \fC(\gl)&:=
\Cov\bigpar{\gf(\Tgl),N_\gl}
\label{fC0}
.  \end{align}
\end{remark}

\begin{remark}\label{Rexists}
  It follows from the proof that
  $\fE(\gl),\fV(\gl),\fC(\gl)$ are finite for every $\gl>0$,
and extend to entire functions,
  and that
  the Mellin transforms
  $\MfE(s),\allowbreak\MfV(s),\MfC(s)$ exist at least in the strip 
$-2<\Re s<-1+\eps/2$, so the values in \eqref{uE0}--\eqref{uC0} and
  \eqref{tevcfou} are well defined.
In fact, at least
  $\MfE(s)$ exists in the strip $-2<\Re s<-1+\eps$, and \eqref{tex*} below
  shows that $\MfC$ extends analytically to the same strip, but we do not
  know whether \eqref{mellin} always converges absolutely there for $\fC$. (The
  integral converges at least conditionally there by the proof of \refL{LEVCx}.)
\end{remark}

\begin{remark}\label{Rcov}
The results \eqref{uV0} and \eqref{uV+} in  \refT{TEVC} extend
immediately to the covariance $\Cov\bigpar{\GF_1(\Tgl),\GF_2(\Tgl)} $ for
two additive functionals with toll functions $\gf_1,\gf_2$ satisfying
\eqref{bze}--\eqref{bzv}; the function $\fV$ in \eqref{fVx} is replaced by,
with $\chi_j:=\gf_j(\Ti)$,
\begin{align}
\fVq{12}(\gl):=
  &\Cov\bigpar{\gf_1(\Tgl),\GF_2(\Tgl)}
  +\Cov\bigpar{\gf_2(\Tgl),\GF_1(\Tgl)}-\Cov\bigpar{\gf_1(\Tgl),\gf_2(\Tgl)}
\notag\\&\quad
+\chi_1\bigpar{\E\GF_2(\Tgl)-\gf_2(\Tgl))}\gl e^{-\gl} 
+ \chi_2\bigpar{\E\GF_1(\Tgl)-\gf_1(\Tgl)}\gl e^{-\gl} 
\notag\\&\quad
-\chi_1\chi_2\bigpar{1-\gl e^{-\gl}}\gl e^{-\gl} 
 .\label{fCov}
\end{align}
This follows by polarization, \ie, by considering $\gf_1\pm\gf_2$.

Note that taking $\gf_2=\gfo$ yields $\Cov\bigpar{\GF_1(\Tgl),N_\gl}$, so we
can regard \eqref{uC0} and \eqref{uC+} as special cases of the bilinear
versions of \eqref{uV0} and \eqref{uV+}.
Indeed, it is easily verified that if $\gf_1=\gf$ and $\gf_2=\gfo$, then
\eqref{fCov} reduces to \eqref{fCx}.
\end{remark}

We use in the sequel frequently 
the number $\chi$, 
the functions $\fE,\fV,\fC$,
their Mellin transforms $\MfE,\MfV,\MfC$, and
the periodic functions $\psiE,\psiV,\psiC$ 
defined  in \refT{TEVC};
these have always the meanings above,
for some given $\gf$. (We  say this explicitly sometimes, for emphasis, but
not always.)
We note a relation between $\fE$ and $\fC$.

\begin{lemma}\label{LEVCx}
  Let $\gf$ be as in \refT{TEVC}. Then,
  for all $\gl$ and $t$, and at least for
  $\Re s\in(-2,-1+\eps/2)$,
  \begin{align}
    \fC(\gl)&=\gl\fE'(\gl),
                                \label{tex}\\
    \MfC(s)&=-s\MfE(s),
               \label{tex*}\\
    \psiC(t)&=\psiE(t)+\psiE'(t),
         \label{texpsi}
  \end{align}
In particular, 
\begin{align}
  \label{texas}
\MfC(-1)=\MfE(-1).
\end{align}
\end{lemma}

\begin{remark}
  The argument in the proof of \eqref{tex} shows also that
  \begin{align}\label{kub}
    \gl\frac{\ddx}{\dd\gl}\GF(\Tgl)=\Cov\bigpar{\GF(\Tgl,N_\gl}.
  \end{align}
This derivative appears in the formula for the asymptotic variance of
$\GF(\Tn)$ already in \cite{JR88};
we regard \eqref{kub} as an explanation of this appearance.

Note also that \eqref{kub} and \eqref{texpsi} imply that
\eqref{uC+} can be regarded as a formal derivative of \eqref{uE+}.
\end{remark}

The next theorem might be regarded as our main result.
It gives asymptotic normality of additive functionals of tries
for both the Poisson and the fixed $n$ model.
The theorem is easy to apply but
still quite general; we will use it to show the results on fringe trees in
\refS{Sfringe}.
We have chosen to state this theorem here, because of its central role in
the paper. However, as said above, we also later
give some more general (and somewhat more technical)
central limit theorems in \refS{Sgeneral};
the proof of \refT{TT}  combines
some of these results from \refSs{Sgeneral} and \refT{TEVC}.
For simplicity, and convenience in many applications, we 
consider in the remainder of this section only toll
function that are bounded. 

\begin{remark}\label{RTT}
The proof of \refT{TT} shows that the assumption on boundedness 
can be relaxed to the moment conditions \eqref{bze}, \eqref{bzv} and
\eqref{Efr} (for any $r>2$) for $\gf$ and $\gf_\pm$.
The same applies to \refT{TLLN}.
\end{remark}

\begin{theorem}\label{TT}
  Let $\gf$ be a bounded toll function 
and
let $\GF$ be the corresponding additive functional given by \eqref{GF}.
Suppose further that $\gf=\gf_{+}-\gf_{-}$ for some bounded toll
functions $\gf_{\pm}$
such that the corresponding functionals $\GF_{\pm}$ are increasing.
Then, with notation as in \refT{T1}, \eqref{chi}--\eqref{fCx} and
\eqref{psiX0}--\eqref{tevcsum}:
\begin{romenumerate}
\item\label{TT0}
  If\/ $\dA=0$, then, as $\gl\to\infty$ and \ntoo,
\begin{align}\label{tt0gl}
{\frac{\GF(\Tgl)-\E \GF(\Tgl)}{\sqrt{\gl}}}
  &\dto N\bigpar{0,\gss},
  \\
{\frac{\GF(\Tn)-\E \GF(\Tn)}{\sqrt{n}}}
  &\dto N\bigpar{0,\hgss},\label{tt0n}
\end{align}
with all \absmoment{s},
where
\begin{align}
  \gss&=\chi^2+H^{-1}\Mx{\fV}(-1),\label{tt0gs}
  \\
    \hgss&=H^{-1}\Mx{\fV}(-1)-H^{-2}\Mx{\fC}(-1)^2
-2\chi H\qw\MfC(-1).\label{tt0hgs}
\end{align}

\item\label{TT+}
For any $\dA\ge0$, 
as $\gl\to\infty$ and \ntoo,
\begin{align}\label{tt+gl}
{\frac{\GF(\Tgl)-\E \GF(\Tgl)}{\sqrt{\gl}}}
  &\approxd N\bigpar{0,\gss(\gl)},
  \\
{\frac{\GF(\Tn)-\E \GF(\Tn)}{\sqrt{n}}}
  &\approxd N\bigpar{0,\hgss(n)},\label{tt+n}
\end{align}
with all \absmoment{s}, where
\begin{align}
  \gss(\gl)&=\chi^2+H^{-1}\psiV(\ln \gl),\label{tt+gs}
  \\
    \hgss(n)&=H^{-1}\psiV(\ln n)-H^{-2}\psiC(\ln n)^2
-2\chi H\qw\psiC(\ln n),\label{tt+hgs}
\end{align}
with 
continuous $d$-periodic functions
$\psiV,\psiC$.  

\item \label{TTE}
 We have
 \begin{align}
\label{tte}
\E\GF(\Tn)=\E\GF(\Tqq{n})+o\bigpar{\sqrt n}
 \end{align}
and may thus replace $\E\GF(\Tn)$ by $\E\GF(\Tqq{n})$ in \eqref{tt0n}
and \eqref{tt+n}.

\item\label{TTconv}
If\/  $\liminf_\ntoo\Var\GF(\Tn)/n>0$, then, for any $\dA\ge0$,
\begin{align}\label{ttgl}
{\frac{\GF(\Tgl)-\E \GF(\Tgl)}{\sqrt{\Var\GF(\Tgl)}}}
  &\dto N\xpar{0,1},
  \\
{\frac{\GF(\Tn)-\E \GF(\Tn)}{\sqrt{\Var\GF(\Tn)}}}
  &\dto N\xpar{0,1},\label{ttn}
\end{align}
 with convergence of  all \absmoment{s}.
\item \label{TTmean}
The means $\E\GF(\Tgl)$ and $\E\GF(\Tn)$ satisfy
\begin{align}
  \label{ttmean}
\E\GF(\Tgl)&=\gl\bigpar{\chi+ H\qw\psiE(\ln\gl)}+o(\gl),
\\   \label{ttmeann}
\E\GF(\Tn)&=n\bigpar{\chi+ H\qw\psiE(\ln n)}+o(n).
\end{align}
\end{romenumerate}
\end{theorem}

\begin{remark}\label{Rmulti}
  \refT{TT}\ref{TT0}--\ref{TTE} 
extend in an obvious way to multivariate limits for several
  functionals $\GF_k$; this follows by the Cram\'er--Wold device
  (or by modifying the proof).
\end{remark}

\begin{remark}\label{Rlower}
  We can in \eqref{tt0gl}, \eqref{tt+gl} and \eqref{ttgl}
\emph{not} replace $\E\GF(\Tgl)$ by its asymptotic value
$\gl\bigpar{\chi+ H\qw\psiE(\ln\gl)}$ in
\eqref{ttmean}.
The reason is that that 
when $\dA=0$, the $o(1)$ error term in \eqref{uE+} typically 
is larger than $\gl\qqw$; in fact, this error term
is in general not $O(\gl^{-\eps})$ for any $\eps>0$. 
When $\dA>0$, the error is $O(\gl^{-\eps})$ for some $\eps>0$
depending on the probabilities $\vp$,
but this $\eps$ may be arbitrarily small; in particular,
also in the case $\dA>0$, 
the error is in general not $o(\gl\qqw)$.
Thus the error term in \eqref{ttmean} is in general not $o(\gl\qq)$.
These error estimates is implicit in
\citet{FlajoletRV}; see \refApp{Alower} for details.

The same holds for $\Tn$ and \eqref{tt0n}, \eqref{tt+n}, \eqref{ttn},
as a consequence of these results for $\Tgl$ and \eqref{tte}.
\end{remark}

As a corollary we obtain a weak law of large numbers.
This is much weaker than the central limit theorem in \refT{TT}, and
presumably holds under weaker conditions (with a more direct proof),
but we do not pursue this here.

\begin{theorem}\label{TLLN}
  Let $\gf$ be a toll function satisfying the conditions of \refT{TT}.
 Let  $\MfE(s)$, $\psiE(t)$  and $\chi$ be as in \refT{TEVC}.
\begin{romenumerate}
\item \label{TLLNa}
Then, as \gltoo{} and \ntoo,
\begin{align}\label{tllna1}
  \frac{\GF(\Tgl)}{\gl}
  -
  H\qw\psiE(\ln\gl) 
-\chi&\pto0,
  \\\label{tllna2}
  \frac{\GF(\Tn)}{n}
  -
  H\qw\psiE(\ln n) 
-\chi&\pto0.
\end{align}

In particular, if $\dA=0$, then, as \ntoo,
\begin{align}\label{tllna3}
  \frac{\GF(\Tn)}{n} 
\pto H\qw\Mx{\fE}(-1)+\chi
  = H\qw \intoo {\E[\gf(\Tgl)]}\gl^{-2}\dd \gl+\chi.
\end{align}
\item \label{TLLNb}
If furthermore $\gf\ge0$ and $\P(\gf(\Tn)>0)>0$ for some $n\ge1$, then
$\inf_t\bigpar{H\qw\psiE(t)+\chi}>0$, and thus, for some $c>0$,
 as \ntoo,
\begin{align}\label{stalis}
\P\bigpar{\GF(\Tn)\ge cn}\to1. 
\end{align}
\end{romenumerate}
\end{theorem}

\begin{problem}
  Do the limits \eqref{tllna1}--\eqref{tllna3} hold a.s.?
\end{problem}

We give one case where the condition in \refT{TT}\ref{TTconv} holds;
it holds in many other cases too, but see \refE{E0} for a counterexample.

\begin{lemma}\label{LC}
  Let $\GF$ be an additive functional with bounded toll function $\gf$ and
  suppose that there exists $n_0\in\bbN$ such that $\gf(\Tn)=a_n$ (a.s.)\ for
  $n\ge n_0$ and some constants $a_n$.
  Suppose also that $\Var\GF(\Tn)\neq0$ for some $n\ge1$.
  Then
  $\Var\GF(\Tn)=\Omega(n)$ as \ntoo.
\end{lemma}

Our formulas for
variance asymptotics and asymptotic variances, \eqref{uV0}, \eqref{uV+}
and \eqref{tt+gs}--\eqref{tt+hgs}, use 
$\MfV$ and $\psiV$ which are defined using
$\fV(\gl)$.
The definition \eqref{fVx} of $\fV(\gl)$ is 
less useful for explicit calculations.
We therefore give also an alternative formula, which will be used in the
applications in \refS{Sfringe}.
For simplicity, we consider only the case $\chi=0$.

We use for convenience the special notation
\begin{align}\label{sumax}
  \sumax:=\sum_{\bga:|\bga|\ge0}+\sum_{\bga:|\bga|>0}
\end{align}
where thus every $\bga\in\cAx$
except  $\bga=\emptystring$ is counted twice.

\begin{lemma}\label{LfV}
  Let $\gf$ be a bounded toll function with $\gf(\Ti)=0$,
and let $\GF$ be
  the corresponding additive functional.
  Then,
  for $\gl>0$ and
  (at least)
  $\Re s\in(-2,-\frac12)$,
  \begin{align}
    \fV(\gl)&
    =\sumax\Cov\bigpar{\gf(\Tgl),\gf(\Tglbga)},
\label{lfv1}\\
    \MfV(s)
&=\sumax\intoo\Cov\bigpar{\gf(\Tgl),\gf(\Tglbga)} \gl^{s-1}\dd \gl,
\label{lfv2}  \end{align}
with sums and integrals absolutely convergent.
\end{lemma}

We give also another useful formula for $\MfE$.
\begin{lemma}\label{LfE}
Let $\gf$ be a bounded toll function.
Then, at least for $-2<\Re s<0$,
\begin{align}\label{lfe1}
  \MfE(s)=\sum_{n=2}^\infty\frac{\gG(n+s)}{n!}\E\gf(\Tn).
\end{align}
In particular,
\begin{align}\label{lfe2}
  \MfE(-1)=\MfC(-1)=\sum_{n=2}^\infty\frac{\E\gf(\Tn)}{(n-1)n}.
\end{align}
\end{lemma}

\begin{example}\label{E0}
  The following example is in a sense negative, since it shows how trivial
  results can be derived by non-trivial calculations from the theorems above.
However, the example serves both as an illustration of the formulas above,
and as a counterexample and warning that there may be cancellations that
  are not obvious, leading to, for example, vanishing asymptotic variance
or absence of expected oscillations.

Consider the toll function
\begin{align}\label{e0gf}
  \gf(T):=\suma\indic{\abse{T^\ga}=1}.
\end{align}
Then, if $v$ is a leaf in $T$, then $\gf(T^v)=1$, while if $v$ is
an internal node, then $\gf(T^v)$ equals the number of children that are
leaves. Since every leaf is a child of some internal node, except in the
case $\abse{T}=1$, it follows from \eqref{GF} that if $\abse{T}>1$, then
$\GF(T)$ is twice the number of leaves. In general,
using the notation \eqref{gf0},
\begin{align}\label{e0GF}
  \GF(T)=2\abse{T}-\indic{\abse{T}=1}
=2\abse{T}-\gfo(T).
\end{align}
In particular, $\GF(\Tn)=2n$ for $n\ge2$, and $\Var\GF(\Tn)=0$.
Also, $\GF(\Tgl)=2N_\gl+O(1)$ and 
$\Var \GF(\Tgl)=4\gl+(4\gl^2-3\gl)e^{-\gl}-\gl^2 e^{-2\gl}\sim4\gl$.

The additive functional $\GF$ is increasing and the toll function $\gf$ is
bounded, so \refT{TT} applies.

We have $\chi=\gf(\Ti)=1$ and, by \eqref{fEx},
\begin{align}\label{monfe}
  \fE(\gl)=\E\gf(\Tgl)-\gl e^{-\gl}
=\suma p_\ga\gl e^{-p_\ga\gl}-\gl e^{-\gl}.
\end{align}
Thus, 
when $\Re s>-1$,
using \eqref{mellin} and \eqref{Qks},
\begin{align}\label{mona}
  \MfE(s)&
=\suma \intoo p_\ga\gl^s e^{-p_\ga\gl}\dd s
-\intoo \gl^s e^{-\gl} \dd\gl
\notag\\&
=\suma p_\ga^{-s}\gG(s+1)-\gG(s+1)
=
\bigpar{\rho(-s)-1}\gG(s+1).
\end{align}
By analytic continuation, \eqref{mona} holds for $\Re s>-2$, with a
removable singularity at $s=-1$.
Letting $s\to-1$ yields, using \eqref{QH},
\begin{align}\label{monb}
  \MfE(-1)=\lim_{s\to-1}\frac{\rho(-s)-1}{s+1}\gG(s+2)
=\dds \rho(-s)\big|_{s=-1}=H.
\end{align}
Note that in the periodic case $\dA>0$, the sum \eqref{tevcfou} is over 
roots $\zeta_m:=-1-2\pi m\ii/d$
of $\rho(-s)=1$, and \eqref{mona} shows that $\MfE(\zeta_m)=0$ for each
such root $\zeta_m\neq-1$. 
Hence, $\psiE(t)$ is constant also in the periodic case, and for any $\vp$,
\begin{align}
  \label{e0psiE}
\psiE(t)=\MfE(-1)=H.
\end{align}
In other words, the oscillations that usually occur
vanish in this example. (This is not so obvious from \eqref{tevcsum}.)
Hence, for any $\vp$, \refT{TT}\ref{TTmean} gives 
\begin{align}
  \E\GF(\Tgl)/\gl\to \chi+H\qw\MfE(-1)=2
\qquad \text{as \gltoo},
\end{align}
and similarly for $\E\GF(\Tn)$.
Of course, this is trivial from \eqref{e0GF}.

By \eqref{tex*} and \eqref{mona}, also $\MfC(\zeta_m)=0$ for 
the roots $\zeta_m\neq-1$ of $\rho(-s)=1$, and thus \eqref{tevcfou},
\eqref{texas} and \eqref{monb} yield that, for any $\vp$,
\begin{align}\label{e0psiC}
\psiC(t)=\MfC(-1)=H,  
\end{align}
so this too is constant even in the periodic case.

Similarly, \eqref{fVx} yields, after some calculations,
\begin{align}\label{monfv}
  \fV(\gl)&=\suma (3p_\ga\gl-4p_\ga^2\gl^2)e^{-p_\ga\gl}
+\suma p_\ga^2\gl^2e^{-2p_\ga\gl}
\notag\\&\qquad
-(3\gl-4\gl^2)e^{-\gl}-\gl^2e^{-2\gl},
\end{align}
and thus
\begin{align}\label{mons}
  \MfV(\gl)=\bigpar{\rho(-s)-1}\bigpar{3\gG(s+1)-4\gG(s+2)+2^{-s-2}\gG(s+2)}.
\end{align}
Thus also $\MfV(\zeta_m)=0$ for the roots $\zeta_m\neq-1$, 
and \eqref{mons} leads to, for any $\vp$,
\begin{align}\label{e0psiV}
  \psiV(t)=\MfV(-1)=3H.
\end{align}

Note that \eqref{tt+gs} and \eqref{tt+hgs} yield, using \eqref{e0psiC} and
\eqref{e0psiV},
$\gss(\gl)=4$ and $\hgss(n)=0$. 
Of course, \eqref{tt+gl} and \eqref{tt+n} with these variances are trivial
from \eqref{e0GF}.

This example has $\chi=1$, and we see how $\chi$ and the functions $\psiX$
interact in \eqref{tt+gs}--\eqref{tt+hgs} and \eqref{ttmean}--\eqref{ttmeann}. 
Consider now the modification
\begin{align}
  \gff(T):=\gf(T)-\gfo(T).
\end{align}
This equals the number of children of the root that are external nodes.
By \eqref{e0GF}, 
\begin{align}\label{e0GFx}
  \GFF(T)=\abse{T}-\indic{\abse{T}=1}
=\abse{T}-\gfo(T).
\end{align}
In particular, again $\GFF(\Tn)$ is deterministic.
Similar calculations, or simpler the general \eqref{f+} and
\eqref{psi+} in the proof of \refT{T1}, yield
$\fEp(\gl)=\fE(\gl)$, $\fCp(\gl)=\fC(\gl)$ given by \eqref{tex} and
\eqref{monfe}, and
\begin{align}
  \fVp(\gl)=\fV(\gl)-2\fC(\gl)
&=\suma (p_\ga\gl-2p_\ga^2\gl^2)e^{-p_\ga\gl}
+\suma p_\ga^2\gl^2e^{-2p_\ga\gl}
\notag\\&\qquad
-(\gl-2\gl^2)e^{-\gl}-\gl^2e^{-2\gl}.
\end{align}
Hence,
\begin{align}
  \psiEx(t)&=\psiE(t)=\psiCx(t)=\psiC(t)=H,\label{montec}
\\\label{montv}
\psiVx(t)&=\psiV(t)-2\psiC(t)=H,
\end{align}
and \eqref{tt+gs}--\eqref{tt+hgs} yield $\gss_\sjw(\gl)=1$ and $\gss_\sjw(n)=0$.
Again, \refT{TT}\ref{TT+} and \ref{TTmean} hold trivially.

Finally, consider the modification
\begin{align}\label{gfxx}
  \gfxx(T):=\gf(T)-2\gfo(T).
\end{align}
By \eqref{e0GF}, this toll function yields the additive functional
\begin{align}\label{e0GFxx}
  \GFxx(T)=-\indic{\abse{T}=1}
=-\gfo(T).
\end{align}
Hence $\GFxx(\Tn)=0$ for $n\ge2$, and $\GFxx(\Tgl)=-\indic{N_\gl=1}$
converges rapidly to 0. 
This additive functional is thus essentially 0, although the toll function
in \eqref{gfxx} looks non-trivial.
Both \eqref{ttgl} and \eqref{ttn} obviously fail.
The other parts of \refT{TT} apply also to this degenerate case.
We have $\chi_\qxx=-1$ and,
for example using \eqref{f+} and
\eqref{psi+} again,
$\fEq{\qxx}(\gl)=\fE(\gl)$, $\fCq{\qxx}(\gl)=\fC(\gl)$,
and
\begin{align}
  \fVq{\qxx}(\gl)=\fVp(\gl)-2\fC(\gl)
&=-\suma p_\ga\gl e^{-p_\ga\gl}
+\suma p_\ga^2\gl^2e^{-2p_\ga\gl}
\notag\\&\qquad
+\gl e^{-\gl}-\gl^2e^{-2\gl},
\end{align}
and thus
\begin{align}
  \psiEq{\qxx}(t)&=\psiE(t)=\psiCq{\qxx}(t)=\psiC(t)=H,
\\
\psiVq{\qxx}(t)&=\psiVx(t)-2\psiC(t)=-H.
\end{align}
Thus \eqref{tt+gs}--\eqref{tt+hgs} yield 
$\gss_\qxx(\gl)=0$ and $\hgss_\qxx(n)=0$.
Again \refT{TT}\ref{TT+} and \ref{TTmean} hold trivially.

\end{example}

\section{Central limit theorems for fringe tries}\label{Sfringe}

We give some applications of the general results above, including
applications to the distribution of random fringe trees.

We often state results only for the fixed $n$ model $\Tn$; 
similar results for the Poisson model $\Tgl$ follow similarly, 
but are only sometimes stated explicitly. 

We use the notation of \refS{Smain}, in particular
$\chi,\fX,\MfX,\psiX$
defined in \refT{TEVC}; recall also $\sumaxx$ defined in
\eqref{sumax}. We will
distinguish different additive functionals by subscripts,
and we sometimes use these subscripts in an obvious way also for 
$\chi$, $\fX$ and so on,
but we often omit subscripts when there is no risk of
confusion.

In all examples below, asymptotics for means and variances are given by 
\eqref{ttmean}--\eqref{ttmeann} and \eqref{tt+gs}--\eqref{tt+hgs}, using
$\psiE,\psiV,\psiC$ that are given by the Mellin transforms $\MfE,\MfV,\MfC$
and  \eqref{tevcfou} (absolutely convergent in all cases). 
We calculate these  Mellin transforms in several
cases, but usually omit stating explicitly 
the  formulas for asymptotic means and variances
that they lead to.

\subsection{The size}\label{SSsize}
As a warm-up, we
consider first the size of the trie, measured as $\GF_\si(T):=\absi{T}$, the
number of internal nodes.
This example has been studied by many authors. 
In particular, asymptotic normality was
shown already by \citet{JR88};
see also \cite[Section 5.4]{Mahmoud}.
Variance asymptotics is also studied there and in several other papers,
see the detailed analysis by \citet{FHZ} and the
many references given there.
We show here how these results follows by our methods.

The functional $\GF_\si(T)$ is an additive functional with toll function 
\begin{align}\label{xsgf}
\gf_\si(T)=\indic{\text{the root is an internal node}},   
\end{align}
and thus
\begin{align}\label{xsgfl}
  \gf_\si(\Tgl)=\indic{N_\gl\ge2}.
\end{align}
In this case, $\GF_\si$ is an increasing functional, so \refTs{TT} and
\ref{TLLN}
apply with
$\gf_+=\gf_\si$ and $\gf_-=0$.

 \refL{LC} shows that  $\Var \GF_\si(\Tn)=\Theta(n)$, and thus
\refT{TT}\ref{TTconv} applies; consequently, \refT{TT} shows immediately
that both $\GF_\si(\Tgl)$ and $\GF_\si(\Tn)$ are asymptotically normal;
more precisely, the following holds.
(For the means, recall also \refR{Rlower}.)
\begin{theorem}\label{Tsize}[\citet{JR88}]
Consider the  size $\GF_\si(T)=\absi{T}$.
Then, the central limit theorems
\eqref{tt+gl}--\eqref{tt+n} and \eqref{ttgl}--\eqref{ttn}
hold, with all \absmoment{s},
and with asymptotic variances given by 
\eqref{tt+gs}--\eqref{tt+hgs}
(and thus by \eqref{tt0gs}--\eqref{tt0hgs}  when $\dA=0$).
Furthermore, the means satisfy \eqref{ttmean} and \eqref{ttmeann},
and the laws of large numbers \eqref{tllna1}--\eqref{tllna2} hold.
\nopf
\end{theorem}

We have $\chi=0$ (so the formulas simplify a little)
and, by \eqref{xsgfl},
\begin{align}\label{xsfE}
  \fE(\gl)=\E\gf_\si(\Tgl)=\P\bigpar{N_\gl\ge2}
  =1-(1+\gl)e^{-\gl}
\end{align}
and thus
\begin{align}\label{xsMfE}
  \MfE(s)
  =\intoo \bigpar{1-(1+\gl)e^{-\gl}}\gl^{s-1}\dd\gl
  =-\frac{\gG(s+2)}{s},
  \qquad -2<\Re s<0;
\end{align}
where the integral can be evaluated \eg{} using integration by parts,
\cf{}  \cite[Proof of Theorem 5.3]{SJ242}.
In particular, or by \eqref{lfe2},
\begin{align}\label{xsMfE1}
\MfE(-1)=1, 
\end{align}
so if $\dA=0$, then $\GF_\si(\Tn)/n\pto 1/H$ by
\eqref{tllna3}.

By \eqref{tex*} and \eqref{xsMfE},
\begin{align}\label{xsMfC}
  \MfC(s)=\gG(s+2),\qquad \Re s>-2.
\end{align}

For any trie $T$ and any $\bga\in\cAx$, if $\gf_\si(T)=0$ then
$\gf_\si(T^\bga)=0$. Hence, $\gf_\si(T)\gf_\si(T^\bga)=\gf_\si(T^\bga)$, and thus,
using \eqref{xsfE},
\begin{align}
  \Cov\bigpar{\gf_\si(\Tgl),\gf_\si(\Tglbga)}
  &= \bigpar{1-\E\gf_\si(\Tgl)}\E\gf_\si(\Tglbga)
  \notag\\&
  =(1+\gl)e^{-\gl}\bigpar{1-(1+P(\bga)\gl)e^{-P(\bga)\gl}}.
  \label{pipq}
\end{align}
Consequently, \eqref{lfv2} yields
\begin{align}\label{qj0}
  \MfV(s)
=\sumax\intoo (1+\gl)e^{-\gl}\bigpar{1-(1+P(\bga)\gl)e^{-P(\bga)\gl}}\gl^{s-1}\dd\gl.
\end{align}
For $s>0$, the \rhs{} is, by standard Gamma integrals, evaluated as
\begin{align}\label{qj}
&\sumax\Bigpar{
  \gG(s)+\gG(s+1)-(1+P(\bga))^{-s}\gG(s)-(1+P(\bga))^{1-s-1}\gG(s+1)
    \notag\\&\hskip4em
  -\Pbga(1+P(\bga))^{-s-2}\gG(s+2)}
\notag  \\
  &=\sumax
\frac{(1+s)\gG(s)}{(1+\Pbga)^{s+2}}\Bigpar{(1+\Pbga)^{s+2}-(1+\Pbga)^{2}-s\Pbga}.
\end{align}
The terms in the final sum are, by Taylor expansions, $O(\Pbga^2)$
for fixed $s$, and thus
the sum converges for every $s>0$ by \eqref{Qrr}; 
hence the Mellin transform $\MfV(s)$ is
finite for $s>0$ and equals \eqref{qj}. (Note that the expression in
\eqref{pipq} is positive; hence 
we may interchange the
order of summation and integration in \eqref{lfv2} for real $s$.)
Since the domain of existence of the Mellin transform always is a vertical
strip, this shows that $\MfV(s)$ exists in the half-plane $\Re s>-2$, and
analytic continuation yields that it equals \eqref{qj};
hence,
for all such $s\neq0$,
rewriting $(s+1)\gG(s)=\gG(s+2)/s$,
\begin{align}
  \MfV(s)
  &=\sumax
\frac{\gG(s+2)}{s(1+\Pbga)^{s+2}}\Bigpar{(1+\Pbga)^{s+2}-(1+\Pbga)^{2}-s\Pbga}.   
\end{align}
In particular, 
\begin{align}\label{xsMfV}
  \MfV(-1)=\sumax\frac{\Pbga^2}{1+\Pbga}.
\end{align}
Using \eqref{xsMfC} and \eqref{xsMfV}, we obtain expressions for $\psiC$
and $\psiV$ from \eqref{tevcfou}, leading to (somewhat complicated)
formulas for $\gss(\gl)$ and
$\hgss(n)$ by \eqref{tt0gs}--\eqref{tt0hgs} and \eqref{tt+gs} and
\eqref{tt+hgs}. 
This yields the results found by Jacquet and R\'egnier \cite{JR88,JR89},
\citet{FHZ} and others
by somewhat different methods.

\subsection{Size of fringe tries}\label{SSfringesize}

We turn to the fringe (sub)trees of a random trie.
We first consider their sizes, in this section
measured as their number of external nodes
(leaves).
(Note the difference from \refSS{SSsize}.)

Let $k\ge1$ and let
\begin{align}
\gf_k(T):=\indic{\abse{T}=k}.  
\end{align}
Then, the corresponding additive functional
$\GF_k$ counts the number of fringe trees with exactly $k$ leaves.
Note that $\gf_1=\gfo$ in \refE{Eleaves}, and thus $\GF_1(\Tn)=n$.
In the sequel we mainly consider $k\ge2$.

The functional $\GF_k$ is not increasing, but the functional
$\GF_{\ge  k}:=\sum_{j\ge k}\GF_j$ is,
and $\GF_k = \GF_{\ge k}-\GF_{\ge k+1}$;
furthermore, $\GF_{\ge k}$ has a bounded toll function
$\gf_{\ge  k}:=\sum_{j\ge k}\gf_j$. Hence \refTs{TT} and \ref{TLLN} apply
(with $\gf_+=\gf_{\ge k}$  and $\gf_-=\gf_{\ge k+1}$)
and yield, using also \refR{Rmulti} and \refL{LC}, the following.

\begin{theorem}\label{Tfringesize}
Let $k\ge2$ and consider $\GF_k$, the number of fringe trees with $k$ leaves.
Then, the central limit theorems
\eqref{tt+gl}--\eqref{tt+n} and \eqref{ttgl}--\eqref{ttn}
hold, with all \absmoment{s},
and with asymptotic variances given by 
\eqref{tt+gs}--\eqref{tt+hgs}
(and thus by \eqref{tt0gs}--\eqref{tt0hgs}  when $\dA=0$);
this extends to joint convergence for several $k$.
Furthermore, the means satisfy \eqref{ttmean} and \eqref{ttmeann},
and the laws of large numbers \eqref{tllna1}--\eqref{tllna2} hold.
\nopf
\end{theorem}

Suppose that $k\ge2$.
We then have 
\begin{align}\label{kfe}
  \fEq{k}(\gl):=\E\gf_k(\Tgl)=\P(N_\gl=k)=\frac{\gl^k}{k!}e^{-\gl}.
\end{align}
Hence, or by \refL{LfE},
  the Mellin transform $\MfEq{k}(s)$ exists for $\Re s>-k$, and
\begin{align}\label{kmfe}
  \MfEq{k}(s) = \frac{\gG(k+s)}{k!}.
\end{align}
In particular,
\begin{align}\label{kmfe-1}
  \MfEq{k}(-1)=\frac1{k(k-1)}.
\end{align}

If $\gf_k(\Tgl)=1$ and $\bga\in\cAx$, then $\gf_k(\Tglbga)=1$ only if all
$k$ strings are passed to $\bga$, which has (conditional) probability
$P(\bga)^k$. Hence,
recalling \eqref{kfe},
\begin{align}\label{tomc}
  \Cov\bigpar{\gf_k(\Tgl),\gf_k(\Tglbga)}
  =\frac{\gl^k}{k!}e^{-\gl}\Bigpar{\Pbga^k-\frac{(\Pbga\gl)^k}{k!}e^{-\Pbga\gl}}.  
\end{align}
Consequently, by \eqref{lfv2}, for $\Re s>-k$,
\begin{align}\label{tomm}
  \MfVq{k}(s)
&  =
  \sumax\intoo
 \frac{\gl^k}{k!}e^{-\gl}\Bigpar{\Pbga^k-\frac{(\Pbga\gl)^k}{k!}e^{-\Pbga\gl}}
\gl^{s-1}\dd\gl
            \notag\\&
=  \frac{\gG(s+k)}{k!} \sumax\Pbga^k
  -\frac{\gG(s+2k)}{k!^2} \sumax\frac{\Pbga^k}{(1+\Pbga)^{s+2k}},
\end{align}
where the sums converge since $k\ge2$, see \eqref{Qrr}.
In particular, this easily yields, using
$\gQ k:=\sum_{\ga\in\cA}p_\ga^k<1$ as in \eqref{Qks}--\eqref{Qk},
\begin{align}
  \MfVq{k}(-1)
&  
=  \frac{1}{k(k-1)}\frac{1+\gQ k}{1-\gQ k}
  -\frac{(2k-2)!}{k!^2} \sumax\frac{\Pbga^k}{(1+\Pbga)^{2k-1}}.
\end{align}

The asymptotic normality in \refT{Tfringesize} holds, as stated there,
jointly for different $k$. Furthermore, still by \refR{Rmulti},
it holds jointly with the asymptotic normality of $\GF_\si(\Tn)=\absi{\Tn}$
in \refT{Tsize}.
Asymptotic covariances can be calculated by similar arguments as above.
We illustrate  this for
the asymptotic covariance between
$\GF_k(\Tn)$ and $\absi{\Tn}$ for a given $k\ge2$.
(Calculations for other covariances are slightly more complicated, but the
principle is the same.) 

The bivariate version of \eqref{tt+n} and \eqref{tt+hgs}
(\cf{} \refR{Rmulti}) yields
\begin{align}
  n\qw\Cov\bigpar{\GF_k(\Tn),\GF_\si(\Tn)}
=\hgs_{k\si}(n)+o(1),
\end{align}
where
\begin{align}\label{hgski}
  \hgs_{k\si}(n)
=H\qw \psiVq{k\si}(\ln n) - H\qww \psiCq{k}(\ln n)\psiCq{\si}(\ln n)
\end{align}
where $\psiVq{k\si}$ is given by \eqref{tevcfou} with 
$\MfX=\MfVq{k\si}$, 
the Mellin transform of $\fVq{k\si}$ which by \eqref{fCov} is given by
(noting that $\chi_k=\chi_\si=0$)
\begin{align}\label{bro}
  \fVq{k\si}(\gl)
= 
\Cov\bigpar{\gf_k(\Tgl),\GF_\si(\Tgl)}
  +\Cov\bigpar{\gf_\si(\Tgl),\GF_k(\Tgl)-\gf_k(\Tgl)}.
\end{align}
We note that 
\begin{align}\label{brom}
  \E\bigsqpar{\gf_k(\Tgl)\GF_\si(\Tgl)}&
=
  \P\bigpar{\abse{\Tgl}=k}\E\bigsqpar{\GF_\si(\Tgl)\bigm|\abse{\Tgl}=k}
\notag\\&
=
\fEq{k}(\gl)
\E\bigsqpar{\GF_\si(\Tq{k})}.
\end{align}
Furthermore, if $\gf_\si(\Tgl)=0$, then $\gf_k(\Tgl)=\GF_k(\Tgl)=0$.
Hence, using also \eqref{xsfE} and \eqref{kfe}, \eqref{bro} yields,
with $E_{k}:=\E\GF_\si(\Tq k)=\E\absi{\Tq k}$,
\begin{align}\label{brox}
&  \fVq{k\si}(\gl)
=
\fEq{k}(\gl) 
\Bigpar{E_k-\E\bigsqpar{\GF_\si(\Tgl)}}
+\bigpar{1-\fEq{\si}(\gl)}
\Bigpar{\E\GF_k(\Tgl)-\E\gf_k(\Tgl)}
\notag\\&
 = \frac{\gl^k}{k!}e^{-\gl}
\Bigpar{E_k-\sumaa \fEq{\si}(P(\bga)\gl)}
  +(1+\gl)e^{-\gl}\sumaay\fEq{k}(P(\bga)\gl)
.\end{align}
 This yields after simple calculations, partly arguing as for
\eqref{qj0}--\eqref{qj}, 
\begin{align}\label{broz}
  \MfVq{k\si}(s)
= \frac{\gG(k+s)}{k!}\biggl(E_k
&-\sumaa\frac{(1+P(\bga))^{k+s+1}-1-(k+s+1)P(\bga)}{(1+P(\bga))^{k+s+1}}
\notag\\
&+\sumaay\frac{k+s+1+P(\bga)}{\xpar{1+P(\bga)}^{k+s+1}}P(\bga)^k
\biggr)
.\end{align}
Furthermore, 
$\MfCq{\si}=\gG(s+2)$ by \eqref{xsMfC} and 
\begin{align}
  \MfCq{k}=-s\gG(k+s)/k! 
\end{align}
by \refL{LEVCx} and \eqref{kmfe}.
Finally, as said above, $\psiVq{k\si},\psiCq{\si},\psiCq{k}$ are given by
\eqref{tevcfou}, and \eqref{hgski} yields $\hgs_{k\si}(n)$. 
In the aperiodic case, $\hgs_{k\si}$ is constant and the formulas simplify:
\begin{align}
  \hgs_{k\si}&
=H\qw\MfVq{k\si}(-1)-H\qww\MfCq{k}(-1)\MfCq{\si}(-1)
\notag\\&
=H\qw\MfVq{k\si}(-1)-H\qww\MfEq{k}(-1)\MfEq{\si}(-1)
\notag\\&
= \frac{H\qw}{k(k-1)}\biggl(E_k
-\sumaa\frac{(1+P(\bga))^{k}-1-kP(\bga)}{(1+P(\bga))^{k}}
\notag\\&\phantom{=\frac{H\qw}{k(k-1)}\biggl(E_k}
+\sumaay\frac{k+P(\bga)}{\xpar{1+P(\bga)}^{k}}P(\bga)^k
-H\qw
\biggr)
.
\end{align}

\subsubsection{Asymptotic distributions}
We use these results to study the distribution of the size of
 a (uniformly) random fringe subtree $\Tn\rf$ of $\Tn$,
defined as in \refSS{SSfringe} as $\Tn^v$ for a uniformly random node $v$ in
$\Tn$.
Note that we allow both internal and external nodes $v$.

\begin{remark}
  Alternatively, one might consider a random internal fringe tree by taking
  only internal nodes $v$. This is equivalent to conditioning the fringe
  tree $\Tn^v$ on $v$ being an internal node. Since $v$ is external if and
only if $\Tn^v=\Ti$, this random internal fringe tree 
equals the random fringe tree $\Tn\rf$ (defined as
  above) conditioned on $\Tn\rf\neq\Ti$. The results below are easily
  transferred to this version.
\end{remark}

The total number of nodes in $\Tn$ is
\begin{align}
  \abs{\Tn}=\absi{\Tn}+\abse{\Tn}
=\GF_{\si}(\Tn)+n.
\end{align}
Hence, by \refT{Tsize} and \eqref{tllna2},
\begin{align}
  \abs{\Tn}/n = H\qw\psiEq{\si}(\ln n)+1+\op(1).
\end{align}
Similarly, by \refT{Tfringesize}, for $k\ge2$,
\begin{align}
  \GF_k(\Tn)/n = H\qw\psiEq{k}(\ln n)+\op(1).
\end{align}
This implies the following result,
using also \eqref{xsMfE1} and \eqref{kfe}.
\begin{theorem}\label{Tcor}
The fringe tree size distribution of $\Tn$ satisfies
\begin{align}\label{tcor1}
  \P\bigpar{\abse{\Tn\rf}=k\mid\Tn}
= \frac{\GF_k(\Tn)}{\abs{\Tn}}
=  \begin{cases}
\frac{\psiEq{k}(\ln n)}{\psiEq{\si}(\ln n)+H}+\op(1), & k\ge2,\\
\frac{H}{\psiEq{\si}(\ln n)+H}+\op(1), & k=1.    
  \end{cases}
\end{align}
In particular, if $\dA=0$,
the distribution converges in probability:
\begin{align}\label{tcor0}
    \P\bigpar{\abse{\Tn\rf}=k\mid\Tn}
\pto
  \begin{cases}
\frac{1}{(1+H)k(k-1)}, &k\ge2,\\
\frac{H}{1+H},&k=1.    
  \end{cases}
\end{align}
\nopf  
\end{theorem}
We thus have convergence in probability in the aperiodic case, but (as usual)
oscillations in the periodic case.
It is well-known that the oscillations seen for various properties of tries
tend to be numerically small; hence, the limits in \eqref{tcor0} can be
regarded as approximations also in the periodic case.
Note that the limits in \eqref{tcor0} depend on the letter probabilities $\vp$ 
only through the entropy $H$, and that this limit distribution conditioned on
being $\neq1$ is independent of $\vp$.
In the periodic case, the asymptotics in \eqref{tcor1} depend also on $\dA$;
as always, $\psiEq{k}$ and $\psiEq{\si}$ are given by 
\eqref{tevcfou} with the corresponding
$\MfE$ in \eqref{xsMfE} and \eqref{kmfe}.

\begin{remark}
  The result in \refT{Tcor} is of the quenched type, where we condition on
  the random tree $\Tn$ and obtain approximation or
convergence in probability of the
  conditional distribution. By unconditioning, this immediately implies the 
corresponding annealed result, for
the distribution of $\abse{\Tn\rf}$ where we consider the combined random
experiment of first choosing $\Tn$ at random and then a random fringe
subtree of it.
\end{remark}

\begin{remark}
  The asymptotic distribution in \eqref{tcor0} has probabilities, 
say $\pi_k$,
  decaying as $k^{-2}$ for large $k$. This is similar to
the distribution of the size (now defined as the number of nodes) of
fringe trees in, for example, the random
  recursive tree (with $\pi_k=1/(k(k+1))$, $k\ge1$) 
and the binary search tree (with $\pi_k=2/((k+1)(k+2))$, $k\ge1$);
see \cite{Aldous-fringe,SJ296,SJ306}.
Recall that for conditioned Galton--Watson trees 
(with finite offspring variance), the 
probabilities decay more slowly, as $k^{-3/2}$, see
\cite{Aldous-fringe,SJ264,SJ285}.
\end{remark}

The convergence in probability in \refT{Tcor} can be refined to asymptotic
normality of the conditional probabilities.
In order to include the case $k=1$ in a notationally convenient way, 
we (re)define in the rest of this subsection
\begin{align}\label{psi1}
\psiEq1(t):=H,&& 
\psiCq1(t):=H,&&
\psiVq1(t):=H,&&
\psiVq{1\si}(x):=\psiCq{\si}(x).  
\end{align}
Thus the first case in \eqref{tcor1} holds also for $k=1$.
(Our main justification for the fudge \eqref{psi1} is that it works.
One interpretation, and perhaps explanation, 
is that we replace $\GF_1$ by the almost identical
$\GFF$ in \eqref{e0GFx}, which has $\chi_{\sjw}=0$ and $\psiX(t)$ as in
\eqref{psi1}, 
see \eqref{montec}--\eqref{montv}.)

\begin{theorem}\label{Tcer}
The conditional fringe tree size distribution of $\Tn$, given $\Tn$,
has asymptotically normal fluctuations, in the following sense.
Let $k\ge1$ and   let either $a_{kn}:=  \P\bigpar{\abse{\Tn\rf}=k}
= \E\frac{\GF_k(\Tn)}{\abs{\Tn}}$,
or
$a_{kn}:= \frac{\E\GF_k(\Tn)}{\E\abs{\Tn}}$.
Then, with all moments, as \ntoo,
\begin{align}\label{tcer1}
n\qq\Bigpar{ \P\bigpar{\abse{\Tn\rf}=k\mid\Tn}-a_{kn}}
= n\qq\Bigpar{\frac{\GF_k(\Tn)}{\abs{\Tn}}-a_{kn}}
\approxd
N\bigpar{0,\tgss_k(n)},
\end{align}
where, with $t=\ln n$ and $\psiEz(t):=\psiEq{\si}(t)+H$,
\begin{align}\label{tgssk}
\tgss_k(n) 
&:=
\frac{H}{\psiEz(t)^2}
\Bigpar{\psiVq{k}(t)-2\frac{\psiEq{k}(t)}{\psiEz(t)}\psiVq{k\si}(t)
+\frac{\psiEq{k}(t)^2}{\psiEz(t)^2}\psiVq{\si}(t)}
\notag\\
&\hskip2em
-\frac{1}{\psiEz(t)^4}
\Bigpar{\psiEz(t)\psiCq{k}(t)-
\psiEq{k}(t)\psiCq{\si}(t)}^2
.\end{align}
In particular, if $\dA=0$,
then $\tgss_k(n)$ is constant and, 
\begin{align}\label{tgssk0}
\tgss_k(n)& 
=\frac{H}{(1+H)^4}
\Bigl((1+H)^2\MfVq{k}(-1)
-\frac{2(1+H)}{k(k-1)}\MfVq{k\si}(-1)
\notag\\&\phantom{=\frac{H}{(1+H)^4}}\qquad
+\frac{\MfVq{\si}(-1)-H}{k^2(k-1)^2}\Bigr)
,\qquad\qquad k\ge2,
\\\label{tgss10}
\tgss_1(n) &=
(1+H)^{-4}\bigpar{H^3\MfVq{\si}(-1)-H^2}.
\end{align}
Moreover, the approximation in distribution \eqref{tcer1} holds jointly for
any finite number of $k$, with a multivariate normal distribution 
$N\bigpar{0,(\tgs_{k\ell}(n))_{k,\ell}}$.
\end{theorem}
The asymptotic covariances $\tgs_{k\ell}$ can be expressed similarly to the
case $\ell=k$ in \eqref{tgssk}; we leave the details to the reader.

Note that in the periodic case $\dA>0$, the asymptotic variance
\eqref{tgssk} is a continuous periodic function of $\log n$. However
there is no easy way to find its mean or other Fourier coefficients.

\refT{Tcer} follows from joint convergence in \refTs{Tsize} and
\ref{Tfringesize} by standard methods. We prove first a general lemma of
standard type.

\begin{lemma}\label{Lfett}
  Let $(X_n,Y_n)$ be a sequence of random vectors, and assume that,
as \ntoo, 
  \begin{align}\label{fett1}
    n\qqw \bigpar{X_n-\E X_n,Y_n-\E Y_n} 
\approx 
 N\Bigpar{0,\smatrixx{\gs_{XX}(n) & \gs_{XY}(n)\\\gs_{XY}(n) & \gs_{YY}(n)}},
  \end{align}
where $\E X_n=O(n)$, $\E Y_n=\Theta(n)$ and 
$\gs_{XX}(n),\gs_{XY}(n),\gs_{YY}(n)=O(1)$.
\begin{romenumerate}
  
\item \label{Lfett1}
Then, with $x_n:=\E X_n$ and $y_n:=\E Y_n$,
\begin{align}\label{fett2}
  n\qq\Bigpar{\frac{X_n}{Y_n}-\frac{\E X_n}{\E Y_n}}
\approx  N\Bigpar{0, \frac{n^2}{y_n^2}
\Bigpar{\gs_{XX}(n) -2\frac{x_n}{y_n} \gs_{XY}(n)+\frac{x_n^2}{y_n^2} \gs_{YY}(n)}},
\end{align}
\item \label{Lfett2}
If, moreover, \eqref{fett1} holds with all moments, and $Y_n\ge cn$ a.s., for
some $c>0$ and all $n$, then \eqref{fett2} holds with all moments.
Furthermore, we then may replace $\xfrac{\E X_n}{\E Y_n}$ by 
$\E\bigpar{\xfrac{ X_n}{ Y_n}}$ in \eqref{fett2}.
\end{romenumerate}
\end{lemma}

\begin{proof}
\pfitemref{Lfett1}
  Denote the \lhs{} of \eqref{fett1} by $(X'_n,Y'_n$). Then
  \begin{align}\label{semla}
  n\qq\Bigpar{\frac{X_n}{Y_n}-\frac{\E X_n}{\E Y_n}}    
&=
  n\qq\Bigpar{\frac{x_n+n\qq X'_n}{y_n+n\qq Y'_n}-\frac{ x_n}{y_n}}    
=
\frac{n y_n X'_n-n x_n Y'_n}{y_n(y_n+n\qq Y'_n)}    
\notag\\&=
\frac{y_n}{y_n+n\qq Y'_n}\cdot    
\frac{n}{y_n}
\Bigpar{X'_n-\frac{x_n}{y_n} Y'_n},
  \end{align}
and \eqref{fett2} follows since $\xqfrac{y_n}{y_n+n\qq Y'_n}\pto1$.    
(By the subsequence principle, it suffices to consider subsequences such
that
$x_n/n$, $y_n/n$ and 
$\gs_{XX}(n),\gs_{XY}(n),\gs_{YY}(n)$
 converge.)

\pfitemref{Lfett2}
Let $r$ be a positive integer. The assumptions imply that the sequences
$|X'_n|^r$ and $|Y'_n|^r$ are \ui, and
then it follows that the $r$th absolute powers of the 
variables \eqref{semla} are \ui. Hence the $r$th moment converges in
\eqref{fett2}.

In particular, \eqref{fett2} holds with the first moment, and thus
\begin{align}
 n\qq\Bigpar{\E\frac{X_n}{Y_n}-\frac{\E X_n}{\E Y_n}}
=
 \E\Bigsqpar{ n\qq\Bigpar{\frac{X_n}{Y_n}-\frac{\E X_n}{\E Y_n}}}
\to0.
\end{align}
Hence we may replace $\xfrac{\E X_n}{\E Y_n}$ by 
$\E\bigpar{\xfrac{ X_n}{ Y_n}}$ in \eqref{fett2}.
\end{proof}

\begin{proof}[Proof of \refT{Tcer}]
We apply \refL{Lfett} with $X_n:=\GF_k(\Tn)$ and $Y_n:=\abs{\Tn}=n+\GF_\si(\Tn)$.
As noted above, \eqref{fett1} then holds (with all moments)
by 
\refT{TT}
(or \refTs{Tsize} and \ref{Tfringesize})
together with \refR{Rmulti}, if we define,
using \eqref{tt+hgs} and \eqref{hgski},
with $t=\ln n$,
\begin{align}\label{askxx}
  \gs_{XX}(n)&
:=  \hgss_{k}(n)
=  H\qw\psiVq{k}(t)-H\qww \psiCq{k}(t)^2,
\\\label{askxy}
\gs_{XY}(n)&
:=  \hgs_{k\si}(n)
=H\qw \psiVq{k\si}(t) - H\qww \psiCq{k}(t)\psiCq{\si}(t),
\\\label{askyy}
  \gs_{YY}(n)&
:=  \hgss_{\si}(n)
=  H\qw\psiVq{\si}(t)-H\qww \psiCq{\si}(t)^2.
\end{align}
Furthermore, \eqref{ttmeann} yields
\begin{align}\label{askx}
  x_n/n&=\E\GF_k(\Tn)/n=H\qw\psiEq{k}(t)+o(1),
\\\label{asky}
  y_n/n&=1+\E\GF_\si(\Tn)/n=1+H\qw\psiEq{\si}(t)+o(1)
=H\qw\psiEz(t)+o(1).
\end{align}
Note that, as required by \refL{Lfett},
$x_n/n=O(1)$ and $y_n/n=\Theta(1)$.
Note further that \eqref{askxx}, \eqref{askxy} and \eqref{askx} hold
also for $k=1$
by our special definition \eqref{psi1}. 
(Trivially, with $\gs_{XX}(n)=\gs_{XY}(n)=0$ and $x_n=n$; 
recall that $\GF_1(\Tn)=n$ is deterministic.) 
We have, by \eqref{askxx}--\eqref{asky},
\begin{align}
& \frac{n^2}{y_n^2}
\Bigpar{\gs_{XX}(n) -2\frac{x_n}{y_n} \gs_{XY}(n)+\frac{x_n^2}{y_n^2}
  \gs_{YY}(n)}
\notag\\
&=
\frac{H}{\psiEz(t)^2}
\Bigpar{\psiVq{k}(t)-2\frac{\psiEq{k}(t)}{\psiEz(t)}\psiVq{k\si}(t)
+\frac{\psiEq{k}(t)^2}{\psiEz(t)^2}\psiVq{\si}(t)}
\notag\\
&\quad-\frac{1}{\psiEz(t)^2}
\Bigpar{\psiCq{k}(t)^2-
2\frac{\psiEq{k}(t)}{\psiEz(t)}\psiCq{k}(t)\psiCq{\si}(t)
+\frac{\psiEq{k}(t)^2}{\psiEz(t)^2}\psiCq{\si}(t)^2}
+o(1)
\notag\\
&=
\frac{H}{\psiEz(t)^2}
\Bigpar{\psiVq{k}(t)-2\frac{\psiEq{k}(t)}{\psiEz(t)}\psiVq{k\si}(t)
+\frac{\psiEq{k}(t)^2}{\psiEz(t)^2}\psiVq{\si}(t)}
\notag\\
&\quad-\frac{1}{\psiEz(t)^4}
\Bigpar{\psiEz(t)\psiCq{k}(t)-
\psiEq{k}(t)\psiCq{\si}(t)}^2
+o(1)
\end{align}
which equals $\tgss_k(n)+o(1)$ as defined in \eqref{tgssk}. 
Thus, \refL{Lfett} yields \eqref{tcer1} with all moments.
(Note that $Y_n\ge n$ a.s., so \refL{Lfett}\ref{Lfett2} applies.)

When $\dA=0$, $\psiEz(t)=\MfEq{\si}(-1)+H=1+H$ 
by \eqref{psiX0} and \eqref{xsMfE1},
and 
\eqref{tgssk} reduces to
\eqref{tgssk0}--\eqref{tgss10}, 
using also \eqref{kmfe-1} and \eqref{texas}.
\end{proof}

\subsection{Distribution of fringe tries}\label{SSfringeD}

The previous subsection studied the sizes of fringe tries.
For a more detailed study of the distribution of the fringe trees of the
random trie $\Tn$,
let $T$ be a fixed trie, and consider the toll function
\begin{align}
  \gf_T(T'):=\indic{T'=T}
\end{align}
and the corresponding additive functional $\GF_T$ which counts the number of
fringe trees equal 
to $T$.
Let $\tauk=\abse{T}$, and let
$p_T:=\P(\cT_\tauk=T)$.
Note that $\gf_\Ti$ is as defined in \refE{Eleaves}, and coincides with
$\gf_1$ in \refSS{SSfringesize}, so we are mainly interested in the case
$k\ge2$; then $\chi_T:=\gf_T(\Ti)=0$.
For completeness, we include below also the case $T=\Ti$, but in this case
we use the special definitions \eqref{psi1}; thus  $\psiXq{\Ti}:=\psiXq{1}$.

The functional $\GF_T$ is not increasing, but with $\GF_{>k}:=\GF_{\ge k+1}$
defined in \refSS{SSfringesize}, $\GF_T+\GF_{>k}$ is increasing, and thus
\refTs{TT} and \ref{TLLN} apply to $\GF_T=(\GF_T+\GF_{>k})-\GF_{>k}$.
Furthermore, \refL{LC} applies (with $n_0=k+1$ and $a_n=0$).
Consequently, the arguments in \refSS{SSfringesize} 
yield the following
analogues of \refTs{Tfringesize}, \ref{Tcor}, and \ref{Tcer},
using also \eqref{dodo} which we postpone until after the theorems.

\begin{theorem}\label{TfringeT}
Let $T$ be a fixed trie 
and consider $\GF_T$, the number of fringe trees equal to $T$ 
(as ordered trees).
Then, the central limit theorems
\eqref{tt+gl}--\eqref{tt+n} and \eqref{ttgl}--\eqref{ttn}
hold, with all \absmoment{s},
and with asymptotic variances given by 
\eqref{tt+gs}--\eqref{tt+hgs}
(and thus by \eqref{tt0gs}--\eqref{tt0hgs}  when $\dA=0$);
this extends to joint convergence for several tries $T$.
Furthermore, the means satisfy \eqref{ttmean} and \eqref{ttmeann},
and the laws of large numbers \eqref{tllna1}--\eqref{tllna2} hold.
\nopf
\end{theorem}

\begin{theorem}\label{TcorT}
The fringe tree distribution of $\Tn$ satisfies
\begin{align}\label{tcor1T}
  \P\bigpar{\Tn\rf=T\mid\Tn}
= \frac{\GF_T(\Tn)}{\abs{\Tn}}
= 
\frac{\psiEq{T}(\ln n)}{\psiEq{\si}(\ln n)+H}+\op(1)
.\end{align}
In particular, if $\dA=0$,
then the distribution converges in probability:
\begin{align}\label{tcor0T}
    \P\bigpar{\abse{\Tn\rf}=T\mid\Tn}
\pto
  \begin{cases}
\frac{p_T}{(1+H)k(k-1)}, &\abs{T}\ge2,\\
\frac{H}{1+H},&T=\Ti.    
  \end{cases}
\end{align}
\nopf  
\end{theorem}

\begin{theorem}\label{TcerT}
The conditional fringe tree distribution of $\Tn$, given $\Tn$,
has asymptotically normal fluctuations, in the following sense.
Let $T$ be a fixed trie and   let either $\aTn:=  \P\xpar{\Tn=k}
= \E\frac{\GF_T(\Tn)}{\abs{\Tn}}$,
or
$\aTn:= \frac{\E\GF_T(\Tn)}{\E\abs{\Tn}}$.
Then, with all moments, as \ntoo,
\begin{align}\label{tcer1T}
n\qq\Bigpar{ \P\bigpar{{\Tn\rf}=T\mid\Tn}-\aTn}
= n\qq\Bigpar{\frac{\GF_T(\Tn)}{\abs{\Tn}}-\aTn}
\approxd
N\bigpar{0,\tgss_T(n)},
\end{align}
where, with $t=\ln n$ and $\psiEz(t):=\psiEq{\si}(t)+H$,
\begin{align}\label{tgssT}
\tgss_T(n) 
&:=
\frac{H}{\psiEz(t)^2}
\Bigpar{\psiVq{T}(t)-2\frac{\psiEq{T}(t)}{\psiEz(t)}\psiVq{T\si}(t)
+\frac{\psiEq{T}(t)^2}{\psiEz(t)^2}\psiVq{\si}(t)}
\notag\\
&\hskip2em
-\frac{1}{\psiEz(t)^4}
\Bigpar{\psiEz(t)\psiCq{T}(t)-
\psiEq{T}(t)\psiCq{\si}(t)}^2
.\end{align}
In particular, if $\dA=0$,
then $\tgss_T(n)$ is constant.
Moreover, the approximation in distribution \eqref{tcer1T} holds jointly for
any finite number of $T$, with a multivariate normal distribution 
$N\bigpar{0,(\tgs_{TT'}(n))_{T,T'}}$.
\end{theorem}

Asymptotic means, variances and covariances may be calculated as in
\refSS{SSfringesize}. Suppose $\tauk:=\abse{T}\ge2$.
Then, recalling \eqref{kfe},
\begin{align}\label{de}
  \fEq{T}(\gl)&
=\P(\Tgl=T)
=\P(N_\gl=\tauk)\P\bigpar{\Tgl=T\mid N_\gl=k}
\notag\\&
=\P(N_\gl=\tauk)\P(\Tq{\tauk}=T)
=p_T\fEq{k}(\gl)
  =p_T\frac{\gl^k}{k!}e^{-\gl}.
\end{align}
Hence,  using \eqref{kmfe}--\eqref{kmfe-1}, for $\Re s>-k$,
\begin{align}\label{dd}
  \MfEq{T}(s)
=p_T  \MfEq{k}(s)
=p_T\frac{\gG(s+k)}{k!}.
\end{align}
and
\begin{align}\label{dodo}
  \MfEq{T}(-1)
=\frac{p_T}{k(k-1)}.
\end{align}

Furthermore,
if $|\bga|>0$, then $\gf_T(\Tgl)\gf_T(\Tglbga)=0$.
Hence, \cf{} \eqref{tomc},
\begin{align}\label{df}
  \Cov\bigpar{\gf_T(\Tgl),\gf_T(\Tglbga)}
&  =\fEq{T}(\gl)\indic{|\bga|=0}-\fEq{T}(\gl)\fEq{T}(P(\bga)\gl)
\notag\\&\hskip-3em
 =\fEq{T}(\gl)\indic{|\bga|=0}-p_T^2\frac{\gl^k(\Pbga\gl)^k}{k!^2}e^{-(1+\Pbga)\gl}.  
\end{align}
Consequently, by \eqref{lfv2} and \eqref{dd},
\cf{} \eqref{tomm},
for $\Re s>-k$,
\begin{align}\label{dv}
  \MfVq{T}(s)
  &  =
 \MfEq{T}(s)
    -
p_T^2 \sumax\intoo
 \frac{\Pbga^k\gl^{2k}}{k!^2}e^{-(1+\Pbga)\gl}
\gl^{s-1}\dd\gl
\notag\\&
  =
    p_T\frac{\gG(k+s)}{k!}
-p_T^2\frac{\gG(s+2k)}{k!^2} \sumax\frac{\Pbga^k}{(1+\Pbga)^{s+2k}}
.\end{align}
In particular, 
\begin{align}\label{dv-1}
  \MfVq{T}(-1)
  =
\frac{p_T}{k(k-1)}
-p_T^2\frac{(2k-2)!}{k!^2} \sumax\frac{\Pbga^k}{(1+\Pbga)^{2k-1}}.
\end{align}

We leave further calculations of variances and covariances to the reader.

\begin{example}[Cherries]\label{Echerry}
  A cherry is the tree $\Tch$
with one internal node (the root) and two external
  nodes.
  This is a trie generated by two strings with  different
  first letters.
  Suppose for simplicity that $\cA=\setoi$, and write 
$p_0=p$, $p_1=q$. 
  Then $p_\Tch=\P(\cT_2=\Tch)=2pq$.
Hence, \eqref{dd} and \eqref{dv} yield
\begin{align}\label{chv}
  \MfEq{\Tch}(s) &=pq\gG(s+2),
            \\
  \MfVq{\Tch}(s)
          &  =
            pq\gG(s+2)
-p^2q^2\gG(s+4) \sumax\frac{\Pbga^2}{(1+\Pbga)^{s+4}}.
\end{align}
\end{example}

\subsection{Protected nodes}

The \emph{rank} of a node $v$ in a rooted tree is the minimum distance to a
descendant of $v$ that is a leaf.
(In particular, leaves are the nodes with rank 0.)
For a trie $T$ and a node $\bga\in T$, we thus have,
recalling that $\nbga$ is the number of the generating strings that have
$\bga$ as a prefix, \cf{} \eqref{bamse},
\begin{align}\label{rank}
  \rank(\bga):=\min\bigset{|\bgb|:\bga\bgb \text{ is a leaf in $T$}}
=\min\bigset{|\bgb|:\nx{\bga\bgb}=1}.
\end{align}

Nodes with rank $\ge k$ are called \emph{$k$-protected}.
Here $k\ge0$; the interesting cases are $k\ge2$.
(For $k=1$ we get just the internal nodes. 
The results below then reduce to corresponding results in \refSS{SSsize}.)

Let $\GFkprot(T)$ be the number of $k$-protected nodes in $T$.
This is an additive functional with toll function,
for $T\neq\emptyset$, 
\begin{align}\label{gfkprot}
  \gfkprot(T):=\indic{\text{the root $\rot$ of $T$ is $k$-protected}}
=\indic{\rank(\rot)\ge k}
.
\end{align}
$\GFkprot$ is not an additive functional, since adding a new leaf
may make some nodes unprotected. However, the only nodes that may lose
protection are the $k-1$ nearest ancestors of the new leaf, and thus
$\GFkprot+k\GF_{\Ti}$ is an increasing functional.
Hence, \refTs{TT} and \ref{TLLN} apply to $\GFkprot$, and we obtain
analogues of \refTs{Tfringesize}--\ref{Tcer} and \ref{TfringeT}--\ref{TcerT}
yielding asymptotic normal distributions of the number and proportion of
$k$-protected nodes.
(We omit detailed statements.)

At least the asymptotic mean is rather easily calculated.
For a trie $T$ we have, using \eqref{gfkprot} and \eqref{rank}, 
for $k\ge1$ and including
the case $T=\emptyset$,
\begin{align}\label{gfkprot2}
  \gfkprot(T)
=\indic{\nx{\bgb}\neq1 \;\forall \bgb\in\cA^{k-1} }-\indic{\nx{\emptystring}=0}
.
\end{align}
In particular, for the Poisson random trie $\Tgl$, 
where $\nx{\bgb}=\Nglx{\bgb}$,
\begin{align}\label{ql}
  \gfkprot(\Tgl)
=\indic{\Nglbga\neq1 \;\forall \bga\in\cA^{k-1} }-\indic{N_\gl=0}
.
\end{align}
Hence,
since $\Nglbga\sim\Po(P(\bga)\gl)$ are independent for $\bga\in\cA^{k-1}$,
\begin{align}
  \fE(\gl)&=
\E\gfkprot(\Tgl)
=\prod_{\bga\in\cA^{k-1}} \bigpar{1- P(\bga)\gl e^{P(\bga)\gl}}-e^{-\gl}
\label{qm0}\\
&=\sum_{\emptyset\neq S\subseteq\cA^{k-1}}(-1)^{|S|}\prod_{\bga\in S}P(\bga) \cdot
e^{-\sum_{\bga\in S}P(\bga) \gl} \gl^{|S|}
-\bigpar{e^{-\gl}-1}
.\label{qm}
\end{align}
For $-1<\Re s<0$, the Mellin transform $\MfE(s)$ can be calculated 
using \eqref{qm} in \eqref{mellin} and integrating termwise, yielding
\begin{align}\label{qn}
  \MfE(s)=
\sum_{\emptyset\neq S\subseteq\cA^{k-1}}(-1)^{|S|}\prod_{\bga\in S}P(\bga) \cdot
\Bigpar{\sum_{\bga\in S}P(\bga)}^{-|S|-s} \gG(|S|+s)
-\gG(s).
\end{align}
We know that $\MfE$ is analytic in the strip $-2<\Re s<0$, see
\refR{Rexists}, and the \rhs{} of \eqref{qn} 
is analytic for $-2<\Re s<0$ except possibly at $s=-1$.
Hence, \eqref{qn} holds in this strip, with a removable singularity at $-1$.
To find $\MfE(-1)$, let $g(s)$ be the sum over $|S|\ge2$ in \eqref{qn};
then, using \eqref{Qkm},
\begin{align}\label{qo}
  \MfE(s)&=-\sum_{\bga\in\cA^{k-1}}P(\bga)^{-s}\gG(s+1)+g(s)-\gG(s)
\notag\\&
=-\gG(s+1)\gQ {-s}^{k-1}-\gG(s)+g(s)
\notag\\&
=-\frac{\gG(s+2)}{s}\cdot
\frac{s\gQ {-s}^{k-1}+1}{s+1}+g(s)
  \end{align}
and thus, letting $s\to-1$ and recalling \eqref{QH},
\begin{align}\label{qr}
 \MfE(-1)&= \dds\bigpar{s\gQ {-s}^{k-1}}\big|_{s=-1}+g(-1)
\notag\\&
=1-
(k-1)\dds \gQ {-s}\big|_{s=-1}+g(-1)
\notag\\&
=1-(k-1)H+
\sum_{S\subseteq\cA^{k-1},|S|\ge2}(-1)^{|S|}
\frac{\prod_{\bga\in S}P(\bga)}
{\bigpar{\sum_{\bga\in S}P(\bga)}^{|S|-1}} (|S|-2)!
.
\end{align}

As in earlier applications, 
this yields asymptotics for the mean. $\MfV$ and variance asymptotics may be
calculated by
similar arguments, but the results are more complicated and we omit the details.

\begin{example}
For the number of 2-protected nodes
in a binary trie we have
 $k=2$ and $\cA=\setoi$, and then \eqref{qn} and \eqref{qr} yield
\begin{align}\label{q2}
\MfE(s)&=-\bigpar{p_0^{-s}+p_1^{-s}}\gG(s+1)+p_0p_1\gG(s+2)-\gG(s)
\intertext{with}
  \MfE(-1)&=1-H+p_0p_1.
\end{align}
In particular, for $p_0=\frac12$, $\MfE(-1)=\frac{5}{4}-\ln 2\doteq0.55685$.
Hence, 
by the analogue of \refTs{Tcor} and \ref{TcorT},
for a large random symmetric binary trie  
the proportion of 2-protected
nodes is
roughly (ignoring small oscillations), 
\begin{align}
  \frac{\MfE(-1)}{1+H}=\frac{5/4-\ln2}{1+\ln2}\doteq0.32888.
\end{align}

For comparison, the corresponding proportion in a binary search tree
converges (in probability) to $11/30\doteq0.36667$
\cite{MahmoudWard-bst, Bona,SJ283,SJ296};  
in a uniformly random binary tree the proportion converges to
$33/64=0.515625$ 
\cite{SJ283}. 

In this example, one can also use the easily verified fact that for any
binary tree $T$ with $\abse{T}>1$, with $T_\ch$ the cherry in \refE{Echerry},
\begin{align}
  \GF_{\xprot2}(T)=\GF_\si(T)-\GF_{\Ti}+\GF_{T_\ch}.
\end{align}
Hence results in this case alternatively follow from results in the 
\refSs{SSsize}--\ref{SSfringeD}.
\end{example}

In general, the sums in \eqref{qn} and \eqref{qr} have almost
$2^{|\cA|^{k-1}}$ terms, which quickly becomes very large for larger $k$ or
$|\cA|$. However, in the symmetric case, the sums simplify by symmetry since
the summands then depend only on $|S|$.

\begin{example}\label{Eprotsymm}
  Consider the symmetric case with $|\cA|=r\ge2$ and $p_\ga=1/r$ for all
  $\ga\in\cA$. We calculate $\MfE$, which we denote by $\MfEq{\kprot,r}$.

For $k=2$, 
\eqref{qn} and \eqref{qr} yield 
\begin{align}\label{gang}
  \MfEq{\xprot2,r}(s)
&=
\sum_{j=1}^r \binom{r}{j}(-1)^j {r^{s}} j^{-j-s}\gG(s+j)-\gG(s). 
\\
  \MfEq{\xprot2,r}(-1)
&=
1-\ln r+\sum_{j=2}^r \binom{r}{j}(-1)^j {r^{-1}} j^{1-j}(j-2)!
\notag\\
&=
1-\ln r+\sum_{j=2}^{r} (-1)^j\frac{(r-1)!}{(r-j)!\,(j-1)j^j } 
. \label{gang-1}
\end{align}
Furthermore, for general $k\ge2$, \eqref{qn}  implies
\begin{align}\label{gangk}
  \MfEq{\kprot,r}(s)
= \MfEq{\xprot2,r^{k-1}}(s).
\end{align}
For example,  for the binary case and $k=3,4$,
\begin{align}
    \MfEq{\xprot3,2}(-1)&
= \MfEq{\xprot2,4}(-1)
=\frac {1897}{1152}-2\ln2
\doteq 0.26041 
,\\
\MfEq{\xprot4,2}(-1)&
= \MfEq{\xprot2,8}(-1)
=
\frac {13666493449090877}{6245298339840000}
-3\ln2
\notag\\&
\doteq 0.10884 
.\end{align}
Recall that the asymptotic  proportion of $k$-protected nodes,
ignoring the oscillations, equals $\MfEq{\kprot,2}(-1)/(1+H)$, where $H=\ln2$.
\refTab{tab:protected} gives numerical values for small $k$.
\begin{table}
  \centering
  \begin{tabular}{r|l|l}
$k$ & $\MfEq{\kprot,2}$    & $\MfEq{\kprot,2}/(1+\ln2)$\\
\hline
1 & 1 & 0.59061 \\
2 & 0.55685 &0.32888 \\
3& 0.26040 & 0.15380\\
4& 0.10884 & 0.06428\\
5& 0.04718 & 0.02786\\
6& 0.02182 & 0.01289\\
7& 0.01039 & 0.00613\\
8& 0.00502 & 0.00296\\
9& 0.00244 & 0.00144\\
10& 0.00120 & 0.00070
  \end{tabular}
  \caption{Approximate proportions 
of $k$-protected nodes in symmetric   random binary  tries
(right column).}
  \label{tab:protected}
\end{table}
%
%
%
\end{example}

The numerical values in \refTab{tab:protected} suggest that the proportions
decrease geometrically as \ktoo.
In fact, this holds for any $r$.

\begin{theorem}\label{Tmars}
Consider symmetric tries as in \refE{Eprotsymm}, and assume $r,k\ge2$.
  As \ktoo{} or \rtoo{} (or both),
  \begin{align}\label{tmarsk}
    \MfEq{\kprot,r}(-1)
\sim \frac1{2r^{k-1}}.
  \end{align}
In particular, for symmetric binary tries,
  \begin{align}\label{tmars2}
    \MfEq{\xprot{k},2}(-1)
\sim 2^{-k}.
  \end{align}
\end{theorem}

In other words, for large $k$ and much larger $n$, the proportion of
$k$-protected nodes in a symmetric binary trie is roughly (again ignoring
oscillations) $2^{-k}/(1+\ln2)$.

\begin{proof}
  By \eqref{gangk}, it suffices to consider $k=2$.
In this case, \eqref{qm0} yields
\begin{align}\label{mar1}
  \MfEq{\xprot2,r}(-1)
&=
\intoo \fEq{\xprot2,r}(\gl)\gl\qww\dd\gl
\notag\\&
=\intoo \Bigpar{\Bigpar{1-\frac{\gl}{r}e^{-\gl/r}}^r-e^{-\gl}}
\frac{\ddx \gl}{\gl^2}
.\end{align}
Let
\begin{align}\label{marg}
  g_r(x):=
\Bigpar{1-\frac{x}{r}e^{-x/r}}^r-e^{-x}.
\end{align}
Note first that as \rtoo,
by the change of variables $x=ru$ and dominated convergence,
\begin{align}\label{mar3}
r  \int_r^\infty g_r(x)\frac{\ddx x}{x^2}
=   \int_1^\infty g_r(ru)\frac{\ddx u}{u^2}
\le    \int_1^\infty \bigpar{1-ue^{-u}}^r\,\frac{\ddx u}{u^2}
\to0.
\end{align}

Furthermore, 
for $x\in[0,r]$,
write $y_1:=1-\frac{x}{r}e^{-x/r}$ and $y_0:=e^{-x/r}$,
so $g_r(x)=y_1^r-y_0^r$.
For $y\in\oi$, we have
$0\le (1-ye^{-y})-e^{-y}\le y^2/2$, and 
\begin{align}\label{mar4}
  (1-ye^{-y})-e^{-y}= \tfrac12y^2+O(y^3),
\end{align}
and thus, by the mean value theorem, for some $\gth=\gth(x,r)\in\oi$,
\begin{align}
  g_r(x)&
=y_1^r-y_0^r=(y_1-y_0)r \xpar{y_0+\gth(y_1-y_0)}^{r-1}
\label{mara}\\&
=r\Bigpar{\frac12\parfrac{x}{r}^2 +O\parfrac{x}{r}^3}
 \Bigpar{1-\frac{x}{r}+O\parfrac{x}{r}^2}^{r-1}.
\label{mar5}
\end{align}
Hence, for fixed $x\ge0$, $rg_r(x)\to \frac12x^2e^{-x}$ as \rtoo.
Moreeover, again by \eqref{mara}, for $x\in[0,r]$ and $r\ge2$,
\begin{align}\label{mar6}
  r g_r(x) \le r^2 (y_1-y_0) y_1^{r-1}
\le r^2\parfrac{x}{r}^2 \Bigpar{1-\frac{x}{r}e^{-1}}^{r-1}
\le x^2 e^{-x/(2e)}.
\end{align}
Consequently, as \rtoo, dominated convergence yields
\begin{align}\label{mar7}
r  \int_0^rg_r(x)\frac{\ddx x}{x^2}
\to
\intoo\frac12 x^2e^{-x}\frac{\ddx x}{x^2}=\frac12,
\end{align}
which together with \eqref{mar3} and \eqref{mar1}--\eqref{marg} 
yields the result.
\end{proof}

\begin{problem}
  Extend these results to the non-symmetric case.
In particular, for a general $\vp$, does $\MfEq{\kprot}$ decrease
geometrically as \ktoo?
If so, at which rate?
\end{problem}

Some similar (but less complete) 
results for binary and $m$-ary search trees are given in
\cite{BonaPittel} and \cite[Section 10.1]{SJ306}.

\subsection{Number of subtrees}
Let $s(T)$  be the number of subtrees of a tree $T$, and
$s_1(T)$ the number of subtrees that contain the root.
Then, as noted by \citet{Wagner12,Wagner15},
$\GF(T):=\ln(1+s_1(T))$ is an additive functional with toll function
\begin{align}
  \gf(T):=\ln\bigpar{1+1/s_1(T)}.
\end{align}
The functional $\gf$ is bounded (by $\ln2$). Moreover, $\GF(T)$ is an
increasing functional, and thus \refTs{TT} and \ref{TLLN} apply and yield 
asymptotic normality for $\GF(\Tn)$. 
This time we do not see a simple argument showing $\Var \GF(\Tn)=\Omega(n)$, 
so we cannot apply \eqref{ttgl}--\eqref{ttn}; nevertheless
\eqref{tt+gl}--\eqref{tt+n} hold, and we obtain the following theorem.
(We conjecture that $\Var \GF(\Tn)=\Omega(n)$ in this application too, but
leave this as an open problem.)

\begin{theorem}\label{Tsub}
  As \ntoo,
  \begin{align}\label{tsub}
\frac{\log s(\Tn)-\E[\log s(\Tn)]}{\sqrt n}    
\approx\frac{\log s_1(\Tn)-\E[\log s_1(\Tn)]}{\sqrt n}    
\approx N\bigpar{0,\hgss(n)},
  \end{align}
with all \absmoments, where $\hgss(n)$ is a continuous bounded function
given by \eqref{tt+hgs}.
\end{theorem}

\begin{proof}
 We have $\GF(T)=\ln s_1(T)+O(1)$
and $s_1(T)\le s(T)\le \abs{T}s_1(T)$ (see \cite{Wagner12,Wagner15}),
and thus, 
recalling $\abs{\Tn}=n+\GF_\si(\Tn)$ and
using \eqref{tllna2} for $\GF_\si$, 
\begin{align}\label{tra}
  \ln s(\Tn)
= \ln s_1(\Tn) + O\bigpar{\ln\abs{\Tn}}
= \GF(\Tn) + O\bigpar{\ln n}+ \Op(1),
\end{align}
where as usual $\Op(1)$ denotes a random variable (depending on $n$) that is
bounded in probability.
Furthermore, for any fixed $m\ge1$, by \refT{Tsize},
\begin{align}\label{tray}
  \E\bigsqpar{ \ln^m\abs{\Tn}}
\le C_m \E\bigsqpar{\abs{\Tn}^{m/4}}
\le C_m n^{m/4} =o\bigpar{n^{m/2}}.
\end{align}
Taking $m=1$, we obtain from \eqref{tra} and \eqref{tray},
\begin{align}\label{trax}
\E\bigsqpar{  \ln s(\Tn)}
= \E\bigsqpar{\ln s_1(\Tn)} + o\bigpar{n\qq}
= \E\GF(\Tn) + o\bigpar{n\qq},
\end{align}
The asymptotic normality \eqref{tt+n} in \refT{TT} 
together with \eqref{tra} and \eqref{tray}
yields
\eqref{tsub}, with \absmoments.
\end{proof}

Cf.\ similar results for some other classes of random trees in
\cite{Wagner12,Wagner15} and \cite{SJ285}.

\subsection{Shape parameter}
The \emph{shape parameter} is defined as the logarithm of the product of all
fringe tree sizes; this is thus an additive functional $\GF(T)$
with toll function $\gf(T)=\ln\abs{T}$.
The shape functional $\GF(T)$ is increasing. However, $\gf(T)$ is unbounded,
so we cannot use \refTs{TT} and \ref{TLLN} as stated.
Nevertheless, we have by \refT{Tsize}, as in \eqref{tray}, for any $r\ge1$,
\begin{align}\label{traz}
  \E\bigsqpar{ \ln^r\abs{\Tgl}}
\le C_r \E\bigsqpar{\abs{\Tgl}^{r/4}}
\le C_r \gl^{r/4}.
\end{align}
In particular, \eqref{bze}, \eqref{bzv} and \eqref{Efr} (for any $r$) hold,
and thus by \refR{RTT}, or using \refT{T1n} below,
we find, for example,
\begin{align}
  \frac{\GF(\Tn)-\E\GF(\Tn)}{\sqrt n}\approx N\bigpar{0,\hgss(n)},
\end{align}
with all moments.

Cf.\ similar results for some other classes of random trees in
\cite{MeirMoon}, \cite{FillKapur} and
\cite{Wagner15}.

\subsection{Bucket tries}\label{SSbucket2}
The results above are easily adapted to bucket tries for a fixed bucket size
$b$, by noting that the internal nodes of a bucket trie are precisely the
nodes $\bga$ of the corresponding trie with $\nu_\bga>b$.
In particular, if the bucket tries corresponding to $\Tgl$ and $\Tn$ are denoted
$\Tgl\bb$ and $\Tn\bb$, 
then $\absi{\Tgl\bb}=\GF_{>b}(\Tgl)$
and $\absi{\Tn\bb}=\GF_{>b}(\Tn)$,
and it follows that \refT{Tsize} holds for $\Tgl\bb$ and $\Tn\bb$ too.
We have, generalizing the case $b=1$ in \eqref{xsfE}--\eqref{xsMfE1},
\begin{align}\label{bar}
  \fE(\gl)
=  \fEq{>b}(\gl)
=\P(N_\gl>b)=1-\sum_{i=0}^b \frac{\gl^i}{i!}e^{-\gl}
\end{align}
and thus, 
\begin{align}\label{barb}
  \MfE(\gl)&=\intoo\Bigpar{1-\sum_{i=0}^b \frac{\gl^i}{i!}e^{-\gl}}\gl^{s-1}\dd\gl
=-\frac{\gG(s+b+1)}{b!\,s},
\\
\MfE(-1)&=\frac{1}{b}\label{barb1}
.\end{align}

Consider now the number of buckets containing exactly $k$ strings, for some
fixed $k\in\set{1,\dots,b}$. If we assume $n>b$ (so the root is internal),
this equals the additive functional $\GF_{b;k}$ with toll function
\begin{align}
  \gf_{b;k}(T) = \suma \indic{\nu_{\rot} >b,\, \nu_\ga=k}.
\end{align}
$\GF_{b;k}$ is not increasing, but 
$\GF_{b;k}+\GFo$ is, so
\refTs{TT} and \ref{TLLN} apply to 
$\GF_{b;k}=(\GF_{b;k}+\GFo)-\GFo$. 

 Since $N_\gl-\Nglga$ and $\Nglga$ are independent,
\begin{align}
\fEq{b;k}(\gl)&
=\E \gf_{b;k}(\Tgl) 
= \suma \P\bigpar{N_{\gl} >b,\, \Nglga=k}
\notag\\&
= \suma \P\bigpar{N_{\gl}-\Nglga >b-k}\P\bigpar{\Nglga=k}
\notag\\&
= \suma \Bigpar{1-\sum_{i=0}^{b-k}\frac{(1-p_\ga)^i\gl^i}{i!}e^{-(1-p_\ga)\gl}}
\frac{p_\ga^k\gl^k}{k!}e^{-p_\ga\gl}.
\end{align}
Hence,
\begin{align}
  \MfEq{b;k}(s)&
=\frac{1}{k!} \suma p_\ga^{-s}\gG(s+k)
-
\suma \sum_{i=0}^{b-k}
\frac{ p_\ga^k}{k!} 
\frac{(1-p_\ga)^i}{i!}\gG(s+k+i)
\notag\\&
=\frac{1}{k!}\bigpar{\rho(-s)-\rho(k)}\gG(s+k)
-
\sum_{i=1}^{b-k}\suma{ p_\ga^k}{(1-p_\ga)^i}\frac{\gG(s+k+i)}{k!\,i!}.
\end{align}
In particular, for $k\ge2$,
\begin{align}\label{topp}
  \MfEq{b;k}(-1)&
=\frac{1-\rho(k)}{k(k-1)}
-
\sum_{i=1}^{b-k}\suma{ p_\ga^k}{(1-p_\ga)^i}\frac{(k+i-2)!}{k!\,i!}.
\end{align}
For $k=1$, we obtain by taking the limit as $s\to-1$, 
\eqref{topp} with the first fraction (now undefined)  replaced by $H$,
\cf{} \eqref{monb}.

We leave calculations of $\MfV$ and (co)variances to the reader.

\section{General central limit theorems}\label{Sgeneral}

We state here several related general central limit theorems for additive
functionals on tries; proofs are given in \refS{Spfcentral}.
As said in the introduction, the theorems use conditions on moments of the
additive functionals and their toll functions;
we will later obtain \refT{TT} as a
special case of the results below by using \refT{TEVC} to verify these
moment conditions.

In the statements of the
theorems below, we use several functions $a(\gl)$, $b(\gl)$ and $c(\gl)$,
(with indices in the multivariate versions).
This might seem frightening, but is intended to be friendly and flexible for
applications; the meaning of these functions is as follows.

First, $a(\gl)$ is an approximation of the mean $\E \GF(\Tgl)$, and $b(\gl)$
and $c(\gl)$ are approximations of variances and covariances, see
\eg{} \eqref{EF}, \eqref{VF}, \eqref{CovFN}.
We may choose  $a(\gl):=\E \GF(\Tgl)$,
$b(\gl):=\Var \GF(\Tgl)$, and
$c(\gl):=\Cov\bigpar{\GF(\Tgl),N_\gl}$,
and then \eqref{EF}, \eqref{VF} and \eqref{CovFN} are trivial,
but in applications it is often
preferable to use simpler approximations of the means and (co)variances,
which is precisely what these functions are intended to be.
Note that here the means and (co)variances are for the Poisson model, also
in the
theorems for the  model with fixed $n$; this is both because of our
proofs, and because in applications, the moments typically are easier to
compute for the Poisson model. 
However, the mean for fixed $n$ is asymptotically the same as for the
Poisson model, and the variances are related;
see \eg{} \eqref{EFn}--\eqref{VFn}.

\begin{remark}\label{Rgla}
The conditions on these functions
  in the theorems below  are asymptotic, as $\gltoo$.
Hence the values of these functions for small $\gl$ are irrelevant, and it is
enough that they are defined for large $\gl$.
\end{remark}

\begin{remark}\label{Rgln}
In the theorems below we assume that the assumptions hold for arbitrary
 real $\gl$. (Or at least for sufficiently large $\gl$, see \refR{Rgla}.)
  However, the results hold (by the same proofs) also if we consider only a
  given sequence   $\gl_n\to\infty$.
\end{remark}

In general, there
there are oscillations in the variance.
We therefore state many of the results as approximations (in distribution)
using the notation $\approxd$ defined in \refSS{SSapprox}.
(This is especially important in the multivariate versions.)
Note that we then include rather trivial cases when
the normalized variable
(\eg{} the \lhs{} of \eqref{t1x} or \eqref{t1bx})
converges to 0 (in probability).

We begin with a general central limit theorem for the Poisson model.

\begin{theorem}\label{T1}
  Let $\gf$ be a toll function and
let $\GF$ be the corresponding additive functional given by \eqref{GF}.
Let 
$a(\gl)$ and $b(\gl)$  be real-valued
functions and
suppose that
for some $r>2$, 
as \gltoo, 
\begin{align}
  \E \GF(\Tgl)& = a(\gl)+o\bigpar{\sqrt{\gl}}\label{EF}
  \\ 
  \Var \GF(\Tgl)&=
   b(\gl)+o(\gl),\label{VF}
   \\
  \Var \GF(\Tgl)&=
   O(\gl),\label{VFO}
   \\
  \Var \gf(\Tgl) &= o(\gl), \label{Vf}
  \\
  \E |\gf(\Tgl)-\E \gf(\Tgl)|^r &= O\bigpar{\gl^{r/2}}.
 \label{Efr}
\end{align}
\begin{romenumerate}
  
\item \label{T1a}
Then, as \gltoo,
\begin{align}\label{t1x}
  \frac{\GF(\Tgl)-a(\gl)}{\sqrt{\gl}}
  \approxd N\bigpar{0,b(\gl)/\gl}
\end{align}
or, equivalently,
\begin{align}\label{t1bx}
  \frac{\GF(\Tgl)-\E \GF(\Tgl)}{\sqrt{ \gl}}
  \approxd N\bigpar{0,\Var[\GF(\Tgl)]/\gl},
\end{align}
in both cases with all \absmoment{s} of order $s<r$.
  
\item\label{T1b}
  Suppose further that
  \begin{align}
  b(\gl) &= \Omega(\gl).
  \label{bB}    
  \end{align}
Then, as \gltoo,
\begin{align}\label{t1}
  \frac{\GF(\Tgl)-a(\gl)}{\sqrt{b(\gl)}}
  \dto N(0,1)
\end{align}
and
\begin{align}\label{t1b}
  \frac{\GF(\Tgl)-\E \GF(\Tgl)}{\sqrt{\Var \GF(\Tgl)}}
  \dto N(0,1),
\end{align}
in both cases with convergence of all \absmoment{s} of order $s<r$.

\end{romenumerate}
\end{theorem}

\begin{remark}
  \label{Rs<r}
  We do not know whether \eqref{Efr} implies that also
  the $r$th  moment  converges in \eqref{t1}--\eqref{t1b},
  and we leave this as an open problem.
  (The proof shows that this moment stays bounded, but this is not enough to
  imply convergence.)
  Nevertheless, the theorem shows that if \eqref{Efr} holds for 
all $r>2$, 
then   \eqref{t1x}--\eqref{t1bx}
  hold with
  all \absmoments{}
  and that, if also \eqref{bB} holds, then
  all \absmoments{}  converge
  in \eqref{t1}--\eqref{t1b}.
  The same applies to the theorems below.
\end{remark}

We derive results for the  model with fixed $n$
by conditioning. For this we assume that the functional $\GF$
can be written as a difference between two
increasing functionals with suitable conditions.
(In particular, the theorem applies to increasing functionals $\GF$.)

\begin{theorem}\label{T1n}
  Let $\gf$ be a toll function and
let $\GF$ be the corresponding additive functional given by \eqref{GF}.
Let 
$b(\gl)$ be a real-valued
function that satisfies
\eqref{VF},
and let $c(\gl)$ be a function such that,
as \gltoo,
\begin{align}
  \label{CovFN}
    \Cov\bigpar{ \GF(\Tgl),N_\gl}&=
   c(\gl)+o(\gl).
\end{align}
Suppose further that $\gf=\gf_{+}-\gf_{-}$ for some toll functions $\gf_{\pm}$
such that
the corresponding functionals $\GF_{\pm}$
are increasing, and furthermore
\eqref{VFO},
\eqref{Vf} and \eqref{Efr} hold for $\GF_{\pm}$ and $\gf_{\pm}$ and some $r>2$.
\begin{romenumerate}
  
\item \label{T1na}
Then, as \ntoo,
\begin{align}
  \E \GF(\Tn)&
 =\E \GF(\Tqq n)+o\bigpar{\sqrt{n}}
                \label{EFn}
  \\
  \Var \GF(\Tn)& = b(n)-c(n)^2/n+o\bigpar{n}
\label{VFn0}  
  \\&
  =\Var \GF(\Tqq n)-\Cov\bigpar{\GF(\Tqq n),N_n}^2/n
  +o\bigpar{n}
                  , \label{VFn}
\end{align}
and
\begin{align}\label{t1na}
  \frac{\GF(\Tn)-\E \GF(\Tn)}{\sqrt{n}}
  \approxd N\Bigpar{0,\frac{\Var \GF(\Tn)}n}
    \approxd N\Bigpar{0,\frac{b(n)}n-\frac{c(n)^2}{n^2}},
\end{align}
 with all \absmoment{s} of order $s< r$.

\item \label{T1nb}
Suppose further that
$a(\gl)$ is a function satisfying \eqref{EF},
and that, as \ntoo,
\begin{align}\label{bc}
  b(n)-c(n)^2/n = \Omega\bigpar{n}.
\end{align}
Then, as \ntoo,
\begin{align}\label{t1n}
  \frac{\GF(\Tn)-a(n)}{\sqrt{b(n)-c(n)^2/n}}
  \dto N(0,1)
\end{align}
and, equivalently,
\begin{align}\label{t1nb}
  \frac{\GF(\Tn)-\E \GF(\Tn)}{\sqrt{\Var \GF(\Tn)}}
  \dto N(0,1),
\end{align}
in both cases with convergence of
all \absmoment{s} of order $s< r$.

\end{romenumerate}
\end{theorem}

These theorems are easily extended to multivariate versions.
This can essentially be done by the standard Cram\'er--Wold device,
with a (minor) technical complication because
of the possibility of oscillations in the covariance matrix, and thus no
straightforward limit distribution.
We begin with a multivariate extension of \refT{T1}.
For later convenience, we give two
equivalent versions of this extension, using functions $a_k$ and
$b_{k\ell}$ as discussed above in \refC{Ck} but not in \refT{Tk}.

\begin{theorem}\label{Tk}
  Let $\gf_1,\dots,\gf_K$ be toll functions, for some $K\ge1$,
let $\GF_k$ be the corresponding additive functionals given by \eqref{GF}, 
and
assume that,
as \gltoo,
\eqref{VFO}, \eqref{Vf} and \eqref{Efr} hold for each  $\gf_k$ and some $r>2$.
  Then, as \gltoo,
\begin{align}\label{t1k}
\Bigpar{\frac{\GF_k(\Tgl)-\E\GF_k(\Tgl)}{\sqrt{\gl}}}_{k=1}^K
  \approxd N\bigpar{0,\gS(\gl)},
\end{align}
where the covariance matrix $\gS(\gl)=\bigpar{\gs_{k\ell}(\gl)}_{k,\ell=1}^{K}$
is given by
\begin{align}\label{gskl}
  \gs_{k\ell}(\gl)
 :=\frac{\Cov\bigpar{ \GF_k(\Tgl),\GF_\ell(\Tgl)}}{\gl}.
\end{align}
Furthermore, 
\eqref{t1k} holds with 
all   \absmoments{} of  order $s< r$.
\end{theorem}

\begin{corollary}
  \label{Ck}
  Suppose in addition to the assumptions of \refT{Tk} that
  $a_k(\gl)$ and $b_{k\ell}(\gl)$,
  for $k,\ell=1,\dots,K$,
  are real-valued
functions such that,
as \gltoo,
\eqref{EF}  holds for each $\GF_k$
(with $a_k(\gl)$),
and
\eqref{VF} holds in the form
\begin{align}\label{CovFF}
    \Cov\bigpar{ \GF_k(\Tgl),\GF_\ell(\Tgl)}&=
   b_{k\ell}(\gl)+o(\gl).
\end{align}
Then, as \gltoo,
\begin{align}\label{t1kc}
\Bigpar{\frac{\GF_k(\Tgl)-a_k(\gl)}{\sqrt{\gl}}}_{k=1}^K
  \approxd N\bigpar{0,\gS(\gl)},
\end{align}
where the covariance matrix $\gS(\gl)=\bigpar{\gs_{k\ell}(\gl)}_{k,\ell=1}^{K}$
is given by
\begin{align}\label{gsklc}
  \gs_{k\ell}(\gl):=\frac{b_{k\ell}(\gl)}{\gl}.
\end{align}
Furthermore, \eqref{t1kc} holds with 
all   \absmoments{} of  order $s< r$. 
\end{corollary}

We state also a corresponding multivariate
extension of \refT{T1n} for the  model with fixed $n$.

\begin{theorem}\label{Tknabc}
  Let $\gf_1,\dots,\gf_K$ be toll functions, for some $K\ge1$,
let $\GF_k$ be the corresponding additive functionals given by \eqref{GF}, 
and let 
$a_k(\gl)$, $b_{k\ell}(\gl)$ and $c_k(\gl)$  be real-valued
functions such that 
\eqref{EF} 
and \eqref{CovFN}
hold for each $\GF_k$ 
(with $a_k(\gl)$ and $c_k(\gl)$), and
\eqref{CovFF} holds.

Suppose further that each $\gf_k=\gf_{k+}-\gf_{k-}$ for some toll
functions $\gf_{k\pm}$
such that the corresponding functionals $\GF_{k\pm}$
are increasing, and furthermore
\eqref{VFO},
\eqref{Vf} and \eqref{Efr} hold for $\GF_{k\pm}$ and $\gf_{k\pm}$ and some $r>2$,

Then, as \ntoo,
\begin{align}\label{t1kn}
\Bigpar{\frac{\GF_k(\Tn)-a_k(n)}{\sqrt{n}}}_{k=1}^K
  \approxd N\bigpar{0,\hgS(n)},
\end{align}
where the covariance matrix $\hgS(\gln)=\bigpar{\hgs_{k\ell}(\gln)}_{k,\ell=1}^{K}$
is given by
\begin{align}\label{hgskln}
  \hgs_{k\ell}(\gln)
  :=\frac{b_{k\ell}(\gln)-c_k(\gln)c_\ell(\gln)/\gln}{\gln}
  =\frac{b_{k\ell}(\gln)}{\gln}-\frac{c_k(\gln)}{\gln}\frac{c_\ell(\gln)}{\gln}.
\end{align}
Moreover, \eqref{t1kn} holds with
all  \absmoment{s} of  order $s< r$; in particular,
\begin{align}
 \E \GF_k(\Tn)&=a_k(n)+o\bigpar{\sqrt{n}}, \label{bro1}
  \\
  \Cov\bigpar{\GF_k(\Tn),\GF_\ell(\Tn)}
 &= b_{k\ell}(n)-\frac{c_k(n)c_\ell(n)}{n}
                   +o\bigpar{n}. \label{bro2}
\end{align}
\end{theorem}

\begin{remark}\label{Rrr}
  In
\eqref{t1kn}, we may replace $a_k(n)$ by either
  $\E\GF_k(\Tgn)$ (since we may choose $a_k(n):=\E\GF_k(\Tgn)$),
  or by $\E\GF_k(\Tn)$ (by \eqref{EFn}).
In these cases \eqref{EF} holds automatically and does not have to be
  verified.
\end{remark}

\section{Proofs of general central limit theorems}
\label{Spfcentral}

We first note that if $\gf$ is a functional of tries and either $\gf\ge0$ or
$\E|\gf(\Tgl)|<\infty$, then, since $N_\gl\sim\Po(\gl)$,
and $\gf(\emptyset)=0$,
  \begin{align}\label{l000}
    \E \gf(\Tgl)
    =e^{-\gl}\sumni\frac{\gl^n}{n!}a_n,
  \end{align}
with $a_n:=\E \gf(\Tn)$.

\begin{lemma}\label{L0}
  Let\/ $0\le\gl_1\le\gl_2$.
  \begin{romenumerate}
  \item \label{L0a}
  If\/ $\gf:\stt\to\bbR$ is an arbitrary functional,
  then
  \begin{align}\label{l0a}
    \E |\gf(\Tqq{\gl_1})|
    \le e^{\gl_2}\E |\gf(\Tqq{\gl_2})|.
  \end{align}
\item \label{L0m}
  Moreover, if\/ $m$ is such that\/ $\gf(T)=0$ when $\abse{T}<m$, then
    \begin{align}\label{l0m}
    \E |\gf(\Tqq{\gl_1})|
    \le \Bigparfrac{\gl_1}{\gl_2}^me^{\gl_2}\E |\gf(\Tqq{\gl_2})|.
  \end{align}
\end{romenumerate}
\end{lemma}

\begin{proof}
By \eqref{l000} applied to $|\gf|$,
  \begin{align}\label{l00}
    \E |\gf(\Tgl)|=e^{-\gl}\sumni\frac{\gl^n}{n!}a_n,
  \end{align}
  where $a_n:=\E |\gf(\Tn)|\ge0$,
and both \eqref{l0a} and \eqref{l0m} follow. (The latter because $a_n=0$
  for $n<m$.)
\end{proof}

\begin{lemma}
  \label{L1}
 Let $\gf$ be a toll function and
let $\GF$ be the corresponding additive functional given by \eqref{GF}.
  Let $r\ge1$ and
assume that $\E|\gf(\Tgl)|^r<\infty$ for some
$\gl>0$. Then $\E |\GF(\Tgl)|^r<\infty$.
\end{lemma}


\begin{proof}
  We consider first three special cases.

\pfcase{$\gf(T)=0$ unless $\abse{T}=1$.}\label{C1}
Then,
using \refE{Eleaves},
if $a:=\gf(\Ti)$,
we have $\gf = a\gfo$, 
and 
$\GF(\Tgl)=a\abse{\Tgl} = a N_\gl$. Hence,
$\E|\GF(\Tgl)|^r<\infty$ for every $r<\infty $. 

\pfcase{There exists $m\ge 2$ such that $\gf(T)=0$ unless $\abse{T}=m$.}
\label{C2}
Consider first the random trie $\Tq m$ constructed from $m$ strings
$\Xi\sss1,\dots,\Xi\sss m$.
Note that $\Tq m\eqd (\Tgl\mid N_\gl=m)$ and thus
\begin{align}
  \E|\gf(\Tq m)|^r=\E\bigpar{|\gf(\Tgl)|^r\mid N_{\gl}=m}
\le \P(N_\gl=m)\qw\E|\gf(\Tgl)|^r<\infty.
\end{align}

Let $\bga\in\cAx$ and consider the fringe tree
$\Tq{m}^\bga$. If not all $m$ strings $\Xi\sss j$ have the prefix $\bga$,
then this fringe tree has less than $m$ leaves, and thus, by our assumption,
$\gf(\Tq{m}^\bga)=0$. Furthermore, if we condition on the opposite event,
\ie, that all $m$ strings have prefix $\bga$, then the fringe tree $\Tq{m}^\bga$
has the same distribution as the unconditioned $\Tq m$. Hence,
\begin{align}\label{mib}
  \E |\gf(\Tq{m}^\bga)|^r
  =\P\bigpar{\Xi\succ\bga}^m\E|\gf(\Tq{m})|^r
   =C\P\bigpar{\Xi\succ\bga}^m 
    =C P\xpar{\bga}^m.
\end{align}
Moreover, for every $\ell\ge0$, there exists at most one $\bga\in\cA^\ell$
such that $\gf(\Tq{m}^\bga)\neq0$.
Hence, if we let
\begin{align}
  X_\ell:=\sum_{|\bga|=\ell}\gf(\Tq{m}^\bga),
\end{align}
then, by \eqref{mib} and \eqref{Qkm},
where by \eqref{Qk}
$\gQ m:=\sum_{\ga\in\cA}p_\ga^m<1$,
\begin{align}\label{Qell}
  \E |X_\ell|^r
  =\E\sum_{|\bga|=\ell}|\gf(\Tq{m}^\bga)|^r
  =C\sum_{|\bga|=\ell}P(\ga)^m
  =C \gQ m^\ell.
\end{align}
Thus $\norm{X_\ell}_r\le C \gQ m^{\ell/r}$.
Hence, \eqref{GFbga} and Minkowski's inequality yield
\begin{align}\label{mia}
  \norm{\GF(\Tq{m})}_r
  =\Bignorm{\sum_{\ell\ge0}X_\ell}_r
  \le\sum_{\ell\ge0} C \gQ m^{\ell/r}<\infty.
\end{align}

Now return to the random trie $\Tgl$ in the Poisson model.
Consider the bucket trie with bucket size $m$,
based on the same strings.
As said in \refSS{SSbucket},
the trie $\Tgl$ is obtained from the bucket trie
by letting a small trie grow from each bucket.
By our assumption,
the only non-zero contributions to $\GF(\Tgl)$ in \eqref{GF} then comes
from the small tries grown from the buckets that contain exactly $m$
strings.
Condition on the bucket trie, and let $M$ be the number of buckets with  $m$
strings. Then the small tries grown from them are $M$ independent copies of
$\Tq m$. Hence, if $W_1,W_2,\dots,$ are \iid{} copies of $\GF(\Tq m)$,
we have
\begin{align}
  \bigpar{\GF(\Tgl)\mid M}\eqd \sum_{j=1}^M W_j.
\end{align}
Consequently, by Minkowski's inequality and \eqref{mia},
\begin{align}\label{mic}
  \bignorm{\bigpar{\GF(\Tgl)\mid M}}_r=
  \Bignorm{\sum_{j=1}^M W_j}_r
  \le M \bignorm{W_1}_r
  =M\bignorm{\GF(\Tq m)}_r
  =CM.
\end{align}
Furthermore, since the sets of strings in the buckets are disjoint, $M\le
N_\gl/m\le N_\gl$. Consequently, 
\begin{align}\label{midjan}
  \E|\GF(\Tgl)|^r
  =
  \E  \bigsqpar{  \E\bigpar{|\GF(\Tgl)|^r\mid M}}
  \le \E(CM^r)\le C \E N_\gl^r<\infty .
\end{align}

\pfcase{$\gf(T)=0$ if\/ $\abse{T}\le r$.}\label{C2r}
Then, in particular, $\gf(\Ti)=0$.

Let $m:=\floor{r}+1$.
Then \eqref{tpaxx}
and \refL{L0}\ref{L0m} (applied to $|\gf(T)|^r$)
  yield, for any
  $\bga\in\cAx$,
  \begin{align}
\E|\gf(\Tgl^\bga)|^r
    =    \E |\gf(\Tqq{\gl P(\bga)})|^r
 \le  P(\bga)^{m} e^\gl\E |\gf(\Tgl)|^r.
  \end{align}
Hence, for some $C_\gl<\infty$,
\begin{align}\label{sw}
  \norm{\gf(\Tgl^\bga)}_r \le C_\gl P(\bga)^{m/r}.
\end{align}
Consequently,
\eqref{GFbga}, Minkowski's inequality and \eqref{Qrr} yield,
since $m/r>1$,
\begin{align}\label{swb}
  \norm{\GF(\Tgl)}_r
  \le \sum_{\bga\in\cAx} \norm{\gf(\Tgl^\bga)}_r
\le  \sum_{\bga\in\cAx} C_\gl P(\bga)^{m/r}
  <\infty, 
\end{align}
and thus $\E| \GF(\Tgl)|^r<\infty$.

\pfcase{The general case.}
Decompose
\begin{align}
  \gf=\sum_{1\le j\le \floor{r}}\gf_j+\gf',
\end{align}
where $\gf_j(T):=\gf(T)\ett{\abse{T}=j}$
and $\gf'(T):=\gf(T)\ett{\abse{T}>r}$.
Then \refCase{C1} applies to $\gf_1$, \refCase{C2} to $\gf_j$ for
$2\le j\le \floor{r}$, and \refCase{C2r} to $\gf'$.
Consequently,  the corresponding additive functionals $\GF_j$ and $\GF'$
satisfy $\E |\GF_j(\Tgl)|^r<\infty$ and $\E |\GF'(\Tgl)|^r<\infty$,
and the result follows by 
Minkowski's inequality since
$\GF(\Tgl)=\sum_{j=1}^m\GF_j(\Tgl)+\GF'(\Tgl)$.
\end{proof}

\begin{lemma}
  \label{Lr}
 Let $\gf$ be a toll function and
let $\GF$ be the corresponding additive functional given by \eqref{GF}.
Let $r>2$ and
assume  that, as \gltoo,
\begin{align}
  \Var \GF(\Tgl)&=
O(\gl),\label{VFOr}
  \\
  \E |\gf(\Tgl)-\E \gf(\Tgl)|^{r} &= O\bigpar{\gl^{r/2}}.
 \label{Efrx}
\end{align}
  Then, 
$\E |\GF(\Tgl)|^r<\infty$ for all $\gl\ge0$ and
  \begin{align}\label{lr}
  \E |\GF(\Tgl)-\E \GF(\Tgl)|^r =  O\bigpar{\gl^{r/2}}, \qquad \gl\ge1.
  \end{align}
\end{lemma}

\begin{proof}\CCreset
Note first that in the special case $\gf(T)=\gfo(T):=\indic{T=\Ti}$ in
\refE{Eleaves}, $\GF(\Tgl)=N_\gl\sim\Po(\gl)$, and \eqref{VFOr}--\eqref{lr}
hold; for \eqref{lr}, this is because as \gltoo,
$(N_\gl-\gl)/\gl\qq\dto N(0,1)$ with all absolute moments.
(This follows \eg{} first for integer $\gl$ from \cite[Theorem 7.5.1]{Gut},
and then in  general using Minkowski's inequality.)
Hence, by subtracting a suitable multiple of $\gfo$ from $\gf$, 
and using Minkowski's inequality for each of \eqref{VFOr}--\eqref{lr},
we may
in the remainder of the proof assume that $\gf(\Ti)=0$.
Then also $\GF(\Ti)=0$.
  
By \eqref{toll} and \eqref{tpaxx} (for $\GF$, using $\GF(\Ti)=0$),
we have the decomposition
  \begin{align}\label{qff}
    \GF(\Tgl)=\gf(\Tgl)+\sum_{\ga\in\cA}\GF(\Tglga)
    =\gf(\Tgl)+\sum_{\ga\in\cA}\GF(\Tglgax).
  \end{align}
Define, for $\ga\in\cA$, 
  \begin{align}\label{Xglga}
  \Xglga:=\GF\bigpar{\Tglgax}-\E \GF\bigpar{\Tglgax}.    
  \end{align}
    Then, 
by \eqref{qff},
\begin{align}\label{qx}
  \GF(\Tgl)-\E\GF(\Tgl) =   \gf(\Tgl)-\E\gf(\Tgl) +\suma \Xglga.
\end{align}
In the Poisson model, 
the different modified branches $\Tglgax$, $\ga\in\cA$,
are independent random tries,
and thus the random variables $\Xglga$, $\ga\in\cA$, are independent.
Furthermore, $\E\Xglga=0$ by \eqref{Xglga}. Hence, 
we may apply the version of Rosenthal's inequality in \refL{LPinelis} below,
and conclude that, if we fix any $K>1$ (this will be chosen later), there
exists
$\CC=\CCx(r,K)$ such that
\begin{align}\label{pj}
  \E \Bigabs{\suma \Xglga}^r
\le K\suma   \E |X_{\gl,\ga}|^r
+ \CCx\Bigpar{\suma\E X_{\gl,\ga}^2 }^{r/2}.
\end{align}
Let
\begin{align}\label{qg}
g(\gl):=\norm{\GF(\Tgl)-\E \GF(\Tgl)}_r
=\bigpar{\E|\GF(\Tgl)-\E \GF(\Tgl)|^r}^{1/r}.   
\end{align}
Since  $\GF(\Tglgax)\eqd\GF(\Tqq{p_\ga\gl })$
by \eqref{tpaxx}, recalling that $P(\ga)=p_\ga$ for $\ga\in\cA$,
it follows from \eqref{Xglga} that
\begin{align}
  \label{qgg}
  \E |\Xglga|^r
  =\E\bigabs{\GF(\Tqq{p_\ga \gl}) -\E\GF(\Tqq{p_\ga \gl})}^r
  =g(p_\ga\gl)^r.
\end{align}

By \eqref{VFOr} and \eqref{Efrx},
there exists $\gl_0\ge1$ such
that for all $\gl\ge\gl_0$, and all $\ga\in\cA$,
  \begin{align}\label{kkc}
    \E \Xglga^2=\Var \GF(\Tglgax) 
   =\Var \GF(\Tqq{p_\ga\gl })
                                   &   \le \CCname{\Ca} p_\ga\gl
                                   \le \CCx \gl,
    \\\label{kkd}
    \norm{ \gf(\Tgl)-\E \gf(\Tgl)}_r &\le \CCname{\Cb} \gl\qq.
  \end{align}
  Hence, by
  \eqref{qg}, \eqref{qx},  Minkowski's inequality, \eqref{pj},  \eqref{qgg},
and \eqref{kkc}--\eqref{kkd},  
for $\gl\ge\gl_0$,
  \begin{align}\label{kkb}
    g(\gl)&\le \Bignorm{\suma \Xglga}_r+ \bignorm{ \gf(\Tgl)-\E \gf(\Tgl)}_r
\notag \\&
    \le \Bigpar{K\suma g\bigpar{p_\ga\gl}^r+ \CCname{\Cc} \gl^{r/2}}^{1/r}
         + \Cb \gl\qq
\notag \\&
    \le \Bigpar{K\suma g\bigpar{p_\ga\gl}^r}^{1/r}+ \CCname{\Cd} \gl\qq.
  \end{align}

Let $\gl_1:=\gl_0/\pmin$, and
$\gl_k:=\gl_1/\pmax^{k-1}$ for $k\ge2$.
We show by induction on $k$ that for some large $A<\infty$,
\begin{align}\label{gind}
  g(\gl)\le A \gl\qq,
  \qquad \gl\in[1,\gl_k].
\end{align}
First,
$\E|\gf(\Tqq{\gl_1})|^r<\infty$ by \eqref{kkd},
and thus \refL{L1} yields
$\E|\GF(\Tqq{\gl_1})|^r<\infty$.
Hence, 
by  \refL{L0}\ref{L0a},
\begin{align}\label{gnu}
 g(\gl)\le 2\norm{\GF(\Tgl)}_r\le \CC \norm{\GF(\Tqq{\gl_1})}_r \le\CC,
\qquad \gl\le\gl_1.
\end{align}
Thus 
\eqref{gind} holds in the base case $k=1$ if $A\ge\CCx$.

For the induction step, assume \eqref{gind}.
It suffices to consider $\gl\in(\gl_k,\gl_{k+1}]$, and then $p_\ga\gl\in
[\gl_0,\gl_k]$ for every $\ga\in\cA$. Hence,
\eqref{kkb} and
the induction hypothesis \eqref{gind} yield, recalling \eqref{Qks},
\begin{align}\label{kkf}
  g(\gl)
  &\le \Bigpar{K\suma A^r (p_\ga\gl)^{r/2}}^{1/r} +  \Cd \gl\qq
  \notag  \\&
  = \Bigpar{\Bigpar{K\suma p_\ga^{r/2} }^{1/r} + \Cd A\qw} A\gl\qq
  \notag  \\&
  = \Bigpar{\bigpar{K\gQ {r/2} }^{1/r} + \Cd A\qw} A\gl\qq.
\end{align}
By \eqref{Qk}, $\gQ {r/2}<1$. We now assume that $K$
was chosen such that $1<K<\gQ {r/2}\qw$. Then $K\gQ {r/2}<1$, and we may choose
$A$ so large that
\begin{align}\label{kkg}
\bigpar{K \gQ {r/2} }^{1/r} + \Cd A\qw\le1.
\end{align}
 Then \eqref{kkf} shows that \eqref{gind} holds also for
 $\gl\in(\gl_k,\gl_{k+1}]$, showing the induction step.

We have shown that \eqref{gind} hold for every $k\ge1$, and thus
$g(\gl) \le A \gl\qq$ for all $\gl\ge1$, which by the definition
\eqref{qg} is the same as \eqref{lr}. 
We have also shown in \eqref{gnu} that $g(\gl)=O(1)$ for $\gl\le1$.
Hence,
$\E|\GF(\Tgl)|^r<\infty$ for every $\gl$.
\end{proof}

The proof above used the following version of Rosenthal's inequality.
The standard version of Rosenthal's inequality, see \eg{} 
\cite[Theorem 3.9.1]{Gut}, is \eqref{rosen} with $K=C=C(r)$ (growing with $r$); 
the fact needed here that one can choose $K$ arbitrarily close to 1 (at the
expense of increasing $C$) is due to \citet{Pinelis}, see also
\cite{Pinelis2} for a sharper result.

\begin{lemma}[Rosenthal, Pinelis \cite{Pinelis}]\label{LPinelis}
For every $r>2$ and every $K>1$, there exists a constant $C=C(r,K)$ such that
for any independent random variables $X_1,\dots,X_n$ with means $\E X_i=0$,
\begin{align}\label{rosen}
  \E\Bigabs{\sumin X_i}^r 
\le K\sumin \E|X_i|^r + C \Bigpar{\sumin \E|X_i|^2}^{r/2}.
\end{align}
\end{lemma}

\begin{remark}\label{Rrosen}
 Note that in the special case $r=4$, a simple calculation shows \eqref{rosen}
directly, with $K=1$ and $C=3$. Similarly, when $r$ is any even integer,
\eqref{rosen} (with any $K>1$)
is easily shown by elementary calculations and \Holder's inequality.
(These cases suffice for most applications of our theorems.)
\end{remark}

\begin{proof}
  \citet[Corollary and (4)]{Pinelis} (with $t=r$)
yields the inequality \eqref{rosen} with $K$  replaced by
\begin{align}
  \label{pink}
\sum_{j=0}^{\floor{r/2}-1} c_j(r)\xfrac{a_j^{-2j}}{j!}
=
c_0(r)+\sum_{j=1}^{\floor{r/2}-1} c_j(r)\xfrac{a_j^{-2j}}{j!}
\end{align}
and $C$ given by an
explicit formula (involving $c_j(t)$ and $a_j$) that we ignore;
here 
$a_j>0$ are arbitrary and
$c_j(r)$ are some numbers defined from some other numbers
$p(s)$ and $q(s)$, $s\in[2,r]$, that can be chosen freely under the
conditions
$p(s)\ge1$, $q(s)\ge1$ and, when $s>3$,
\begin{align}\label{pinkod}
p(s)^{-1/(s-3)}+q(s)^{-1/(s-3)}\le1;
\end{align}
in particular, $c_0(r)=q(r)$. (See \cite{Pinelis} for further details.)

Given $K>0$ we can choose first $q(r)$ with $1<q(r)<K$ and then
$p(r)\ge1$ so large that if $r>3$, \eqref{pinkod} holds for $s=r$. 
We choose also, for example, 
$p(s)=q(s)=\max\set{1,2^{s-3}}$ for $s\in[2,r)$.
This defines the numbers $c_j(r)$ for $j=0,\dots,\floor{r/2}-1$ with
$c_0(r)=q(r)<K$, and we then can choose $a_j$, $j\ge1$, so large that
the sum in \eqref{pink} is $\le K$.
\end{proof}

We may now prove \refT{T1}, the general
central limit theorem  for the Poisson model.

\begin{proof}[Proof of \refT{T1}]\CCreset
All limits and asymptotic notions below are as \gltoo.

Let $m\ge1$.
Using the decomposition \eqref{toll} recursively $m$ times on the tree
$\Tgl$,
we obtain,
  \begin{align}
    \GF(\Tgl)
& =\sum_{|\bga|<m}\gf\bigpar{\Tglbga}+\sum_{|\bga|=m}\GF\bigpar{\Tglbga}
                \label{Fix1}\\&
=\sum_{|\bga|<m}\gf\bigpar{\Tglbga}
+\sum_{|\bga|=m}\bigpar{\GF\bigpar{\Tglbga}-\GF\bigpar{\Tglbgax}}
+\sum_{|\bga|=m}\GF\bigpar{\Tglbgax}    
\label{Fix2} \\&
    =:R'_m+R''_m+\sum_{|\bga|=m}\GF\bigpar{\Tglbgax},
\label{Fix3} \end{align}
defining $R'_m$ and $R''_m$ as the first two sums in \eqref{Fix2}.
By \eqref{tpax},
  $\Tglbgax\eqd \Tqq{\gl P(\bga)}$, and thus
  \eqref{Vf} implies
that for every fixed $\bga\in\cAx$,
  \begin{align}\label{ff+}
    \Var \gf(\Tglbgax)
    =    \Var \gf\bigpar{\Tqq{\gl P(\bga)}}
    = o\bigpar{\gl P(\bga)}
        = o(\gl).
  \end{align}
By \eqref{tpax1} and Minkowski's inequality, this implies
  \begin{align}\label{ff}
    \bigpar{\Var \gf(\Tglbga)}\qq
\le \bigpar{\Var \gf(\Tglbgax)}\qq+ O(1)
        = o\bigpar{\gl\qq}.
  \end{align}
Hence,  Minkowski's inequality again yields,
for any fixed $m$,
  \begin{align}\label{frm'}
    \bigpar{\Var R'_m}\qq
    \le \sum_{|\bga|<m}\bigpar{\Var \gf(\Tglbga)}\qq
    = \sum_{|\bga|<m}o\bigpar{\gl\qq}
    =o\bigpar{\gl\qq}.
  \end{align}
Similarly, by \eqref{tpax1} (applied to $\GF$),
$\GF\bigpar{\Tglbga}-\GF\bigpar{\Tglbgax}=O(1)$,
and thus, still for fixed $m$,
$R''_m=O(1)$ and thus $\Var R''_m=O(1)$. Hence, defining $R_m:=R_m'+R''_m$,
\begin{align}\label{VR}
  \Var R_m \le 2\Var R'_m+2\Var R''_m =o(\gl).
\end{align}
Consequently, $R_m$ is negligible, and the major term in \eqref{Fix3} is the
last sum.
We subtract the expectations, and obtain from \eqref{Fix3}
\begin{align}\label{Fixx}
  \GF(\Tgl)-\E \GF(\Tgl)
  =  R_m-\E R_m +
\sum_{|\bga|=m} \Xglbga,
  \end{align}
  where
  \begin{align}\label{Xglbga} 
    \Xglbga:={\GF\bigpar{\Tglbgax}-\E \GF\bigpar{\Tglbgax}}.
  \end{align}
\refL{Lr} applies, since \eqref{VFOr} and \eqref{Efrx} are
our assumptions \eqref{VFO} and \eqref{Efr}; thus, 
for $\gl\ge1$,
\begin{align}\label{gro}
  \E\bigabs{\GF(\Tgl)-\E \GF(\Tgl)}^r
\le \CCname{\CCa} {\gl^{r/2}}
.\end{align}
Hence, using again \eqref{tpax},
for any  $m\ge0$
and all $\bga\in\cA^m$,
at least for $\gl\ge \pmin^{-m}$,
\begin{align}\label{gruk}
  \E \abs{\Xglbga}^r
&=\E\bigabs{\GF\bigpar{\Tqq{\gl P(\bga)}}-\E \GF\bigpar{\Tqq{\gl P(\bga)}}}^r
\le \CCx(\gl P(\bga))^{r/2}
,
\end{align}
where the constant ${\CCa}$  does not depend on $m$.

The random modified fringe trees $\Tglbgax$
for $|\bga|=m$ are independent;
hence the random
variables $\Xglbga$ in \eqref{Fixx} are independent.
Furthermore,
by the definition \eqref{Xglbga}, $\E \Xglbga=0$.
Moreover,
\eqref{gruk} and \eqref{Qkm} imply that for  $\gl\ge \pmin^{-m}$, 
\begin{align}\label{kk}
  \sum_{|\bga|=m}\frac{\E| \Xglbga|^r}{\gl^{r/2}}
  &\le \sum_{|\bga|=m}\CCx P(\bga)^{r/2}
= \CCx \gQ {r/2}^m.
\end{align}

We have so far kept $m$ fixed, and shown that
\eqref{kk} holds for large $\gl$, and also that \eqref{VR} holds, and thus,
for example, for large $\gl$,
  \begin{align}\label{frmm}
{\Var R_m}
\le
   \frac{1}{m}{\gl}.
  \end{align}
In other words, there exist $\gl(m)<\infty$ such that 
\eqref{kk} and \eqref{frmm} hold for  $\gl\ge\gl(m)$.
We may also assume $\gl(m+1)\ge\gl(m)+1$.
Now define (for large $\gl$)  $m(\gl):=\max\set{m:\gl(m)\le\gl}$,
and take in the remainder of the proof $m:=m(\gl)$.
Then, $m=m(\gl)\to\infty$ as \gltoo,
and by definition, \eqref{kk} and \eqref{frmm}
hold with $m=m(\gl)$.
Since $m\to\infty$ and $\gQ {r/2}<1$, see \eqref{Qk}, we have
$\gQ {r/2}^m\to0$, and thus \eqref{kk} shows that 
\begin{align}\label{ll}
  \sum_{|\bga|=m}\frac{\E| \Xglbga|^r}{\gl^{r/2}}
\to0.
\end{align}
Furthermore, \eqref{frmm} implies
\begin{align}
  \label{Rmo}
\Var R_m=o(\gl).
\end{align}

By the subsequence principle in \refR{Rsubsub}, 
it suffices to show that for any given sequence
$\gl_n\to\infty$, the results hold for some subsequence.
By \eqref{VFO}, $\Var(\GF(\Tgl))/\gl=O(1)$, and thus we may, by selecting a
subsequence $(\gl'_n)$ of the given sequence $(\gl_n)$,
assume that $\Var(\GF(\Tgl))/\gl\to\gam$ for some $\gam\ge0$,
and thus also, by \eqref{VF}, $b(\gl)/\gl\to\gam$.

If $\gam=0$, \ie, $\Var\GF(\Tgl)=o(\gl)$ along the subsequence, 
then the \lhs{} of \eqref{t1bx}
tends to 0 in probability, and \eqref{t1bx} holds trivially (along the
subsequence). The same holds for \eqref{t1x} by \eqref{EF}--\eqref{VF}.
(Note also that $\gam=0$ is impossible in \ref{T1b}
since we there assume \eqref{bB}.) 

Now suppose $\gam>0$, and consider only the selected subsequence $(\gl_n')$.
Then \eqref{bB} holds,
and thus \eqref{Rmo} yields
\begin{align}\label{Rmb}
\Var R_m = o(\gl)=o\bigpar{b(\gl)}.  
\end{align}
Hence,
\eqref{VF} and Minkowski's inequality yield
\begin{align}\label{brb}
  \Var\bigpar{\GF(\Tgl)-R_m}
  =b(\gl)+o(\gl)
    =b(\gl)+o\bigpar{b(\gl)}.
\end{align}
 Now use the decomposition \eqref{Fixx} with  $m=m(\gl)$.
Note that
\eqref{Fixx} and \eqref{brb} yield
\begin{align}\label{ii}
  \sum_{|\bga|=m}\Var\bigpar{ \Xglbga}
  =
  \Var\Bigpar{  \sum_{|\bga|=m} \Xglbga}
  =\Var\bigpar{\GF(\Tgl)-R_m}
  \sim b(\gl).
\end{align}
Hence, 
the central limit theorem (see \eg{} \cite[Theorem 7.2.4]{Gut} or
\cite[Theorem 5.12]{Kallenberg}) applies to the sum
$\sum_{|\bga|=m}\Xglbga/b(\gl)\qq$, with Lyapounov's condition
(as in \cite[Theorem 7.2.2]{Gut})
verified by \eqref{ll} and \eqref{bB}.
Consequently,
\begin{align}\label{fy}
  b(\gl)\qqw\sum_{|\bga|=m}\Xglbga \dto N(0,1).
\end{align}
Furthermore,
\eqref{Rmb} implies $b(\gl)\qqw (R_m-\E R_m)\pto0$.
 Thus 
\eqref{Fixx}, 
\eqref{fy} and the
Cram{\'e}r--Slutsky theorem \cite[Theorem 5.11.4]{Gut}
yield
\begin{align}
 \frac{\GF(\Tgl)-\E \GF(\Tgl)}{ b(\gl)\qq} \dto N(0,1).
\end{align}
The conclusions \eqref{t1} and \eqref{t1b} (along the subsequence)
now follow using
\eqref{EF},  \eqref{VF}, and \eqref{bB}.
Moreover, by multiplying \eqref{t1}
by $\sqrt{b(\gl)/\gl}\to\gam\qq$, it follows that the \lhs{} of \eqref{t1x}
converges in distribution to $N(0,\gam)$, which yields \eqref{t1x},
see \refR{Rapprox}. Similarly, \eqref{t1bx} holds.

Combining the two cases above,
we have shown, 
for any $\gam\ge0$,
that
\eqref{t1x}--\eqref{t1bx} and
\eqref{t1}--\eqref{t1b} (assuming \eqref{bB}) hold along the subsequence
$(\gl'_n)$,
Since we started with an arbitrary subsequence $(\gl_n)$,
they hold for arbitrary $\gltoo$, see \refR{Rsubsub}.
This proves 
\eqref{t1x}--\eqref{t1bx} and
\eqref{t1}--\eqref{t1b}. 

It remains only to show that these hold with moments as stated.
If \eqref{bB} holds, then \eqref{gro} implies,
recalling also \eqref{EF} and \eqref{VF}, that
the $r$th absolute moments of the \lhs{s} of \eqref{t1} and \eqref{t1b}
are bounded as \gltoo,
which as is well-known implies that every
power of lower order is uniformly integrable, and thus
every \absmoment{} 
of lower order converges to the corresponding moment of $N(0,1)$.
(See \eg{} \cite[Theorems 5.4.2 and 5.5.9]{Gut}.)

Similarly,
if we write 
\eqref{t1x} or \eqref{t1bx} as $X_\gl\approx Y_\gl$, 
then \eqref{gro} implies that 
$\E|X_\gl|^r=O(1)$ for $\gl\ge1$,
and thus if $0<s<r$, then the variables  
$|X_\gl|^s$, $\gl\ge1$, are \ui.
The same holds for $|Y_\gl|^s$ (at least for large $\gl$), 
since $Y_\gl$ is normal with $\Var Y_\gl=O(1)$ as \gltoo.
Hence, 
\refL{Lapprox0} applies to $X_{\gl_n}$ and $Y_{\gl_n}$
for any sequence $\gl_n\to\infty$,
and it follows that $X_{\gl}\approxd Y_{\gl}$
with \absmoments{} of order $s$.
\end{proof}

We next prove the multivariate extensions of \refT{T1}.

\begin{proof}[Proof of \refT{Tk}]
 By the subsequence principle in \refR{Rsubsub},
it suffices to show that for any given sequence
$\gl_n\to\infty$, the result holds for some subsequence.

The \CSineq{} and \eqref{VFO} (for $\GF_k,\GF_\ell$) yield, as \gltoo,
\begin{align}
  \Cov\bigpar{ \GF_k(\Tgl),\GF_\ell(\Tgl)}
  =O(\gl).
\end{align}
By selecting a suitable subsequence $(\gl'_n)$ of the given sequence $(\gl_n)$,
we may thus assume that
\begin{align}\label{beta}
  \gs_{k\ell}(\gl):=
  \frac{\Cov\bigpar{\GF_k(\Tgl),\GF_\ell(\Tgl)}}{\gl}
  \to \gb_{k\ell}
\end{align}
as \gltoo{} along the subsequence, for 
all $k,\ell$ and some real $\gb_{k\ell}$.

Let $(t_1,\dots,t_K)$ be an arbitrary vector in $\bbR^K$ and
consider the linear combination
\begin{align}
  \label{YGF}
  \GF:=\sumkK t_k  \GF_k.
\end{align}
This is an additive functional with toll function
$  \gf:=\sumkK t_k  \gf_k$.
Then \eqref{VFO}, \eqref{Vf} and \eqref{Efr} hold by the assumptions and
Minkowski's inequality.
Let $a(\gl):=\E\GF(\Tgl)$ and $b(\gl):=\Var\GF(\Tgl)$, so \eqref{EF} and
\eqref{VF} hold trivially.
Thus \refT{T1} applies, and \eqref{t1bx} holds.

Furthermore, \eqref{YGF} yields 
\begin{align}\label{qb}
\Var {\GF(\Tgl)}
  = \sumklK t_kt_\ell
\Cov\bigpar{\GF_k(\Tgl),\GF_\ell(\Tgl)}
.\end{align}
Hence, \eqref{beta} implies that,
along the subsequence $(\gl'_n)$,
we have
$\Var {\GF(\Tgl)}/\gl\to \sum_{k,\ell} t_kt_\ell \gb_{k\ell}$,
and thus \eqref{t1bx} implies that
\begin{align}\label{t1knti}
  \sumkK t_k{\frac{\GF_k(\Tgl)-\E\GF_k(\Tgl)}{\sqrt{\gl}}}
  \dto N\Bigpar{0,\sum_{k,\ell} t_kt_\ell \gb_{k\ell}}.
\end{align}
Since the vector $(t_1,\dots,t_K)$ is arbitrary,
it follows by the Cram\'er--Wold device that,
along the subsequence,
\begin{align}\label{ystad}
  \Bigparfrac{\GF_k(\Tgl)-\E \GF_k(\Tgl)}{\sqrt{\gl}}_1^K
  \dto N\bigpar{0,(\gb_{k\ell})_{k,\ell=1}^K}.
\end{align}
Combined with \eqref{beta}, this shows that
\eqref{t1k} holds along the subsequence, see \refR{Rapprox}.
Since we started with an arbitrary subsequence $(\gl_n)$,
the subsequence principle shows that \eqref{t1k} holds in general,
see \refR{Rsubsub}.

Finally,
if we write \eqref{t1k} as $X_\gl\approxd Y_\gl$, then, as in the proof of
\refT{T1}, \refL{Lr} implies that $\E|X_\gl|^r=O(1)$ for $\gl\ge1$, and  
\refL{Lapprox0} shows that \eqref{t1k} holds with \absmoment{s} of order $s$ 
for $s<r$.
\end{proof}

\begin{proof}[Proof of \refC{Ck}]
Consider again a subsequence where \eqref{beta} holds for some
$\gb_{k\ell}$.
By \eqref{CovFF}, we also have
\begin{align}\label{bbeta}
  \frac{b_{k\ell}(\gl)}{\gl}
  \to \gb_{k\ell}.
\end{align}
Hence the proof of \refT{Tk} just given shows that \eqref{t1k} holds also
with $\gS(\gl)$ defined by \eqref{gsklc} instead of \eqref{gskl}.

Furthermore, \eqref{EF} implies that 
we may replace $\E \GF_k(\gl)$ by $a_k(\gl)$ in \eqref{ystad},
and thus in \eqref{t1k}. The result \eqref{t1kc} follows.
Finally, the same argument as in the proofs of \refTs{T1} and \ref{Tk} shows
that \eqref{t1kc} holds with \absmoment{s} of order $s<r$.
\end{proof}

We turn to proofs of the theorems for the model $\Tn$ with a given number
of leaves.
This time we begin with the multivariate version.

\begin{proof}[Proof of \refT{Tknabc}]
  We consider the toll functions $\gf_{k\pm}$, and also the toll function
$\gfo(T):=\indic{T=\Ti}$ in \refE{Eleaves}; recall that
$\GFo(\Tgl)=N_\gl$ by \eqref{gF00}.
  Note that \eqref{Vf} and \eqref{Efr} are trivial for $\gfo$, since
  $\gfo(T)=O(1)$, and that \eqref{VFO} holds for $\GFo$ because
  \begin{align}\label{greta}
  \Var\GFo(\Tgl) = \Var N_\gl =\gl .
  \end{align}
Hence, \refT{Tk} applies to the set of toll functions
$\set{\gfo,\gf_{k\pm}}$. 
(We  use $\Ti$ and $k\pm$ (for $k=1,\dots,K$) as
indices instead of \set{1,\dots,2K+1}.)

Consider $\bbR^{2K}$ with the usual coordinate-wise partial order, \ie,
$(x_i)_1^{2K}\le (y_i)_1^{2K}$ if $x_i\le y_i$ for every $i$.
Since each $\GF_{k\pm}$ by assumption is an 
increasing functional,
and $\Tq{n+1}$ is obtain by adding a new string to $\Tq{n}$,
it follows that if $n_1\le n_2$, then
\begin{align}
\bigpar{\GF_{k\pm}(\Tq{n_1})}_{k\pm} \le \bigpar{\GF_{k\pm}(\Tq{n_2})}_{k\pm}
\qquad\text{in $\bbR^{2K}$}.
\end{align}
Furthermore,
by the construction of the random trie $\Tqq{\gl}$,
if we condition on $N_\gl=n$, then we recover $\Tn$ (in distribution), \ie,
$\bigpar{\Tqq{\gl}\mid N_\gl=n}\eqd \Tn$.
It follows that the random vector $(\GF_{k\pm}(\Tgl))_{k\pm}$
is stochastically increasing in $\GFo(\Tgl)=N_\gl$ in the sense that
for any $\bfx\in\bbR^{2K}$ and $n_1\le n_2$, 
\begin{multline}\label{selma}
  \P\Bigpar{(\GF_{k\pm}(\Tgl))_{k\pm}\le\bfx \bigm| \GFo(\Tgl)=n_1}
  =
    \P\bigpar{(\GF_{k\pm}(\Tq{n_1}))_{k\pm}\le\bfx}
    \\
  \ge
    \P\bigpar{(\GF_{k\pm}(\Tq{n_2}))_{k\pm}\le\bfx}
=  \P\Bigpar{(\GF_{k\pm}(\Tgl))_{k\pm}\le\bfx \bigm| \GFo(\Tgl)=n_2}.
  \end{multline}

Consider now the sequence $\gl_n=n$, and take an arbitrary subsequence
$(n_j)$ such that,
for the set of functionals $\set{\GFo,\GF_{k\pm}}$,  
the covariances converge as in \eqref{beta}, and thus
\eqref{ystad} holds by the proof of \refT{Tk} above.
Note that then, by \eqref{beta} and \eqref{greta},
\begin{align}\label{kx}
  \gb_{\Ti,\Ti}=\lim_\jtoo\gs_{\Ti,\Ti}({n_j})
  =\lim_\jtoo\frac{\Var \GFo(\Tqq {n_j})}{{n_j}}
  =\lim_\jtoo\frac{\Var N_{n_j} }{{n_j}}=1,
\end{align}
and, similarly,
$\E \GFo(\Tqq {n})=\E N_{n}=n$. 
We may now apply a theorem by
\citet[Theorem~1]{Nerman}, or (slightly more conveniently)
its corollary \cite[Theorem~2.3]{SJ190},
which allows us to  condition on $\GFo(\Tgl)=n$ in \eqref{ystad}
(under the stochastic monotonicity \eqref{selma} just shown).
Consequently, we obtain that, along the subsequence $(n_j)$,
\begin{align}\label{ystad+}
  \Bigparfrac{\GF_{k\pm}(\Tn)-\E \GF_{k\pm}(\Tqq n)}{\sqrt{n}}_{k\pm}
  &\eqd
\Bigpar{\Bigparfrac{\GF_{k\pm}(\Tqq n)-\E \GF_{k\pm}(\Tqq n)}
    {\sqrt{n}}_{k\pm}  \Bigm| \GFo(\Tqq n)=n}
\notag  \\&
  \dto N\bigpar{\Ti,(\xgb_{k\pm,\ell\pm})},
\end{align}
where, for $\eta_1,\eta_2\in\set{+,-}$, recalling \eqref{kx},
\begin{align}\label{xgb}
  \xgb_{k\eta_1,\ell\eta_2}
  :=
  \gb_{k\eta_1,\ell\eta_2}-\frac{\gb_{k\eta_1,\Ti}\gb_{\ell\eta_2,\Ti}}{\gb_{\Ti,\Ti}}
=  \gb_{k\eta_1,\ell\eta_2}-{\gb_{k\eta_1,\Ti}\gb_{\ell\eta_2,\Ti}}.
\end{align}
Note that in \eqref{ystad+}  we normalize $\GF_{k\pm}(\Tq n)$
using $\E\GF_{k\pm}(\Tqq n)$ for the Poisson model.

Since $\GF_k=\GF_{k+}-\GF_{k-}$, it follows from \eqref{ystad+} that,
along the subsequence,
\begin{align}\label{kivik}
  \Bigparfrac{\GF_{k}(\Tn)-\E \GF_{k}(\Tqq n)}{\sqrt{n}}_{k=1}^K
  \dto N\bigpar{0,(\xgb_{k\ell})_{k,\ell=1}^K},
\end{align}
where, using \eqref{xgb},
\begin{align}\label{hjo}
  \xgb_{k\ell}&=\xgb_{k+,\ell+}-\xgb_{k+,\ell-}-\xgb_{k-,\ell+}+\xgb_{k-,\ell-}
                \notag\\
&  =\gb_{k+,\ell+}-\gb_{k+,\ell-}-\gb_{k-,\ell+}+\gb_{k-,\ell-}
  -  (\gb_{k+,\Ti}-\gb_{k-,\Ti})(\gb_{\ell+,\Ti}-\gb_{\ell-,\Ti}).
\end{align}
We are considering a subsequence such that \eqref{beta} holds 
for the functionals $\set{\GFo,\GF_{\pm}}$
along the subsequence.
It follows from \eqref{hgskln}, \eqref{CovFF}, \eqref{CovFN}, \eqref{beta}
(for the set $\set{\GFo,\GF_{\pm}}$), 
$\GFo(\Tgn)=N_n$,
linearity and \eqref{hjo},
that,
along the subsequence,
\begin{align}\label{malby}
  \hgs_{k\ell}(n)
&  =
\frac{\Cov\bigpar{ \GF_k(\Tgn),\GF_\ell(\Tgn)}}{n}
  -\frac{\Cov\bigpar{ \GF_k(\Tgn),N_n}}{n}
                    \frac{\Cov\bigpar{ \GF_\ell(\Tgn),N_n}}{n}
+o(1)
                    \notag\\&
  \to\xgb_{k\ell}.
\end{align}
By \eqref{EF}, we may replace $\E\GF_k(\Tqq n)$ by $a_k(n)$ in \eqref{kivik},
and thus \eqref{malby} shows that \eqref{t1kn} holds along the subsequence.
Hence, \eqref{t1kn} holds in general by the subsequence principle.

Furthermore, the proof of \refT{Tk} shows also that,
along the subsequence $(n_j)$ above,
\eqref{ystad} holds with absolute moments of order $s<r$.
By \cite[Section 4]{Nerman} (see also \cite[Theorem 2.6]{SJ190}),
the same holds after conditioning on $\GFo(\Tgn)=n$, \ie, in
\eqref{ystad+}.
Since absolute moment convergence here is equivalent to 
uniform $s$th power integrability 
\cite[Theorem 5.5.9]{Gut},
it follows that
also \eqref{kivik} holds with 
uniform $s$th power integrability.
We may again replace $\E\GF_k(\Tqq n)$ by $a_k(n)$, using
\eqref{EF}.
Hence, \eqref{t1kn} holds 
along the subsequence 
with 
uniform $s$th power integrability,
and thus with convergence of  $s$th \absmoment{s}.
Hence, 
by the subsequence principle again,
\eqref{t1kn} holds with  $s$th \absmoment{s}.

In particular, \eqref{t1kn} holds with moments of order 1 and 2, which
gives \eqref{bro1} and \eqref{bro2}.
\end{proof}

\begin{proof}[Proof of \refT{T1n}]
This is essentially the special case $K=1$ of \refT{Tknabc}.
In  part \ref{T1na}, we
do not assume any function $a(\gl)$. However, we may then define
$a(\gl):=\E\GF(\Tgl)$, so \eqref{EF} holds trivially.
Thus we may throughout the proof assume
that we have a function $a(\gl)$ such that \eqref{EF} holds.
Then \refT{Tknabc} applies with $K=1$.
In particular, \eqref{bro1}--\eqref{bro2} hold, which
yields \eqref{EFn}--\eqref{VFn0} using
choice $a(\gl)=\Var\GF(\Tgl)$ just made for \ref{T1na};
then \eqref{VFn} follows by
 \eqref{VF} and \eqref{CovFN},
noting that \eqref{VFO} implies 
$ \Cov\bigpar{ \GF(\Tgl),N_\gl} = O(\gl)$ by the \CSineq.
The approximations \eqref{t1na} follow from \eqref{t1kn} and 
\eqref{EFn}--\eqref{VFn0},
with \absmoment{s} of order $s<r$.

For part \ref{T1nb},
we have by \eqref{t1kn} (or \eqref{t1na}),
\begin{align}\label{daw}
  \frac{\GF(\Tn)-a(n)}{\sqrt{n}}\approxd N\bigpar{0,\hgss(n)},
\end{align}
with
\begin{align}\label{hgss}
  \hgss(n):=\frac{b(n)-c(n)^2/n}{n}.
\end{align}
The assumptions
\eqref{VFO}, \eqref{VF} and
 \eqref{bc} imply $\hgss(n)=\Theta(1)$.
Hence, \eqref{daw} implies that for any subsequence such that $\hgss(n)$
converges, say $\hgss(n)\to\gam$, we have $\gam>0$, and then \eqref{daw}
implies 
\begin{align}\label{Vaa}
  \frac{\GF(\Tn)-a(n)}{\sqrt{b(n)-c(n)^2/n}}
  =
    \frac{\GF(\Tn)-a(n)}{\sqrt{n\hgss(n)}}\dto N\xpar{0,1}
\end{align}
along the subsequence.
By the subsequence principle, \eqref{Vaa} holds in general, which is
\eqref{t1n}.
This yields also \eqref{t1nb},
using \eqref{EFn}--\eqref{VFn0}, \eqref{EF}, and again \eqref{bc}.
Moment convergence follows by the same argument.
\end{proof}

\section{Proof of \refT{TEVC}}\label{Smean}

Before proving \refT{TEVC}, we give some lemmas. To begin with, we assume
that $\gf(\Ti)=0$.

\begin{lemma}\label{LM1}
  Suppose that $\gf(\Ti)=0$ and $\E|\gf(\Tgl)|<\infty$ for some $\gl>0$.
  Then
  \begin{align}\label{lm1}
    \E \GF(\Tgl) 
=  \sum_{\bga\in\cAx}\E{\gf(\Tgl^\bga)}
= \sum_{\bga\in\cAx}\E\gf(\Tqq\glbga),
  \end{align}
where the
sums have finite summands and converge absolutely.
Moreover,
  \begin{align}\label{lm0}
\sum_{\bga\in\cAx}\E\bigabs{\gf(\Tgl^\bga)}
=
\sum_{\bga\in\cAx}\E\bigabs{\gf(\Tqq\glbga)}
<\infty.
  \end{align}
\end{lemma}
\begin{proof}
  By \refL{L0}\ref{L0m}, with $m=2$,
  \begin{align}\label{P2}
    \E|\gf(\Tqq\glbga)|
    \le P(\bga)^2 e^{\gl}\E|\gf(\Tqq\gl)|
    = C_\gl P(\bga)^2.   
  \end{align}
Hence, using \eqref{Qrr},
\begin{align}\label{55}
  \sum_{\bga\in\cAx}   \E|\gf(\Tqq\glbga)|
  \le C_\gl  \sum_{\bga\in\cAx}   P(\bga)^2
<\infty,
\end{align}
which proves the inequality in \eqref{lm0}.
The equality in \eqref{lm0} follows from \eqref{tpaxx}.

Finally, the first equality in \eqref{lm1} follows by \eqref{GFbga} and
Fubini's theorem, using \eqref{lm0} which also implies
 absolute convergence of the sum.
The second equality follows by \eqref{tpaxx}.
\end{proof}

For the variance, we give in the next lemma several different formulas.

\begin{lemma}\label{LV}
  Suppose that $\gf(\Ti)=0$ and $\E|\gf(\Tgl)|^2<\infty$ for some $\gl>0$.
  Then
  \begin{align}
    \Var \GF(\Tgl)
&=
\sum_{\bga,\bgb\in\cAx}\Cov\bigpar{\gf(\Tglbga),\gf(\Tglbgb)}
\label{lv1}    \\&
=\sum_{\bga,\bgb\in\cAx}\E\bigsqpar{\gf(\Tglbga)\gf(\Tglbgb)}
- \Bigpar{\sum_{\bga\in\cAx}\E\gf(\Tglbga)}^2
\label{lv2}    \\&
    =2\sumaa\Cov\bigpar{\gf(\Tglbga),\GF(\Tglbga)}-\sumaa\Var\gf(\Tglbga) 
\label{lv3}    \\&
    =\sumaa
    \Bigpar{2\Cov\bigpar{\gf(\Tqq\glbga),\GF(\Tqq\glbga)}-\Var\gf(\Tqq\glbga)}
\label{lv4}  \end{align}
and
\begin{align}
  \Cov\bigpar{\GF(\Tgl),N_\gl}
  &= \sumaa \Cov\bigpar{\gf(\Tqq{\glbga}),N_{\glbga}}
\label{lc1}  \\
    &= \sumaa \E\bigsqpar{\gf(\Tqq{\glbga})N_{\glbga}}
      -   \gl \sumaa P(\bga)\E\gf(\Tqq{\glbga}),
      \label{lc2}
\end{align}
where all sums 
have finite summands and converge absolutely.
\end{lemma}

\begin{remark}\label{RLV}
 Analoguous formulas for the covariance $\Cov\bigpar{\GF_1(\Tgl),\GF_2(\Tgl)}$
 for two toll functions $\gf_1$ and $\gf_2$ with $\gf_1(\Ti)=\gf_2(\Ti)=0$
 follow immediately by polarization in \eqref{lv1}--\eqref{lv4}; 
the details are omitted.

 However, note that these formulas for variance and covariance do not
 include the case $\gfo(T):=\indic{T=\Ti}$ with $\GFo(T)=\abse{T}$, see
 \refE{Eleaves}.
 (It is easily checked that \eg{} \eqref{lv4} and \eqref{lc1}
fail for $\gfo$.) Hence, separate formulas are given in \refL{LV} for
the covariance with $\GFo(\Tgl)=N_\gl$. 
\end{remark}

\begin{proof}[Proof of \refL{LV}]
Let, recalling \eqref{tpaxx},
  \begin{align}\label{Xbga}
    \Xbga:=\gf(\Tglbga)=\gf(\Tglbgax)
    \eqd \gf(\Tqq\glbga),
  \end{align}
and define, for $k\ge0$,
  \begin{align}\label{Yl}
    Y_k:=\sum_{|\bga|=k} \Xbga.      
  \end{align}
By \eqref{Xbga} and \refL{L0}\ref{L0m}, applied to $\gf^2$ and  with $m=2$,
\cf{} \eqref{P2},
  \begin{align}\label{bill}
    \E \Xbga^2
    =  \E|\gf(\Tqq\glbga)|^2 \le C_\gl P(\bga)^2.
  \end{align}
  Furthermore, for any $k\ge0$,
  the random variables $\Xbga$ with $|\bga|=k$ are
  independent.
Hence, using \eqref{bill} and \eqref{Qkm}, 
  \begin{align}\label{bull}
 \Var Y_k=\sum_{|\bga|=k}\Var \Xbga
    \le   \sum_{|\bga|=k}\E \Xbga^2
    \le    \sum_{|\bga|=k}C_\gl P(\bga)^2
    =C_\gl \gQ 2^k.
  \end{align}
Since $\gQ 2<1$ by \eqref{Qk}, it follows from \eqref{GFbga}, \eqref{Xbga}
and \eqref{Yl}  
that
\begin{align}\label{mans}
  \GF(\Tgl)-\E\GF(\Tgl)=\sumbga\bigpar{\Xbga-\E \Xbga}
  =\sumk\bigpar{Y_k-\E Y_k},
\end{align}
where the sums converge absolutely in $L^1$ 
since $\sumaa\norm{X_\bga}_1<\infty$ 
by \eqref{lm0},
and the final  sum converges absolutely in $L^2$ by \eqref{bull},
\ie,
$  \sumk\norm{Y_k-\E Y_k}_2<\infty$.
(The first sum does not always
converge absolutely in $L^2$; this is why we introduce $Y_k$.)
Hence,
\begin{align}\label{moa}
  \Var \GF(\Tgl)
  =
  \sum_{k,\ell\ge0} \Cov(Y_k,Y_\ell)
  <\infty
\end{align}
with absolute convergence.

Suppose temporarily that $\gf\ge0$. Then 
$\sumk\E Y_k =\sumaa \E \Xbga<\infty$
by \eqref{Yl}, \eqref{Xbga}, and \eqref{lm0}, and  thus \eqref{moa} yields
\begin{align}\label{mob}
  \sum_{\bga,\bgb\in\cAx}\E\bigsqpar{\Xbga\Xbgb}
  =\sum_{k,\ell}\E\bigsqpar{Y_kY_\ell}
    =\sum_{k,\ell}\bigpar{\Cov(Y_k,Y_\ell)+\E Y_k\E Y_\ell}
<\infty.
\end{align}
Returning to a general $\gf$, we apply \eqref{mob} to $|\gf|$, and find
\begin{align}
  \label{moc}
    \sum_{\bga,\bgb\in\cAx}\E\bigabs{\Xbga\Xbgb}<\infty.
\end{align}
We see from \eqref{moc} and \eqref{lm0} that the sums in \eqref{lv2} are
absolutely convergent. It follows that so is the sum in \eqref{lv1}, and that
it equals \eqref{lv2}; furthermore, recalling \eqref{Xbga} and \eqref{Yl},
this sum equals
$\sum_{k,\ell}\Cov(Y_k,Y_\ell)$. Hence \eqref{moa} implies \eqref{lv1} and
\eqref{lv2}. 

Next, rewrite \eqref{moa} as
\begin{align}\label{moab}
  \Var\GF(\Tgl) = 2\sum_{0\le k\le \ell}  \Cov(Y_k,Y_\ell)-\sumk\Var Y_k.
\end{align}
Let $k\le \ell$. If $\bga\in\cA^k$ and $\bgb\in\cA^\ell$ and $\bga$ is not a
prefix of $\bgb$, then $X_\bga$ and $X_\bgb$ are independent.
Thus,
\begin{align}\label{moac}
  \Cov(Y_k,Y_\ell)
  =\sum_{\bga\in\cA^k,\,\bgb\in\cA^\ell}\Cov(X_\bga,X_\bgb)
  =\sum_{\bga\in\cA^k,\,\bgg\in\cA^{\ell-k}}\Cov(X_\bga,X_{\bga\bgg}).
\end{align}
Hence, for any $k\ge0$, recalling \eqref{Xbga} and absolute convergence
in \eqref{lv1},
\begin{align}\label{moad}
\sum_{\ell\ge k}  \Cov(Y_k,Y_\ell)
  =\sum_{\bga\in\cA^k,\,\bgg\in\cAx}\Cov(X_\bga,X_{\bga\bgg}),
\end{align}
with absolute convergence, also when summed over $k$.
Furthermore, by \eqref{mans} applied to $\Tglbga$,
\begin{align}\label{mansan}
  \GF(\Tglbga)-\E\GF(\Tglbga)
  =\sumj \sum_{|\bgg|=j}\bigpar{X_{\bga\bgg}-\E X_{\bga\bgg}}
\end{align}
with the sum over $j$ converging in $L^2$. Thus, for each $\bga$,
\begin{align}
  \Cov(\gf(\Tglbga),\GF(\Tglbga))
  =\sumj \sum_{|\bgg|=j}\Cov(X_\bga,X_{\bga\bgg})
  =
  \sum_{\bgg\in\cAx}\Cov(X_\bga,X_{\bga\bgg}).
\end{align}
Hence, \eqref{moad} yields,
with absolute convergence, also when summed over $k$,
\begin{align}
\sum_{\ell\ge k}  \Cov(Y_k,Y_\ell)
    =\sum_{\bga\in\cA^k}\Cov(\gf(\Tglbga),\GF(\Tglbga)).
\end{align}
Consequently, \eqref{lv3} follows from \eqref{moab} and \eqref{bull}.
Finally, \eqref{lv4} follows from \eqref{lv3} by \eqref{tpax}.

For the covariance, let
$\psi(T):=\gf(T)\abse{T}$, so that
\begin{align}\label{mmix}
\E\psi(\Tgl)
  =\E\bigsqpar{\gf(\Tgl)\abse{\Tgl}}
=\E\bigsqpar{\gf(\Tgl)N_\gl}.
\end{align}
Since $\E|\gf(\Tgl)|^2<\infty$ and $\E N_\gl^2<\infty$, the \CSineq{} implies
that $\E|\psi(\Tgl)|<\infty$.
Hence \refL{L1} applies to both $\gf$ and $\psi$, which shows that both sums
in \eqref{lc2} converge absolutely. Thus so does the sum in \eqref{lc1}.

For any $\bga\in\cAx$, by properties of the Poisson distribution, the two set of
strings $\set{\Xi\sssk:k\le N_\gl \text{ and } \Xi\sssk\not\succ\bga}$ 
and $\set{\Xi\sssk:k\le N_\gl \text{ and } \Xi\sssk\succ\bga}$ 
are independent. Consequently,
$N_\gl-N_{\gl,\bga}$ is independent of $N_{\gl,\bga}$ and of
$\Tglbgax$
and thus of $\Xbga=\gf(\Tglbgax)$.
Hence, 
recalling \eqref{byx}--\eqref{tpax},
\begin{align}\label{skm}
  \Cov\bigpar{\Xbga,N_\gl}
&=
  \Cov\bigpar{\Xbga,N_{\gl,\bga}}
=\Cov\bigpar{\gf(\Tglbgax),\abse{\Tglbgax}}
\notag\\&
=
  \Cov\bigpar{\gf(\Tqq{\glbga}),\abse{\Tqq{\glbga}}}
\notag\\&
=\E{\psi(\Tqq{\glbga})}
-\E{\gf(\Tqq{\glbga})}\glbga.
\end{align}
We sum \eqref{skm} first over $\bga\in\cA^k$, and then over $k\ge0$, using
the $L^2$ convergence in \eqref{mans}, and obtain \eqref{lc1}--\eqref{lc2}.
\end{proof}

\begin{lemma}
  \label{LA}
  Suppose that $\gf$ is a toll function such that
  as \gltoo,
  \begin{align}
    \Var \gf(\Tgl) = O\bigpar{\gl^{1-\eps}}\label{la1}
  \end{align}
  for some $\eps>0$. Then
  \begin{align}
    \Var \GF(\Tgl) = O\bigpar{\gl}, \qquad \gl\in(0,\infty).\label{la2}
  \end{align}
\end{lemma}
\begin{proof}\CCreset
  By subtracting a suitable multiple of $\gfo(T):=\indic{T=\Ti}$ from $\gf$,
  we may assume that $\gf(\Ti)=0$. (Because $\GFo(\Tgl)=N_\gl$ satisfies
  \eqref{la2}.)

By \eqref{la1}, there exist $\gl_0$ and $\CCname{\CCkanga}$ such that, for
$\gl\ge\gl_0$,
  \begin{align}\label{kanga}
    \Var \gf(\Tgl)\le \CCx\gl^{1-\eps}
  \end{align}
\refL{L0}\ref{L0m} applies to $\gf^2$, with $\gl_2=\gl_0$ and $m=2$, and
shows that, for $\gl\le\gl_0$,
  \begin{align}\label{ru}
    \Var \gf(\Tgl)
\le 
    \E [\gf(\Tgl)^2]
\le
\CCname{\CCru}\gl^{2}.
  \end{align}
It follows that, perhaps after increasing $\CCkanga$ and $\CCru$,
\eqref{kanga} and \eqref{ru} both hold for all $\gl\in(0,\infty)$.
  
Let $J_k:= (\pmax^{k+1},\pmax^k]$, $k\ge0$,
and define, recalling \eqref{tpaxx}, 
  \begin{align}\label{z1}
    Z_{\gl,k}:=\sum_{P(\bga)\in J_k}\gf(\Tglbga)
    =\sum_{P(\bga)\in J_k}\gf(\Tglbgax).
  \end{align}
Thus, we have by \eqref{GFbga} the decomposition
  \begin{align}\label{z2}
    \GF(\Tgl)=\sumk Z_{\gl,k}.
  \end{align}
 If $\bga$ is a prefix of $\bgb$ and $\bga\neq\bgb$, then
$P(\bgb)\le\pmax P(\bga)$, and thus $P(\bga)$ and $P(\bgb)$ cannot both
belong to the same $J_k$. Hence, the modified fringe tries $\Tglbgax$
are
independent for all $\bga$ with $P(\bga)\in J_k$.
Consequently,
using \eqref{tpaxx},
  \begin{align}\label{z3}
    \Var Z_{\gl,k}=\sum_{P(\bga)\in J_k}\Var \gf(\Tglbgax)
    =\sum_{P(\bga)\in J_k}\Var \gf(\Tqq{\glbga}).
  \end{align}
By definition, $P(\bga)$ is the probability that the random string $\Xi$ has
$\bga$ as a prefix; hence $\sum_{P(\bga)\in J_k}P(\bga)$ is the expected
number of prefixes $\bga$ in $\Xi$ with $P(\bga)\in J_k$. Since none of
these strings $\bga$ is a prefix of another, as just seen, $\Xi$ can contain at
most one such prefix. Hence,
\begin{align}\label{sump1}
  \sum_{P(\bga)\in J_k}P(\bga)\le 1.
\end{align}
Combining \eqref{z3} with \eqref{kanga} and \eqref{sump1}, we obtain
\begin{align}
  \Var Z_{\gl,k}
  &\le
    \sum_{P(\bga)\in J_k}\CCkanga(\glbga)^{1-\eps}
  =\CCkanga\gl^{1-\eps}     \sum_{P(\bga)\in J_k}P(\bga)^{1-\eps}
\notag\\  &
    \le \CCkanga\gl^{1-\eps} \pmax^{-\eps(k+1)} \sum_{P(\bga)\in J_k}P(\bga)
          \le  \CC\gl\bigpar{\gl \pmax^{k}}^{-\eps}.\label{pu+}
\end{align}
Similarly, using instead \eqref{ru},
\begin{align}
  \Var Z_{\gl,k}
  &\le
    \sum_{P(\bga)\in J_k}\CCru(\glbga)^{2}
  =\CCru\gl^{2}     \sum_{P(\bga)\in J_k}P(\bga)^2
\notag\\  &
    \le \CCru\gl^{2} \pmax^{k} \sum_{P(\bga)\in J_k}P(\bga)
          \le  \CCru\gl\bigpar{\gl \pmax^{k}}.\label{pu-}
\end{align}
By \eqref{z2} and Minkowski's inequality, \eqref{pu+}--\eqref{pu-} imply
\begin{align}
  \bigpar{\Var\GF(\Tgl)}\qq
&  \le \sumk  \bigpar{\Var Z_{\gl,k}}\qq
  \notag\\&
  \le \CC\gl\qq \sumk
   \min\bigset{\bigpar{\gl \pmax^{k}}^{-\eps/2},\bigpar{\gl \pmax^{k}}\qq}
  \notag\\&  \le \CC\gl\qq,
\end{align}
since the last sum is dominated by the sum of two convergent geometric series,
uniformly in $\gl$. This shows \eqref{la2}.
\end{proof}

\begin{remark}\label{Reps0}
 We cannot take $\eps=0$ in \eqref{la1} and assume only $\Var\gf(\Tgl)=O(\gl)$.
  A counter example is $\gf(T)=\abse{T}\indic{\abse{T}\ge2}$;
  then $\GF(T)$ is the  external path length,
  and $\Var \GF(\Tgl)$ is of order $\gl\ln^2\gl$, see
  \cite[Lemma 12]{JR88}.
\end{remark}

\begin{proof}[Proof of \refT{TEVC}]
We prove the theorem in two steps, first in the special case
$\chi=\gf(\Ti)=0$, and then in general.

\setcounter{stepp}{0}
\stepp{$\chi=\gf(\Ti)=0$.}\label{step0}
First, \eqref{bze}--\eqref{bzv} show that $\E|\gf(\Tgl)|$ and $\E|\gf(\Tgl)|^2$
are finite for large $\gl$, and thus for all $\gl>0$ by \refL{L0}.
Hence,  \refLs{LM1} and \ref{LV} apply for any $\gl$.
By \eqref{fE0}--\eqref{fC0}, we can write \eqref{lm1}, \eqref{lv4} and
\eqref{lc1} as
\begin{align}
    \E \GF(\Tgl) = \sum_{\bga\in\cAx}\fE(\glbga),\label{KE}
    \\
   \Var \GF(\Tgl) = \sum_{\bga\in\cAx}\fV(\glbga),\label{KV}
    \\
   \Cov\bigpar{ \GF(\Tgl),N_\gl} = \sum_{\bga\in\cAx}\fC(\glbga),\label{KC}
\end{align}
with absolute convergence; in particular the \lhs{s} are finite and so are
(taking the term $\bga=\emptystring$ in the sums)
$\fE(\gl),\fV(\gl),\fC(\gl)$ for every $\gl>0$.

These equations are all instances of \eqref{Fgl} in \refT{Tsum} in
\refApp{A242},
and we verify the conditions of that theorem.
First, we may write also $\fV$ and $\fC$  using only expectations of
functionals of $\Tgl$:
\begin{align}\label{fVE}
  \fV(\gl)&
=2\E\bigsqpar{\gf(\Tgl)\GF(\Tgl)}
  -2\E\gf(\Tgl)\E\GF(\Tgl)-\E\gf(\Tgl)^2+\bigpar{\E\gf(\Tgl)}^2 ,
\\
  \fC(\gl)&
=\E\bigsqpar{\gf(\Tgl)N_\gl}-\gl\E\gf(\Tgl)
. \label{fCE}
\end{align}
All expectations in \eqref{fE0} and \eqref{fVE}--\eqref{fCE} are finite by
\eqref{bze}--\eqref{bzv} and \refL{L0}, \eqref{KE}--\eqref{KV}, and
the \CSineq.
Hence, it follows from the general formula \eqref{l000} that
$\fE,\fV,\fC$ are all continuous, and in fact, entire analytic.
Furthermore, it follows from \refL{L0}\ref{L0m} that the expectations in
\eqref{fE0}, \eqref{fVE} and \eqref{fCE} all are $O(\gl^2)$ for $\gl\le1$,
and thus \eqref{b1} holds for $\fE,\fV,\fC$.
Equivalently, the entire functions $\fX$ satisfy 
\begin{align}\label{ris}
\fX(0)=\fX'(0)=0.  
\end{align}

Next, $\fE$ satisfies \eqref{b2} by the assumption \eqref{bze}.
Furthermore,
\refL{LA} applies by \eqref{bzv} and yields
$\Var\GF(\Tgl)=O(\gl)$. 
Also, $\Var N_\gl=\gl$.
Hence \eqref{fV0}--\eqref{fC0},
\eqref{bzv} and the \CSineq{} 
yield
\begin{align}\label{baq}
  \fV(\gl),\fC(\gl)=O\bigpar{\gl^{1-\eps/2}}
\end{align}
for large $\gl$, and thus for $\gl\ge1$ (since $\fV$ and $\fC$ are
continuous and thus bounded on finite intervals). In other words,
\eqref{b2} holds for $\fV$ and $\fC$, with $\eps$
replaced by $\eps/2$.

Hence, \refT{Tsum}\ref{Tsum0}--\ref{Tsumcont}
apply to $\fE,\fV,\fC$.
Furthermore, if $\dA>0$ and 
$\fX'(\gl)=O(\gl^{-\eps_1})$ as \gltoo,
then also \refT{Tsum}\ref{Tsumabsconv} applies;
note that $\fX'(\gl)=O(\gl)$ as $\gl\to0$ by \eqref{ris}.
In particular, this is always the case for $\fE$ (when $\dA>0$), since 
it follows from
\eqref{tex} (which will be proved below) and
\eqref{baq} that 
$\fE'(\gl)=O\bigpar{\gl^{-\eps/2}}$ for $\gl\ge1$.

Consequrently, the results in \ref{TEVC0} and \ref{TEVCd} follow from
\refT{Tsum}.

Finally, if $\gf\ge0$ and $\gf(T')>0$ for some trie $T'$, then
$\fE(\gl)=\E\gf(\Tgl)>0$ for every $\gl>0$ by \eqref{l000}; hence
\ref{TEVC+} follows
by \refT{Tsum}\ref{TsumTheta}.

\stepp{The general case.} \label{stepgeneral}
Consider the toll function 
\begin{align}
  \label{japan}
\gff:=\gf-\chi\gfo
\end{align}
and the corresponding
additive functional $\GFF$; note that $\gff(\Ti)=0$.
Then
\begin{align}\label{kina}
  \GFF(\Tgl)=\GF(\Tgl)-\chi\GFo(\Tgl)
=\GF(\Tgl)-\chi N_\gl.
\end{align}
Define,
for $\sX=\sE,\sV,\sC$ as usual,
 $\fXp$
by \eqref{fE0}--\eqref{fC0} for the functionals $\gff$
and $\GFF$, and define 
$\psiXx$ by \eqref{psiX0}--\eqref{tevcfou} with $\fXp$.
Then, by the case just proved, 
\eqref{uE+}--\eqref{uC+} hold for moments of $\GFF(\Tgl)$, if we omit the
terms $\chi$ or $\chi^2$ and replace $\psiX$ by $\psiXx$.
Hence, using \eqref{kina} and
recalling $\E N_\gl=\Var N_\gl=\gl$,
  \begin{align}
\frac{\E\GF(\Tgl)}{\gl}
&=\frac{\E\GFF(\Tgl)+\chi \E N_\gl}{\gl}
= \frac1 H \psiEx(\ln\gl) +\chi+ o(1),\label{uE++}
    \\
\frac{\Var\GF(\Tgl)}{\gl}
&
=\frac{\Var\GFF(\Tgl)+ 2\chi\Cov(\GFF(\Tgl),N_\gl)+\chi^2\Var N_\gl}{\gl}
\notag\\&
= \frac1 H \psiVx(\ln\gl) + 2\chi \frac1 H \psiCx(\ln\gl) 
+\chi^2+ o(1),\label{uV++}
    \\
\frac{\Cov\bigpar{\GF(\Tgl),N_\gl}}{\gl}
&=
\frac{\Cov\bigpar{\GFF(\Tgl),N_\gl}+\chi\Var N_\gl}{\gl}
\notag\\&
= 
\frac1 H \psiCx(\ln\gl)+\chi + o(1).\label{uC++}
  \end{align}
This proves
\eqref{uE+}--\eqref{uC+} 
(and thus \eqref{uE0}--\eqref{uC0} when $\dA=0$)
if we define
\begin{align}\label{psi+}
  \psiE&:=\psiEx,&
\psiV&:=\psiVx+2\chi\psiCx,&
\psiC&:=\psiCx,
\end{align}
which agrees with \eqref{psiX0}--\eqref{tevcsum} if we have
\begin{align}\label{f+}
  \fE&=\fEp,&
\fV&=\fVp+2\chi\fCp,&
\fC&=\fCp.
\end{align}
It remains to verify that \eqref{f+} agrees with the definitions
\eqref{fEx}--\eqref{fCx}.
In fact, \eqref{f+} yields,
by \eqref{fE0}--\eqref{fC0} and \eqref{japan}--\eqref{kina},
\begin{align}\label{fEp}
\fE(\gl)&=  \fEp(\gl)
=\E \gff(\Tgl)
=\E \gf(\Tgl)-\chi\E\gfo(\Tgl),
\\
\fV(\gl)&=
\fVp(\gl)+2\chi\fCp(\gl)
\notag\\&
=2\Cov\bigpar{\gff(\Tgl),\GFF(\Tgl)}
-\Var\bigpar{\gff(\Tgl)}
+2\chi\Cov\bigpar{\gff(\Tgl), N_\gl}
\notag\\&
=2\Cov\bigpar{\gff(\Tgl),\GF(\Tgl)}
-\Var\bigpar{\gf(\Tgl)-\chi\gfo(\Tgl)}
\notag\\&
=2\Cov\bigpar{\gf(\Tgl),\GF(\Tgl)}
-2\chi\Cov\bigpar{\gfo(\Tgl),\GF(\Tgl)}
\notag\\&\qquad
-\Var{\gf(\Tgl)}
+2\chi\Cov\bigpar{\gfo(\Tgl),\gf(\Tgl)}
-\chi^2\Var\gfo(\Tgl),
              \\
\fC(\gl)&=
\fCp(\gl)=\Cov\bigpar{\gf(\Tgl),N_\gl}-\chi\Cov\bigpar{\gfo(\Tgl),N_\gl}.
\label{fCp}
\end{align}
Furthermore, recalling 
that
$\gfo(\Tgl)=\indic{N_\gl=1}$,
we have
\begin{align}
  \E\gfo(\Tgl)&=\P(N_\gl=1)=\gl e^{-\gl},\label{ell1}
\\
\Var \gfo(\Tgl)&=\gl e^{-\gl}\bigpar{1-\gl e^{-\gl}},\label{ell2}
\\
\Cov\bigpar{\gfo(\Tgl),N_\gl}&=\P(N_\gl=1)-\P(N_\gl=1)\gl
=(1-\gl)\gl e^{-\gl}.\label{ell3}
\end{align}
Also, since $\GF(\Ti)=\gf(\Ti)$ and thus
$\gfo(T)\bigpar{\gf(T)-\GF(T)}=0$ for every $T$, we have
\begin{align}\label{ell4}
\Cov\bigpar{\gfo(\Tgl),\gf(\Tgl)-\GF(\Tgl)}&
=-\E \gfo(\Tgl) \E\bigpar{\gf(\Tgl)-\GF(\Tgl)}
\notag\\&
=-\gl e^{-\gl}\bigpar{\E \gf(\Tgl)-\E\GF(\Tgl)}.
\end{align}
Combining \eqref{fEp}--\eqref{fCp} and \eqref{ell1}--\eqref{ell4}, we obtain
\eqref{fEx}--\eqref{fCx} after simple calculations.

The assertion on absolute convergence of the Fourier series follows as in the
case $\chi=0$ from \refT{Tsum}\ref{Tsumabsconv}, using \eqref{f+} to verify 
$\fX(\gl)=O(\gl^2)$ for $\gl<1$.
 
This completes the proof of \ref{TEVC0}--\ref{TEVCd}.
For \ref{TEVC+}, we note that this was proved in Step~\ref{step0} if
$\chi=0$.  If $\chi>0$, 
we note that $\gff\ge0$ and thus 
$\fE(\gl)=\E\gff(\Tgl)\ge0$; hence $\psiE(t)\ge0$ by \eqref{tevcsum}
and $\inf_t\bigpar{H\qw\psiE(t)+\chi}\ge\chi>0$.
Alternatively, we may simply note that $\GFF\ge0$ and thus \eqref{kina}
implies
$\GF(\Tgl)\ge\chi N_\gl$ and  $\E\GF(\Tgl)\ge\chi\gl$.
\end{proof}

Note that \eqref{ris}--\eqref{baq} and the comments between them justify the
claims on existence of the Mellin transforms $\MfX(s)$ in \refR{Rexists}.

\begin{proof}[Proof of \refL{LEVCx}]  
By the proof of \refT{TEVC}, we have $\fE=\fEp$ and $\fC=\fCp$, so it
suffices to consider the case $\chi=0$. 
In this case,
it follows from \eqref{l000} that 
the derivative of the entire analytic function
$\fE(\gl)$ is, using \eqref{fE0} and \eqref{fC0} (or \eqref{fCE}),
\begin{align}
  \fE'(\gl)
  &=  -\fE(\gl)+ e^{-\gl}\sumni\frac{n\gl^{n-1}}{n!}a_n
  =-\fE(\gl)+\gl\qw\E\bigpar{N_\gl\gf(\Tgl)}
\notag  \\
        &=\frac{-\gl\E\gf(\Tgl)+\E\bigpar{\gf(\Tgl)N_\gl}}{\gl}
  =  \frac{\fC(\gl)}{\gl},\label{kut}
\end{align}
which proves \eqref{tex}.
Furthermore, for $-2<\Re s<-1+\eps/2$ (\cf{} \refR{Rexists}), 
\eqref{kut} and an integration by parts gives, using
\eqref{bze} and $\fE(\gl)=O(\gl^2)$ shown above,
\begin{align}
  \MfC(s)=\intoo \fE'(\gl)\gl^s\dd\gl
  =-s\intoo \fE(\gl)\gl^{s-1}\dd\gl
=-s\MfE(s),
\end{align}
showing \eqref{tex*}.
Finally, \eqref{texpsi} follows from \eqref{tex*} when $\dA=0$, and
otherwise from either 
\eqref{tevcsum} and \eqref{tex}, or
\eqref{tevcfou} and \eqref{tex*};
we omit the details.
\end{proof}

\section{Proof of \refT{TT} -- \refL{LfE}}\label{SpfTT}
Finally, we prove
\refT{TT} and the remaining other results in \refS{Smain}.

\begin{proof}[Proof of \refT{TT}]
  Since $\gf$ and $\gf_\pm$ are bounded, 
\eqref{bze}--\eqref{bzv} hold for $\gf$ and $\gf_\pm$ (with
  $\eps=1$), and thus \refT{TEVC} applies to these functionals. In
  particular, 
\eqref{uV+} (or \refL{LA}) implies
  that \eqref{VFO} holds for $\GF$ and $\GF_\pm$.
  Furthermore, \eqref{Vf} and \eqref{Efr} (for any $r>0$)
  hold trivially for $\gf$ and
 $\gf_\pm$, again because the functionals are bounded.
Moreover, define
\begin{align}
a(\gl)&:=\E\GF(\Tgl),
\\
\label{btto}
  b(\gl)&:=
\Bigpar{\chi^2+ \frac1 H \psiV(\ln\gl)}\gl,  
\end{align}
Then 
\eqref{EF} holds trivially, and
\eqref{VF} holds by 
\eqref{uV+}.
Consequently, \refT{T1} applies, with any $r>2$, 
 and \eqref{t1x} yields
 \begin{align}\label{ttk}
\frac{\GF(\Tgl)-\E\GF(\Tgl)}{\sqrt{\gl}}
  \approxd N\bigpar{0,b(\gl)/\gl},
\end{align}
with  all \absmoment{s},
which by \eqref{btto} is
\eqref{tt+gl} with \eqref{tt+gs}.
In the special case $\dA=0$, this yields
\eqref{tt0gl} with
\eqref{tt0gs}.

Moreover,
define also
\begin{align}\label{ctto}
  c(\gl):=
    \Bigpar{\chi+\frac1 H \psiC(\ln\gl)}\gl.
\end{align}
Then \refT{TEVC} shows also that \eqref{CovFN} holds, and thus
\refT{T1n} applies with any $r>2$.
Hence \eqref{t1na} holds, 
with all \absmoments,
which, recalling \eqref{btto} and
\eqref{ctto}, yields
\eqref{tt+n}, 
with \eqref{tt+hgs}.
In the special case $\dA=0$, this yields
\eqref{tt0n} with \eqref{tt0hgs}.
This proves \ref{TT0} and \ref{TT+}.

\ref{TTE} follows by \eqref{EFn}.

For \ref{TTconv}, suppose that $\liminf_\ntoo\Var\GF(\Tn)/n>0$.
By \eqref{VFn0}, this means that \eqref{bc} holds.
Hence, \eqref{t1nb} holds, which is \eqref{ttn},
with all \absmoments.
Furthermore,
\eqref{VFn0}, \eqref{btto} and \eqref{ctto}
show that, 
\begin{align}\label{ros}
  0<
  \liminf_\ntoo\Var\GF(\Tn)/n&
=
  \liminf_{\ntoo} 
  \bigpar{\chi^2+H\qw\psiV(\ln n)-\bigpar{\chi+H\qw \psiC(\ln n)}^2}
\notag\\&  =
  \inf_{x\in\bbR} \bigpar{\chi^2+H\qw\psiV(x)-\xpar{\chi+H\qw \psiC(x)}^2},
  \end{align}
where the final equality holds because this function of $x$ is continuous
and periodic (and constant  if $\dA=0$).
Hence,
also
\begin{align}
\inf_{\gl>0} b(\gl)/\gl=  \inf_{x\in\bbR}\bigpar{\chi^2+ H\qw\psiV(x)} >0.
  \end{align}
 Consequently,
\eqref{bB} holds and \eqref{t1b} follows, which is \eqref{ttgl},
with all \absmoments.

For \ref{TTmean},
\eqref{ttmean} follows from \eqref{uE+}, and then \eqref{ttmeann} follows by
\eqref{EFn}. 
\end{proof}

\begin{proof}[Proof of \refT{TLLN}]
  \pfitemref{TLLNa}
  It follows from
\refT{TT}, more precisely
\eqref{tt+gl}--\eqref{tt+n}
  together with \ref{TTE}, 
that
\begin{align}\label{tt+glp}
{\frac{\GF(\Tgl)-\E \GF(\Tgl)}{\gl}}
  &\pto0,
  \\\label{tt+np}
  {\frac{\GF(\Tn)-\E \GF(\Tgn)}{n}}
  &\pto0.
\end{align}
The result \eqref{tllna1}--\eqref{tllna2}
now follows from \eqref{uE+} in \refT{TEVC}.

\pfitemref{TLLNb}
In this case, \refT{TEVC}\ref{TEVC+} applies and shows that
$\inf_t\bigpar{H\qw\psiE(t)+\chi}>0$, 
and thus that
$\E \GF(\Tgl)>2c \gl$ for some $c>0$ and all large $\gl$.
Hence, \eqref{tt+np} implies \eqref{stalis}.
\end{proof}

\begin{proof}[Proof of \refL{LC}]
Let $m\ge1$ be such that
$\Var\GF(\Tq m)>0$, \ie, $\GF(\Tq m)$ is not deterministic. 
Let $b:=\max\set{n_0,m}$.

We show first that $\Var\GF(\Tb)>0$.
This is clear if $b=m$, so suppose
 that $\ellb>m$.
Let $\ga$ and $\gb$ be two distinct letters in $\cA$.
Condition on the event $\cE_{m,\ellb}$
that the strings $\Xi\sss1,\dots,\Xi\sss m$ begin with
$\ga$, and $\Xi\sss{m+1},\dots,\Xi\sss \ellb$ begin with $\gb$.
Then the root $\emptystring$ of $\cT_\ellb$ has two children $\ga$ and $\gb$, 
with $m$ and $\ellb-m$ strings passed to them, respectively.
By assumption, $\gf(\Tq \ellb)=a_b$, and thus \eqref{toll} yields
\begin{align}\label{klubba}
  \GF(\Tq \ellb)=a_b+\GF(\Tb^\ga)+\GF(\Tb^\gb),
\end{align}
where (still conditioned on $\cE_{m,\ellb}$)
$\GF(\Tb^\ga)$ and $\GF(\Tb^\gb)$ are independent.
Furthermore $\GF(\Tb^\ga)\eqd\GF(\Tq m)$, which is not deterministic;
hence \eqref{klubba} shows that $\GF(\Tb)$ conditioned on $\cE_{m,\ellb}$ is not
deterministic. Thus $\GF(\Tb)$ (unconditioned)
is not deterministic,
and $\Var \GF(\Tb)>0$ in this case too.

 Consider the bucket trie $\Tn'$ with bucket size $b$ grown from the $n$
 strings $\Xi\sss1,\dots,\Xi\sss{n}$.
 Then $\Tn'$ is a subtree of $\Tn$. Let $M_k$ be the number of buckets
 in $\Tn'$ that contain $k$ strings, $k=1,\dots,b$.
 Recall that $\Tn$ may be obtained from the bucket trie $\Tn'$ by growing a
 small trie from every bucket; denote these small tries by
$T_{ki}$, where $k=1,\dots,b$ and $i=1,\dots,M_k$, so $\abse{T_{ki}}=k$.
Recall also that 
conditioned on $\Tn'$, all these small tries are independent,
and that $T_{ki}$ is a copy of $\Tq{k}$.

Let $I(\Tn')$ denote the set of internal nodes of the bucket trie $\Tn'$.
The, \eqref{GF} implies the decomposition
\begin{align}\label{lolo}
  \GF(\Tn)=\sum_{v\in I(\Tn')}\gf(\Tn^v)+\sum_{k=1}^b\sum_{i=1}^{M_k}\GF(T_{ki}).  
\end{align}
By the construction of the bucket trie,
$\abse{\Tn^v}>b$ for every $v\in I(\Tn')$, and thus, by assumption,
$\gf(\Tn^v)=a_{\abse{\Tn^v}}$. Consequently,  the first sum in \eqref{lolo}
depends on the bucket trie $\Tn'$ but not on the small tries $T_{ki}$.
Consequently, conditioning on the bucket trie,
\begin{align}
  \Var \bigpar{ \GF(\Tn)\mid \Tn'}
  =\sum_{k=1}^b\sum_{i=1}^{M_k}\Var[\GF(T_{ki})]
  =\sum_{k=1}^bM_k\Var[\GF(\Tq{k})].  
\end{align}
Hence,
\begin{align}
  \Var \xpar{ \GF(\Tn)}
&  \ge
\E  \Var \bigpar{ \GF(\Tn)\mid \Tn'}
  =\sum_{k=1}^b \E M_k\cdot\Var[\GF(\Tq{k})]
\notag  \\&
  \ge  \E M_b\cdot\Var[\GF(\Tq{b})].  
\end{align}
We have shown that $\Var[\GF(\Tb)]>0$, and it remains only to show that
$\E M_b=\Omega(n)$, \ie, $\liminf_\ntoo\E M_b/n>0$.

It is easily seen that $M_b$ \as{}
equals the number of nodes in $\Tn$ that have
exactly $b$ strings passed to them and have more than one child.
Hence, $M_b=\GFb(\Tn)$ where the additive functional $\GFb$ has toll
function $\gfb(T)$, defined as the indicator that $\abse{T}=b$ and that the
  root of $T$ has more than one child.
  If we add a new string to a trie $T$, then $\GFb(T)$ may decrease
  by at most 1,
  since $\gf(T^v)$ can be affected only for $v$ in the path from the root to the
  new leaf or pair of leaves,
  and in this path there is at most one node with $\gf(T^v)\neq0$.
  It follows that if $\gfo(T):=\indic{T=\Ti}$ as in \refE{Eleaves}, so
  $\GFo(T)=\abse{T}$, then $\GFb+\GFo$ is an increasing functional.
  Hence,
  \refT{TLLN} applies to $\gfb$,
with $\gf_+:=\gfb+\gfo$ and $\gf_-:=\gfo$.
Thus, \eqref{stalis} holds, which implies
$\E M_b = \E\GFb(\Tn) \ge \frac12cn$ for large $n$; this completes the proof. 
\end{proof}

\begin{proof}[Proof of \refL{LfV}]
  Let again $\Xbga:=\gf(\Tglbga)$, see \eqref{Xbga}.  
  Since $|\gf(T)|\le C$, we have, using \eqref{P2},
  \begin{align}\label{pip}
    \bigabs{\Cov\bigpar{\gf(\Tgl),\Xbga}}
    \le \E \abs{C \Xbga} + C \E \abs { \Xbga}
    \le C_\gl P(\bga)^2.
  \end{align}
  Hence, \eqref{Qrr} implies that
the sum in \eqref{lfv1} converges absolutely for every $\gl$.
  Furthermore, again by \eqref{P2}, $\GF(\Tgl)=\sum_\bga\Xbga$ with
  convergence in $L^1$, and thus, since $\gf(\Tgl)$ is bounded,
  \begin{align}
    \Cov\bigpar{\gf(\Tgl),\GF(\Tgl)}
    =\sum_\bga\Cov\bigpar{\gf(\Tgl),\Xbga},
  \end{align}
  which shows the equality of the expressions in \eqref{fV0}
and \eqref{lfv1}.

By \refL{L0}\ref{L0m} (with $m=2$), $\E|\gf(\Tgl)|\le C\gl^2$ for
$\gl\le1$.
Hence, if $\gl\le1$, then for every $\ga$, recalling \eqref{Xbga}, 
$\E|\gf(\Tglbga)|\le C\gl^2P(\bga)^2$ for every $\bga\in\cAx$, and thus,
arguing as in \eqref{pip},
  \begin{align}\label{pipa}
    \sum_{\bga\in\cAx} \bigabs{\Cov\bigpar{\gf(\Tgl),\Xbga}}
    \le \sum_{\bga\in\cAx}  C \E \abs { \Xbga}
    \le \sum_{\bga\in\cAx}  C \gl^2 P(\bga)^2
    =C\gl^2.
  \end{align}

  For $\gl\ge1$, we use instead the decomposition and notation
  in the proof of \refL{LA}.
  Let $\eps_\bga:=\sign \bigpar{\Cov\bigpar{\gf(\Tgl),\Xbga}}\in\set{\pm1}$.
  Then, for each $k\ge0$,
  \begin{align}
    \sum_{P(\bga)\in J_k} \bigabs{\Cov\bigpar{\gf(\Tgl),\Xbga}}
    &=
 \sum_{P(\bga)\in J_k} \eps_\bga \Cov\bigpar{\gf(\Tgl),\Xbga}
  \notag\\&\label{pipb}    =
\Cov\Bigpar{ \gf(\Tgl),\sum_{P(\bga)\in J_k} \eps_\bga \Xbga}.
  \end{align}
Furthermore, the variables $\Xbga=\gf(\Tglbgax)$ are independent for
$\bga\in J_k$,
and thus, see \eqref{z1}, \eqref{z3} and \eqref{pu+}--\eqref{pu-},
\begin{align}
  \Var\Bigpar{\sum_{P(\bga)\in J_k} \eps_\bga \Xbga}
&  =
\sum_{P(\bga)\in J_k} \Var{ \Xbga}
  = \Var Z_{\gl,k}
  \notag\\&\label{pipc}
  \le C \gl \min\bigpar{(\gl\pmax^k)^{-\eps},\gl\pmax^k}.
\end{align}
Hence, by \eqref{pipb} and the \CSineq, recalling that $\gf$ is bounded,
\begin{align}
\sum_{P(\bga)\in J_k} \bigabs{\Cov\bigpar{\gf(\Tgl),\Xbga}}
    \le C \gl\qq \min\bigpar{(\gl\pmax^k)^{-\eps/2},(\gl\pmax^k)^{1/2}}.
\end{align}
We may now sum over $k\ge0$ and obtain,
since $\pmax<1$,
\begin{align}\label{pipp}
\sumaa \bigabs{\Cov\bigpar{\gf(\Tgl),\Xbga}}
    \le C \gl\qq.
\end{align}
The two estimates \eqref{pipa} for $\gl\le1$ and \eqref{pipp} for $\gl\ge1$
imply the same estimates (with a different $C$) for $\sumaxx$, and it follows
that, for $\gs:=\Re s\in(-2,-\frac12)$,
\begin{align}
 \intoo\sumax\bigabs{\Cov\bigpar{\gf(\Tgl),\gf(\Tglbga)} \gl^{s-1}}\dd \gl
  \le\intoo C\min\set{\gl^{\gs+1},\gl^{\gs-1/2}}\dd\gl
    <\infty.
\end{align}
Hence,
Fubini's theorem shows that we may interchange the sum and integral in
\eqref{lfv2}. Thus
\eqref{lfv2} follows by \eqref{lfv1}.
\end{proof}

\begin{proof}  [Proof of \refL{LfE}]
By replacing $\gf$ by $\gff$ as in the proof of \refT{TEVC}, we may again
assume $\chi=0$; recall \eqref{japan} and \eqref{f+}.
Furthermore, by considering the positive and negative parts of $\gf$
separately, we may also assume $\gf\ge0$.
Then, by \eqref{fE0} and \eqref{l000}, we have,
with  $a_n:=\E\gf(\Tn)$,
\begin{align}
  \fE(\gl)=e^{-\gl}\sum_{n=2}^\infty \frac{a_n}{n!} \gl^n .
\end{align}
Hence, at least for $-2<\Re s<-1+\eps$, first for real $s$ and then generally,
\begin{align}
  \MfE(s)=\intoo\fE(\gl)\gl^{s-1}\dd\gl
=\sum_{n=2}^\infty \frac{a_n}{n!} \intoo\gl^{n+s-1}e^{-\gl}\dd\gl
=\sum_{n=2}^\infty \frac{a_n}{n!} \gG(n+s),
\end{align}
yielding \eqref{lfe1}. Taking $s=-1$ yields \eqref{lfe2}, 
recalling $\MfC(-1)=\MfE(-1)$ from \refL{LEVCx}.
\end{proof}

\begin{ack}
I thank Pawel Hitczenko for help with references on
Rosenthal's inequality.
\end{ack}

\appendix

\section{Asymptotics of certain sums}\label{A242}

The following theorem is essentially \cite[Theorem 5.1]{SJ242},
with some extensions as discussed in the proof.
Recall the definition of
the entropy $H$ in \eqref{entropy},
the greatest common divisor $\dA:=\gcd\set{-\ln p_\ga:\ga\in\cA}$ in
\refSS{SSgcd},
and the Mellin transform $\Mf$ in \eqref{mellin}. 

\begin{theorem}\label{Tsum}
  Suppose that $f$ is a real-valued function on  $(0,\infty)$, and that
  \begin{equation}
	\label{Fgl}
    F(\gl)=\sumaa	f(\gl P(\aaa)),
  \end{equation}
  for $\gl>0$,
with $P(\aaa)$ given by \eqref{paaa}.
Assume further that $f$ is \aex\ continuous 
and satisfies the estimates 
\begin{align}
f(x)&=O(x^2),&& 0< x<1, \label{b1}\\ 
f(x)&=O(x^{1-\eps}), && 1<x<\infty, \label{b2}
\end{align}
for some $\eps>0$.

  \begin{romenumerate}
\item\label{Tsum0}
  If\/ $\dA=0$, then, as $\gl\to\infty$,
  \begin{equation}\label{jeiw}
	\frac{F(\gl)}{\gl} \to 
\frac1 H \Mf(-1)
=
\frac1 H \intoo f(x)x\qww\dd x.	
  \end{equation}

\item\label{Tsumd}
More generally, for any $\dA$, 
as $\gl\to\infty$,
  \begin{equation}\label{jeiiw}
	\frac{F(\gl)}{\gl} = \frac1 H \psi(\ln\gl) + o(1),
  \end{equation}
where
$\psi$ is a bounded function defined as follows: 
\begin{enumerate}
\item 
If $\dA=0$ then $\psi$ is constant: for all $t$,
  \begin{align}\label{psi00}
    \psi(t):=\Mf(-1).
  \end{align}

\item 
If $d=\dA>0$, then
$\psi$ is a bounded $d$-periodic function having the Fourier series
\begin{equation}\label{tsumfou}
  \psi(t)\sim\summoooo\hpsi(m)e^{2\pi\ii mt/d}
\end{equation}
with
\begin{equation}\label{tsumfouf}
  \hpsi(m)=\Mf\Bigpar{-1-\frac{2\pi m}{d}\ii}
=\intoo f(x)x^{-2-2\pi\ii m/d}\dd x.
\end{equation}
Furthermore, 
\begin{equation}
  \label{tsumsum}
\psi(t)=d\sumkoooo e^{kd-t}f\bigpar{e^{t-kd}}.
\end{equation}
\end{enumerate}

\item\label{Tsumcont}
  If $f$ is continuous, then $\psi$ is too.

\item\label{TsumTheta}
  If $f$ is continuous and $f>0$ on $(0,\infty)$,
  then $\inf_t\psi(t)>0$. Hence, $\psi(t)=\Theta(1)$ and, as \gltoo,
  $F(\gl)=\Theta(\gl)$. 

\item\label{Tsumabsconv}
Suppose that $\dA>0$.
If $f$ is continuous and  the Fourier series
\eqref{tsumfou} converges absolutely, then
its sum equals $\psi(t)$ everywhere, so
we may replace $\sim$ in \eqref{tsumfou} by $=$.
In particular, this holds if
$f$ is continuously differentiable on $(0,\infty)$ and
$f'(x)=O(x^\eps)$ as $x\to0$ and 
$f'(x)=O(x^{-\eps})$ as $x\to\infty$
for some $\eps>0$.

\end{romenumerate}

\end{theorem}

\begin{proof}
  \pfitemx{\ref{Tsum0},\ref{Tsumd},\ref{Tsumcont}}
This is, as said above, essentially
\cite[Theorem 5.1]{SJ242}. There are three technical differences:
\begin{alphenumerate}
  
\item
  \cite{SJ242} considers for simplicity only $\cA=\setoi$. However, the
  same proof holds for arbitrary $\cA$.

\item
  \cite[Theorem 5.1]{SJ242}
  assumes that $f\ge0$. This is technically convenient in the proof
  (\eg, all sums and integrals are defined), but the result extends
  immediately to real-valued $f$ by considering its positive and negative parts.
  
\item
  \cite[Theorem 5.1]{SJ242} as stated there
  assumes \eqref{b2} with $\eps=1$, but as said in
  \cite[Remark 5.2]{SJ242}, the proof holds for any $\eps>0$.
  (We may similarly relax \eqref{b1} to
$f(x)= O( x^{1+\eps})$ for $0<x<1$,
  but we have no use for this.)
\end{alphenumerate}

With these extensions,
\cite[Theorem 5.1]{SJ242}
yields \ref{Tsum0}--\ref{Tsumcont}.
(Note that, with $g(t)=e^ff(e^{-t})$ as in \cite{SJ242},
$\widehat g(u)=\Mf(-1+u\ii)$. Also,
\ref{Tsumcont}
is trivial if $\dA=0$.)

\pfitemref{TsumTheta}
This too is trivial if $\dA=0$ by \eqref{psi00} and \eqref{jeiw},
since the integral in \eqref{jeiw} now is positive.

Thus, suppose $\dA>0$.
Then $\psi(t)>0$ for every real $t$ by \eqref{tsumsum}. Since $\psi$ is
periodic by \ref{Tsumd} and continuous by \ref{Tsumcont},
it follows that $\inf_t\psi(t)>0$.
Hence, $F(\gl)=\Theta(\gl)$ as \gltoo{} by \eqref{jeiiw}.

\pfitemref{Tsumabsconv} 
$\psi$ is continuous by
\ref{Tsumcont} and has by assumption a  Fourier series that converges
everywhere, which 
implies that the Fourier series converges to $\psi$,
see \eg{}
\cite[III.(3.4) and applications after it]{Zygmund}.

If
$f$ is continuously differentiable on $(0,\infty)$ and
$f'(x)=O(x^\eps)$ as $x\to0$ and 
$f'(x)=O(x^{-\eps})$ as $x\to\infty$,
then \eqref{tsumsum} can be differentiated termwise and the resulting sum
converges uniformly on compact sets. 
Hence $\psi$ has a continuous derivative, and is, in
particular, Lipschitz on $[0,d]$, and thus its Fourier series converges
absolutely by a theorem by Bernstein \cite[Theorem VI.(3.1)]{Zygmund}.
\end{proof}

\begin{remark}
  It follows from \eqref{b1}--\eqref{b2} that the Mellin transform $\Mf(s)$
  exists at least when $-2<\Re s<-1+\eps$, and thus in particular
  when $\Re s=-1$ as in \eqref{jeiw} and \eqref{tsumfouf}.
\end{remark}


\section{Approximation in distribution and moments}
\label{Aapprox}
We prove here the following lemma, which includes \refL{Lapprox0} together
with a converse. 
It extends
the standard
result that if $X_n\dto Y$, then 
convergence of absolute moments of order $s$
is equivalent to uniform integrability of $|X_n|^s$, and implies
convergence of moments, 
see \eg{} \cite[Theorem 5.5.9]{Gut}.
Recall the definitions in \refSS{SSapprox}.

 \begin{lemma}
   \label{Lapprox}
 Let $(X_n)\xoo$  and $(Y_n)\xoo$  be random vectors in $\bbR^d$.
Let further
  $s>0$ be a real number,
and suppose that
the sequence  $\xpar{|Y_n|^s}$ is uniformly integrable.
Then the following are equivalent:
\begin{romenumerate}
  
\item \label{Lappri}
$X_n\approxd Y_n$ with absolute moments of order $s$.
\item \label{Lapprii}
 $X_n\approxd Y_n$
and
the sequence  $(|X_n|^s)$ is uniformly integrable.
\end{romenumerate}

Furthermore, if \ref{Lappri} or \ref{Lapprii} holds (and thus both hold),
and $s$ is an integer, then also
\begin{romenumerateq}
\item \label{Lappriii}
$X_n\approxd Y_n$ with  moments of order $s$.
\end{romenumerateq}
Moreover, if \ref{Lappri} or \ref{Lapprii} holds,
then they also hold with $s$ replaced by any  positive $s'<s$,
and \ref{Lappriii} holds with $s$ replaced by any smaller positive integer.
 \end{lemma}

 \begin{proof}
\ref{Lapprii}$\implies$\ref{Lappri},\ref{Lappriii}.
The assumptions that $|Y_n|^s$ and $|X_n|^s$ are \ui{} imply that the
sequences $(X_n)$  and $(Y_n)$ are tight.
Hence, for any subsequence $(n_j)$, there exists a subsubsequence
along which $X_n\dto X$ and $Y_n\dto Y$ for some random
variables $X$ and $Y$.
The assumption $X_n\approxd Y_n$ implies $X\eqd Y$.
Furthermore, the uniform integrability and convergence in distribution imply
that, still along the subsubsequence, 
$\E |X_n|^s\to\E|X|^s$ and $\E |Y_n|^s\to\E|Y|^s$,
see  \cite[Theorem 5.5.9]{Gut}.
Since
$\E |X|^s=\E|Y|^s$, it follows that \eqref{absmomapprox} holds along the
subsubsequence. By the subsequence principle, \eqref{absmomapprox} holds for
the full sequence, and thus \ref{Lappri} holds.

If $s$ is an integer, then  \eqref{momapproxd} follows by the same argument.

\ref{Lappri}$\implies$\ref{Lapprii}.
This is similar. 
The uniform integrability of $|Y_n|^s$ implies that the
sequence  $Y_n$ is tight.
Hence, for any subsequence $(n_j)$, there exists a subsubsequence
along which $Y_n\dto Y$ for some random variable $Y$.
The assumption $X_n\approxd Y_n$ implies that also $X_n\dto Y$ along the
subsubsequence.
Furthermore, 
still along the subsubsequence, 
the assumption on uniform integrability implies
$\E |Y_n|^s\to\E|Y|^s$, and since we assume
\ref{Lappri} and thus  \eqref{absmomapprox},
we have $\E |X_n|^s\to\E|Y|^s$.
Hence, by \cite[Theorem 5.5.9]{Gut} again,
$|X_n|^s$ is \ui{}
along the subsubsequence. 
The subsequence principle holds also for uniform integrability, 
and thus \ref{Lapprii} holds. (To see this, assume that 
the full sequence $|X_n|^s$ is not \ui; then there exists an $\eps>0$ and a
subsequence $n_j$ such that
$\E\bigsqpar{|X_{n_j}|^s\ett{|X_{n_j}|^s>j}}>\eps$, but 
then no subsubsequence is \ui, a contradiction.)

Finally, it is well known that if \ref{Lapprii} holds for some $s$, it holds
for all smaller $s'$ as well.
 \end{proof}

 \begin{remark}\label{RNM}
The standard case with $X_n\dto Y$ is the case when all $Y_n$ are equal; in
that case \ui{} of $|Y_n|^s$ is redundant. In general, however, it is
needed. Here are some counterexamples without uniform integrability.
\begin{romenumerate}
\item  \label{RNM1}
Let $X_n=Y_n=n$ (non-random); then \ref{Lappri} holds
trivially but not \ref{Lapprii}. 
\item \label{RNM2}
Let $\P(X_n=-n)=1-\P(X_n=n)=1/n$ and $Y_n=n$. Then $|X_n|=|Y_n|$ so
\eqref{absmomapprox} is trivial and \ref{Lappri} holds for any $s$, but
\ref{Lapprii} fails for all $s$ and \ref{Lappriii} for odd integers $s$.
  
\item \label{RNM3}
Let $\P(X_n=n^2)=1-\P(X_n=0)=1/n$ and 
$\P(Y_n=n^3)=1-\P(Y_n=0)=1/n^3$.
Then $X_n\dto0$ and $Y_n\dto0$ and thus $X_n\approxd Y_n$.
Furthermore, \ref{Lappri} holds for $s=2$ but not for $s=1$.
\end{romenumerate}
Such examples indicate that moment
approximation as in \eqref{momapprox}--\eqref{absmomapprox} is not of much
interest unless $|Y_n|^s$ are \ui.
 \end{remark}

 \begin{remark}
   If $s$ is an even integer, then $|X_n|^s=X_n^s$ when $d=1$, and in
   general $|X_n|^s$ is a linear combination of moments $\E X_n^\multim$
   with $|\multim|=s$. Hence, \eqref{momapproxd} implies
   \eqref{absmomapprox},
which together with \refL{Lapprox} shows that,
assuming that $|Y_n|^s$ are \ui,
$X_n\approxd Y_n$ with absolute moments of order $s$ is equivalent to
$X_n\approxd Y_n$ with moments of order $s$.
Again, the condition of uniform integrability is needed here, as shown by 
\refR{RNM}\ref{RNM2}.
 \end{remark}

\section{A lower bound on the approximation error}
\label{Alower}

The rate of convergence of the asymptotic results \eqref{uE0} and
\eqref{uE+} has been studied in detail by \citet{FlajoletRV}.
They focussed on the aperiodic case $\dA=0$ and
gave upper bounds for the error, with a very slow rate of convergence.
It is implicit in their arguments that their bounds are essentially the best
possible, and we state one version of that  as \refT{TC} below. 
For simplicity we consider there, and in most part of this appendix,
only the special case of the
size, as in \refSS{SSsize}, although the results are more general and 
the method applies also to other examples in \refS{Sfringe},
\cf{} \cite[Definition 4 and Lemma 6]{FlajoletRV}.

However, we first review the much simpler periodic case. 
This case is well-known, see \eg{} \cite[Section 7.2]{JS-Analytic},
and included her for completeness and comparison.

\subsection{The symmetric case}\label{SSss}

Consider the case $p_\ga=1/r$ for every $\ga\in\cA$, where $r:=|\cA|$.
Note that $\dA=H=\ln r$.
Let $\gf$ and $\fE$ be as in \refT{TEVC} and assume for simplicity
$\chi=0$.
Then, by \eqref{KE}, 
\begin{align}\label{ss1}
 \frac{\E \GF(\Tgl)}{\gl}=\summ \gl\qw r^m \fE\bigpar{r^{-m}\gl}
=\summoooo \frac{ \fE\bigpar{r^{-m}\gl}}{r^{-m}\gl}
+R(\gl),
  \end{align}
where  
\begin{align}\label{rgl}
R(\gl)=
-\sumki \frac{\fE\bigpar{r^{k}\gl}}{r^k\gl}.
\end{align}
The assumption \eqref{bze} implies that the sum
\eqref{rgl} converges, and
\begin{align}\label{cs3}
  R(\gl) = O\bigpar{\gl^{-\eps}}.
\end{align}
Furthermore,
the last sum in \eqref{ss1} equals $H\qw\psiE(\ln\gl)$, see \eqref{tevcsum}.
Hence, the elementary calculation \eqref{ss1} yields \eqref{uE+} with the
error term $R(\gl)=O(\gl^{-\eps})$. In fact, in several examples in
\refS{Sfringe}, $\fE(\gl)$ decreases exponentially as $\gltoo$, and then the
same holds for $R(\gl)$.

In the special case of the size, treated in \refSS{SSsize}, \eqref{xsfE} yields
\begin{align}\label{rglss}
R(\gl)=
-\sumki \frac{1-(1+r^k\gl)e^{-r^{k}\gl}}{r^k\gl}
=-\frac{1}{r-1}\gl\qw + O\bigpar{e^{-r\gl}}
\end{align}
as $\gltoo$. Thus the error rate \eqref{cs3} (with $\eps=1$) is exact in
this case.

\subsection{The asymmetric periodic case}\label{SSas}

Suppose $d=\dA>0$; then there exist positive integers $\kk_\ga$ such that
$-\ln p_\ga=\kk_\ga d$ and thus $p_\ga=e^{-\kk_\ga d}$, $\ga\in\cA$. Hence, 
\begin{align}\label{fox}
  \rho(s)=\suma p_\ga^s = \suma e^{-\kk_\ga ds}=Q\bigpar{e^{-ds}},
\end{align}
where $Q$ is the polynomial $Q(z)=\suma  z^{\kk_\ga}$
of degree $\kk:=\max\kk_\ga$.
We exclude the
symmetric case in \refSS{SSss}; then $\kk\ge2$.
Denote the roots of $Q(w)=1$ by $w_1,\dots,w_\kk$, with $w_1=e^{-d}>0$.
Consider again $\GF(T):=\absi{T}$, the size.
A standard inversion of the Mellin transform yields, see \eg{}
\cite[Lemma 6]{FlajoletRV} (although there stated for the aperiodic case),
or \cite[Section 7.2]{JS-Analytic},
\begin{align}\label{vix}
  \frac{\E\GF(\Tgl)}{\gl} = \sum_{z_j} a_j \gl^{z_j-1} + O\bigpar{\gl^{-M}}
\end{align}
for any $M<\infty$, where $a_j$ are some complex numbers and $z_j$ ranges 
over the roots of $\rho(z_j)=1$. By \eqref{fox}, these roots are (changing
the notation)
\begin{align}
  z_{k,\ell}:=\frac{-\ln w_k}{d}+ \frac{2\pi\ii}{d}\ell,
\qquad k\in\set{1,\dots,\kk},\; \ell\in\bbZ.
\end{align}
The roots $z_{1,\ell}$ have $\Re z_{1,\ell}=1$, and the corresponding terms in
\eqref{vix} yield the periodic function $H\qw\psiE(\ln\gl)$ in \eqref{uE+}.
Similarly, the terms for $z_{k,\ell}$ for a fixed $k>1$ sum to 
$\gl^{-\eps_k}g_k(\ln\gl)$, where $g_k$ is a $d$-periodic function
and
\begin{align}
  \eps_k:=1-\Re z_{k,0}=1+\ln|w_k|/d,
\end{align}
which necessarily satisfies $\eps_k>0$.
Hence, if $\eps:=\min_{2\le k\le\kk}\eps_k>0$, then 
\eqref{vix} yields
\eqref{uE+} with an error term
\begin{align}
R(\gl)=O\bigpar{\gl^{-\eps}}.
\end{align}
Furthermore, this is the exact order of the error term (for typical $\gl$).
Here, $\eps>0$ depends on the probabilities $(p_\ga)_\ga$ 
and may be arbitrarily small, even in the binary case.
This too is certainly known, but we do not know a reference and give an
example for completeness.

\begin{example}\label{Eeps}\CCreset
Consider the binary case, with $\vp=(p,1-p)$. 
In the periodic case $\dA>0$, denote $\eps$ above by
$\eps(p)$; in the aperiodic case let $\eps(p)=0$. 
Let $\gd>0$.

Let $p_0$ be any number with $\ln (1-p_0)/\ln p_0$ irrational.
Then there exist roots $s$ of $\rho(s)=p_0^{s}+(1-p_0)^s=1$ with $1-\gd<\Re
s<1$,
see \eg{} \cite{FlajoletRV}, or the proof of \refT{TC} below.
Let $s_0$ be one such root. It follows from the implicit function theorem
that for every $p$ sufficiently close to $p_0$, there exists $s$ with
$p^s+(1-p)^s=1$ and $s$ so close to $s_0$ that $\Re s \in(1-\gd,1)$. Hence,
$\eps(p)\le 1-\Re s<\gd$. We may here choose $p$ such that $\ln(1-p)/\ln p$
is rational.
Consequently, the set of $p$ such that $\dA>0$ and 
$\eps(p)<\gd$ is dense in $(0,1)$
for every $\gd>0$.

For a concrete example, let $m\ge1$ and
let $p=p_m$ be the unique positive root of $p+p^{(m+1)/m}=1$.
It is not difficult to show that as $m\to\infty$, $p_m\to\frac12$ and
$\eps(p_m)\to0$. We omit the details.
\end{example}

\subsection{The aperiodic case}
In the aperiodic case $\dA=0$ 
of \refT{TEVC}, \eqref{uE0} (or \eqref{uE+}) says that
\begin{align}\label{c1}
  \frac{\E\GF(\Tgl)}{\gl}=\chi+H\qw\MfE(-1)+R(\gl),
\end{align}
where $R(\gl)=o(1)$ as $\gl\to\infty$. As said above,
it follows from
\citet{FlajoletRV} that $R(\gl)$ typically tends to 0 very slowly.
More precisely, we have the following, for simplicity considering only the size.

\begin{theorem}[implicit in \cite{FlajoletRV}]\label{TC}
  Assume $\dA=0$ and let $\GF(T)=\absi{T}$, the size of $T$.
Then \eqref{c1} holds (with $\chi=0$ and $\MfE(-1)=1$)
and there exist $C<\infty$ and arbitrarily large $\gl$ such that
\begin{align}\label{c2}
  |R(\gl)| > \exp\bigpar{-C(\ln\gl)^{(|\cA|-1)/(|\cA|+1)}}.
\end{align}
\end{theorem}
Note that the lower bound in \eqref{c2} is larger than $\gl^{-\eps}$ for any
$\eps>0$ (and large $\gl$). Cf.\ the results in the periodic cases above.

\citet{FlajoletRV} prove  corresponding upper bounds, for most but not all
probability vectors $(p_\ga)$; see in particular
\cite[Theorem 4 and Corollaries 1 and 2]{FlajoletRV}.
As said above, 
the lower bound in \refT{TC} is only implicit in
\cite{FlajoletRV}; a detailed proof seems to require some work, and since
we do not know any published proof, we give one below for completeness.

\begin{remark}
  It follows further from \cite{FlajoletRV}
that for special vectors $(p_\ga)_{\ga\in\cA}$,
even larger lower bounds hold; in fact, by considering the binary case 
with $\ln p_0/\ln p_1$ an irrational number that can be approximated
extremely well by rationals (a suitable Liouville number), we can make $R(\gl)$
converge arbitrarily slowly to $0$.
\end{remark}

We first prove a lemma.

\begin{lemma}\label{LCC}
  Suppose that 
  \begin{align}\label{c4}
h(x)=\sumjoooo a_j e^{\zeta_jx},
  \end{align}
where the complex numbers $a_j$ and $\zeta_j=-s_j+\ii t_j$ satisfy the
following, for some $c>0$ and all $j\in\bbZ$,
\begin{align}
s_j&>0,\label{cs}\\
t_{j+1}-t_j&\ge c,
\label{ct}\\
|a_j|&= e^{-\Theta(|t_j|)+O(1)}
.\label{cq}
\end{align}
Then there exists $\gd>0$ and $C$ such that for every $j$ with $s_j<\gd$ and
$t_j>0$, there
exists $x$ with $ t_j/(2s_j) < x < 2 t_j/s_j$ and
$|h(x)|>e^{-C t_j}$.

The same result holds for a one-sided sum 
$h(x)=\sumj a_j e^{\zeta_jx}$.
\end{lemma}

\begin{proof}\CCreset
  In this proof $C_i$ and $c_i$ denote positive constants that depend only
  on $c$ in \eqref{ct} and the implicit constants in \eqref{cq}.
Note first that \eqref{cq} and \eqref{ct} imply 
\begin{align} \label{csum}
\sum_k|a_k|\le \CCname{\Csum}.
\end{align}
In particular, the sum \eqref{c4} converges for every $x\ge0$.

We assume first that $s_j\le \gd$ for every $j$, and treat then the general
case.

\setcounter{stepp}{0}
\pfcase{$\max_j s_j\le\gd\le1$.} Consider a $j$ with $t_j\ge1$.
Let $Z\sim N(0,1)$ be a standard normal random variable and define
\begin{align}
\mu_j&:=t_j/s_j\ge  t_j/\gd \ge1/\gd,\label{muj}\\
  h_j(x)&:=e^{-\zeta_j x}h(x)=\sum_k a_k e^{(\zeta_k-\zeta_j)x},\label{hj}\\
Z_j&:=\mu_j+2\mu_j\qq Z \sim N\bigpar{\mu_j,4\mu_j}.\label{Zj}
\end{align}
Note that 
the assumption $s_k\in[0,\gd]$ and \eqref{csum}
imply 
that the sums \eqref{c4} and \eqref{hj}
converge for every real $x$, with
\begin{align}\label{anka}
  |h_j(x)|\le \sum_k|a_k| e^{(s_j-s_k)x}\le \Csum e^{\gd|x|}.
\end{align}
In particular, $h_j(x)$ is defined for all real $x$.
Furthermore, by \eqref{hj},
\begin{align}\label{c19}
  \E h_j(Z_j) &
= \sum_k a_k \E e^{(\zeta_k-\zeta_j)Z_j}
= \sum_k a_k \E e^{(\zeta_k-\zeta_j)\mu_j+2\mu_j(\zeta_k-\zeta_j)^2}
\notag\\&
=:\sum_k A_k,
\end{align}
where we thus denote the terms in the sum by $A_k$.
Note that $A_j=a_j$, so, by \eqref{cq},
\begin{align}\label{hades}
 | A_j|=|a_j| \ge e^{-\CCname{\CCh} t_j}.
\end{align}
For $k\ne j$, we note that \eqref{cq} implies $|a_k|\le\CCname{\CCa}$, 
and thus, by \eqref{c19} and \eqref{muj},
\begin{align}\label{req}
  |A_k|&
=|a_k|e^{(s_j-s_k)\mu_j+2\mu_j(s_j-s_k)^2-2\mu_j(t_j-t_k)^2}
\le \CCx e^{\gd\mu_j+2\mu_j\gd^2-2\mu_j|t_j-t_k|^2}
\notag\\&
\le e^{\CC+3\gd\mu_j-2\mu_j|t_j-t_k|^2}
\le e^{\CCname{\CCqu}\gd\mu_j-2\mu_j|t_j-t_k|^2}
.
\end{align}
It follows from \eqref{ct} that whenever $j\neq k$, we have $|t_j-t_k|\ge
c|j-k|\ge c$. Hence,
if $\CCname{\CCcc}:=c^2/2$ and $\gd$ is small enough,
then \eqref{req} implies
\begin{align}\label{ete}
  |A_k|
\le e^{\mu_j(\CCqu\gd-2c^2|j-k|^2)}
\le e^{\mu_j(\CCqu\gd-c^2-c^2|j-k|^2)}
\le e^{-\CCcc \mu_j- c^2\mu_j|j-k|^2 }.
\end{align}
Recalling \eqref{muj}, we thus find that for $\gd$ small enough, 
\begin{align}\label{nam}
  \sum_{k\neq j}|A_k|\le \CCname{\CCuq} e^{-\CCcc\mu_j}
\le \CCuq e^{-(\CCcc/\gd)t_j}
\le 
e^{-2\CCh t_j}.
\end{align}
Combining \eqref{nam} with \eqref{hades} and \eqref{c19}, 
we find that for $t_j$ sufficiently  large, 
  \begin{align}\label{quam}
\bigabs{\E h_j(Z_j)}
\ge |A_j|-\sum_{k\neq j}|A_k| \ge \tfrac12 |a_j| \ge e^{-(\CCh+1)t_j}.
  \end{align}

Next, let $I_j:=(\frac{1}{2}\mu_j,\frac{3}{2}\mu_j)$. 
By the \CSineq, \eqref{anka} and a standard tail estimate for the 
normal distribution, 
if $\gd$ is small enough,
\begin{align}\label{grip}
\qquad&\hskip-2em  
\bigabs{\E \bigsqpar{h_j(Z_j)\indic{Z_j\notin I_j}}}^2
\le \E \bigsqpar{h_j(Z_j)^2}\P\bigpar{Z_j\notin I_j}
\notag\\&
\le \Csum^2 \bigpar{\E e^{2\gd Z_j} + \E e^{-2\gd Z_j}}
\P\bigpar{|Z|\ge \tfrac14\mu_j\qq}
\notag\\&
\le \CC e^{2\gd\mu_j+8\gd^2\mu_j}e^{-\frac{1}{32}\mu_j}
\le \CCx e^{-\frac{1}{64}\mu_j}
\le \CCx e^{-\frac{1}{64 \gd} t_j}.
\end{align}
Hence, if $\gd$ is small enough, using $t_j\ge1$ and \eqref{quam},
\begin{align}
  \bigabs{\E \bigsqpar{h_j(Z_j)\indic{Z_j\notin I_j}}}
\le e^{-(\CCh+2)t_j} \le e^{-1} |\E h_j(Z_j)|,
\end{align}
and thus, using \eqref{quam} again,
\begin{align}
  \bigabs{\E \bigsqpar{h_j(Z_j)\indic{Z_j\in I_j}}}
\ge \tfrac12 |\E h_j(Z_j)| \ge e^{-\CC t_j}
\end{align}
Consequently, there exists $x_j\in I_j$ such that $|h_j(x_j)|\ge e^{-\CCx t_j}$, 
 and hence,
\begin{align}\label{rolf}
  |h(x_j)|=e^{-s_j x_j}|h_j(x_j)|\ge e^{-(3/2)t_j}|h_j(x_j)|\ge e^{-\CCname{\CCfinal} t_j}.
\end{align}

We have shown that there exist $\gd_1>0$ and $T\ge1$ such that if
$s_k\le \gd_1$ for every $k$, then the result \eqref{rolf}
holds for every $j$ with 
$t_j\ge T$.
We ignore temporarily the finite number of $j$ with $0<t_j<T$.

\pfcase{The general case.}
Let $\gd_1$ and $T\ge1$ be as just said in Case 1.
Let $\gL:=\set{\zeta_j}$  and consider the subsets
$\gLa:=\set{\zeta_j:s_j\le \gd_1}$ and
$\gLb:=\set{\zeta_j:s_j> \gd_1}$.
Define the corresponding sums
\begin{align}
  \ha(x):=\sum_{\zeta_j\in\gLa} a_je^{\zeta_jx},
&&&
  \hb(x):=\sum_{\zeta_j\in\gLb} a_je^{\zeta_jx}.
\end{align}
We may assume that $\set{\zeta_j\in\gLa:t_j>0}$ is infinite, 
since otherwise we may choose $\gd>0$ such that $s_j>\gd$ for all $j$ with
$t_j>0$, and the result is trivial.

Then Case 1 applies to $\ha$ (after relabelling $\zeta_j$). 
Hence, if $\gd\le\gd_1$, for every
$j$ with $s_j<\gd$ and $t_j\ge T$, there exists $x_j$ such that
$t_j/(2s_j)<x_j<2t_j/s_j$ and $|\ha(x_j)|>e^{-\CCfinal t_j}$.
Furthermore, recalling \eqref{csum}, if $\gd$ is small enough,
\begin{align}
  |\hb(x_j)| 
&\le \sum_{\zeta_j\in\gLb}|a_j|e^{-s_jx_j}
\le \Csum e^{-\gd_1 x_j}
\notag\\&
\le \Csum e^{-(\gd_1/2 s_j)t_j}
\le \Csum e^{-(\gd_1/2 \gd)t_j}
\le \tfrac12 e^{-\CCfinal t_j},
\end{align}
and consequently,
\begin{align}
  |h(x_j)| \ge |\ha(x_j)|-|\hb(x_j)|
\ge \tfrac12 e^{-\CCfinal t_j}
\ge e^{-(\CCfinal+1)t_j}.
\end{align}
This shows the result for any $j$ with $s_j<\gd$ and $t_j\ge T$.
By decreasing $\gd$, we may further assume that $s_j\ge\gd$ for each of the
finitely many $j$ with $0<t_j<T$, and the result then holds for every $j$
with $s_j<\gd$ and $t_j>0$.
\end{proof}

\begin{proof}[Proof of \refT{TC}]\CCreset
Let $r:=|\cA|$, the number of letters in the alphabet, and assume without
loss of generality that $\cA=\set{1,\dots,r}$.

By \cite[Lemma 6]{FlajoletRV}, $R(\gl)$ in \eqref{c1} can be written
$R(\gl)=R_1(\gl)+R_2(\gl)$ where 
for some small $\hh>0$,
$R_2(\gl)=O(\gl^{-\hh})$ 
and 
\begin{align}
  R_1(\gl)=\sum_k \frac{(1-z_k)\gG(-z_k)}{\rho'(z_k)}\gl^{z_k-1},
\end{align}
summing over the set \set{z_k} of roots of $\rho(z)=1$ satisfying
$1-\hh\le \Re z_k<1$.
Thus $h(x):=R_1(e^x)$ is a function of the type in \eqref{c4}, 
with $s_j=1-\Re z_j$ and $t_j=\Im z_j$, and \eqref{cs}--\eqref{cq} hold by
results in \cite{FlajoletRV}.

We assume that $\dA=0$; thus, for any fixed $k\in\cA$ at least one ratio 
$\ln p_\ell/\ln p_k$ is irrational. By \cite[Theorem 200]{HardyW}, there exist
infinitely many positive integers $q$ such that for some integers $\kk_\ell$
\begin{align}\label{cm}
\Bigabs{ q\frac{\ln p_\ell}{\ln p_k}-\kk_\ell} < q^{-1/(r-1)},
\qquad \ell=1,\dots,r.
\end{align}
(Note that the case $\ell=k$ is trivial, so we really consider a vector of
$r-1$ elements.) In the terminology of \cite{FlajoletRV} the
approximation function $f_k(q)$ of the vector $(\ln p_\ell/\ln p_k)_\ell$
satisfies 
$
  f_k(q)\ge q^{1/(r-1)}
$ 
for infinitely many $q$.
Let $(q_j)_1^\infty$ be an increasing sequence of positive integers such
that \eqref{cm} holds for each $q=q_j$.
Then the proof of \cite[Theorem 2(ii)]{FlajoletRV} shows that for each
sufficiently large  $q_j$, there exists a
root $z_j=1-s_j+\ii t_j$ of $\rho(z_j)=1$ with 
\begin{align}
  0<s_j&
\le \CC q_j^{-2/(r-1)},\label{c6}
\\
\CC q_j \le t_j&\le \CC q_j.\label{c7}
\end{align}

We apply \refL{LCC} to the function $h(x)$ and find,
for sufficiently large $j$, 
$x_j$ such that,
using \eqref{c6} and \eqref{c7},
\begin{align}
  x_j\ge \frac{t_j}{2s_j}
\ge \CC \frac{t_j}{q_j^{-2/(r-1)}}
\ge \CC \frac{t_j}{t_j^{-2/(r-1)}}
=\CCx t_j^{(r+1)/(r-1)}
\end{align}
and
\begin{align}\label{c8}
|h(x_j)|
\ge e^{-\CC t_j}
\ge e^{-\CC x_j^{(r-1)/(r+1)}}
.\end{align}
Let $\gl_j:=e^{x_j}$.  
Since \eqref{c6} implies $s_j\to0$ and thus $x_j\to\infty$, we have
$\gl_j\to\infty$. Furthermore,
by \eqref{c8}, 
\begin{align}
 | R_1(\gl_j)|=
|h(x_j)|
\ge e^{-\CCx x_j^{(r-1)/(r+1)}}  
=e^{-\CCx (\ln \gl_j)^{(r-1)/(r+1)}}  
.\end{align}
\end{proof}

\newcommand\AAP{\emph{Adv. Appl. Probab.} }
\newcommand\JAP{\emph{J. Appl. Probab.} }
\newcommand\JAMS{\emph{J. \AMS} }
\newcommand\MAMS{\emph{Memoirs \AMS} }
\newcommand\PAMS{\emph{Proc. \AMS} }
\newcommand\TAMS{\emph{Trans. \AMS} }
\newcommand\AnnMS{\emph{Ann. Math. Statist.} }
\newcommand\AnnPr{\emph{Ann. Probab.} }
\newcommand\CPC{\emph{Combin. Probab. Comput.} }
\newcommand\JMAA{\emph{J. Math. Anal. Appl.} }
\newcommand\RSA{\emph{Random Struct. Alg.} }
\newcommand\ZW{\emph{Z. Wahrsch. Verw. Gebiete} }
\newcommand\DMTCS{\jour{Discrete Math. Theor. Comput. Sci.} }
\newcommand\DMTCSproc{{Discrete Math. Theor. Comput. Sci. Proc.} }

\newcommand\AMS{Amer. Math. Soc.}
\newcommand\Springer{Springer-Verlag}
\newcommand\Wiley{Wiley}

\newcommand\vol{\textbf}
\newcommand\jour{\emph}
\newcommand\book{\emph}
\newcommand\inbook{\emph}
\def\no#1#2,{\unskip#2, no. #1,} 
\newcommand\toappear{\unskip, to appear}

\newcommand\webcite[1]{
\texttt{\def~{{\tiny$\sim$}}#1}\hfill\hfill}
\newcommand\webcitesvante{\webcite{http://www.math.uu.se/~svante/papers/}}
\newcommand\arxiv[1]{\webcite{arXiv:#1.}}

\def\nobibitem#1\par{}

\end{document}